\documentclass{amsart}

\usepackage{latexsym,amscd,amssymb,url}
\usepackage{comment}

\usepackage{amsmath}
\usepackage{amsthm}
\usepackage{amssymb}
\usepackage{mathrsfs}

\usepackage[all,cmtip]{xy}  
\usepackage{graphicx}
\usepackage{xcolor}
\usepackage{enumitem} 
\definecolor{dblue}{rgb}{0,0,.6}
\usepackage[colorlinks=true, linkcolor=dblue, citecolor=dblue, filecolor = dblue, menucolor = dblue, urlcolor = dblue]{hyperref}

\numberwithin{equation}{section}

\newtheorem{theorem}{Theorem}[section]

\theoremstyle{plain}

\newtheorem{corollary}[theorem]{Corollary}

\newtheorem{definition}[theorem]{Definition}

\newtheorem{lemma}[theorem]{Lemma}

\newtheorem{proposition}[theorem]{Proposition}
\newtheorem{remark}[theorem]{Remark}

\newcommand{\del}{\partial}
\newcommand{\Z}{\mathbb Z}
\newcommand{\Q}{\mathbb Q}

\newcommand{\C}{\mathbb C}
\newcommand{\N}{\mathbb N}

\newcommand{\CP}{\mathbb P}

\newcommand{\RR}{\operatorname{R}}

\newcommand{\im}{\operatorname{im}}

\newcommand{\Pic}{\operatorname{Pic}}

\newcommand{\id}{\operatorname{id}}

\newcommand{\Spec}{\operatorname{Spec}}

\newcommand{\pr}{\operatorname{pr}}
\newcommand{\NS}{\operatorname{NS}}

\newcommand{\codim}{\operatorname{codim}}

\newcommand{\CH}{\operatorname{CH}}
\newcommand{\supp}{\operatorname{supp}}
\newcommand{\sing}{\operatorname{sing}} 
\newcommand{\red}{\operatorname{red}}

\newcommand{\cl}{\operatorname{cl}}

\newcommand{\F}{\mathbb F}

  \newcommand{\coker}{\operatorname{coker}}
  \newcommand{\Griff}{\operatorname{Griff}}  
\newcommand{\tors}{\operatorname{tors}}  
\newcommand{\Tors}{\operatorname{Tors}}

\newcommand{\alg}{\operatorname{alg}}

  \newcommand{\Grifftilde}{A_0 }  
     \newcommand{\Sh}{\operatorname{Shv}}
\newcommand{\et}{\text{\'et}}
\newcommand{\proet}{\text{pro\'et}}
\newcommand{\Ab}{\operatorname{Ab}}
\newcommand{\Tr}{\operatorname{Tr}}
\newcommand{\Mod}{\operatorname{Mod}}
\newcommand{\VV}{\mathcal{V}}

\newcommand{\colim}{\operatorname{colim}}
\newcommand{\cx}{\operatorname{an}}

\newcommand{\dashedlongrightarrow}{\xymatrix@1@=15pt{\ar@{-->}[r]&}}
\renewcommand{\longrightarrow}{\xymatrix@1@=15pt{\ar[r]&}}
\renewcommand{\mapsto}{\xymatrix@1@=15pt{\ar@{|->}[r]&}}
\renewcommand{\twoheadrightarrow}{\xymatrix@1@=15pt{\ar@{->>}[r]&}}
\newcommand{\hooklongrightarrow}{\xymatrix@1@=15pt{\ar@{^(->}[r]&}}
\newcommand{\congpf}{\xymatrix@1@=15pt{\ar[r]^-\sim&}}
\renewcommand{\cong}{\simeq}

\begin{document}

\title[Refined unramified cohomology of schemes]{Refined unramified cohomology of schemes}
 
\author{Stefan Schreieder} 
\email{schreieder@math.uni-hannover.de}
\address{Institute of Algebraic Geometry, Leibniz University Hannover, Welfengarten 1, 30167 Hannover, Germany.}

\date{March 2, 2023} 
%


\begin{abstract}  
We introduce the notion of  refined unramified cohomology of algebraic schemes and prove comparison theorems that identify some of these groups with cycle groups. 
This generalizes to cycles of arbitrary codimension previous results of Bloch--Ogus, Colliot-Thélène--Voisin, Kahn, Voisin, and Ma. 
We combine our approach with  the Bloch--Kato conjecture, proven by Voevodsky, to show that on a smooth complex projective variety,  any homologically trivial torsion cycle with trivial Abel--Jacobi invariant has coniveau $1$.
This establishes a torsion version of a conjecture of Jannsen  originally formulated $\otimes \Q$. 
 We further show that the group of homologically trivial torsion cycles modulo algebraic equivalence has a finite filtration (by coniveau) such that the graded quotients are determined by higher Abel--Jacobi invariants that we construct.
 This may be seen as a variant for torsion cycles modulo algebraic equivalence of a conjecture of Green.  
 We also prove  $\ell$-adic  analogues of these results  over 
 any field $k$ which contains all $\ell$-power roots of unity.  
\end{abstract}

\maketitle 

\tableofcontents
 
\section{Introduction}
 
Unramified cohomology  
of a smooth variety may be defined as the subgroup of the  cohomology of the generic point given by all classes that have trivial residues at all codimension one points, see \cite{BO} and \cite[4.1.1(a)]{CT}. 
Bloch--Ogus \cite{BO} showed that unramified cohomology in degree 3 is related to the Griffiths group of codimension 2 cycles.
Colliot-Thélène--Voisin \cite{CTV}  computed the failure of the integral Hodge conjecture for codimension 2 cycles on smooth complex projective varieties in terms of unramified cohomology in degree $3$; a similar statement holds for the integral Tate conjecture by Kahn \cite{Kahn}.
A relation between torsion codimension 3 cycles with unramified cohomology in degree $4$ is due to  
Voisin \cite{Voi-unramified} and Ma \cite{Ma}.

The results in \cite{CTV,Kahn,Voi-unramified,Ma} use two main ingredients: the Gersten conjecture proven by Bloch--Ogus \cite{BO}, which identifies unramified cohomology of smooth varieties with the global sections of a certain Zariski sheaf, and the Bloch--Kato conjecture, proven in \cite{MS,Voe:Bloch-Kato}. 

This paper arose from the observation that the aforementioned results from \cite{CTV,Kahn} have more elementary proofs, not relying on the Gersten conjecture, nor on the Bloch--Kato conjecture,  see Section \ref{sec:CTV} 
below for more details. 
This leads us to the notion of refined unramified cohomology,  which generalizes unramified cohomology.  
Our arguments work for cycles of arbitrary codimensions, over arbitrary  fields, and even on singular schemes, see Theorems \ref{thm:main-2-intro} and  \ref{thm:singular} below.
 
Our main results on algebraic torsion cycles that we explain next combine the machinery of refined unramified cohomology with the Bloch--Kato conjecture \cite{Voe:Bloch-Kato}.  

\subsection{Torsion cycles and Abel--Jacobi invariants} \label{subsec:main-result}  
A cycle $z\in \CH^i(X)$ on a complex variety $X$ has coniveau $j$, i.e.\ $z\in N^j\CH^i(X)$, if $z$ is homologically trivial on a closed subset of codimension $j$. 
That is, $z\in N^j$ if $z=\del \gamma$ is the boundary of a locally finite singular chain $\gamma$ whose support is contained in a closed algebraic subset of codimension $j$ in $X$.
This yields a finite descending filtration with $N^0=\CH^i(X)_{\hom}$ and  $N^{i-1}=\CH^i(X)_{\alg}$,  the subgroups of homologically and algebraically trivial cycles, respectively, cf.\ \cite{bloch-duke} and  \cite[p.\ 491, Remark 2]{totaro-JAMS}.

Jannsen showed that on smooth complex projective varieties, $N_j\CH^i(X):=N^{i-j}\CH^i(X)$ is a filtration by adequate equivalence relations, see \cite[Theorem 5.6]{jannsen-3} 
(stated $\otimes \Q$, but the same arguments work integrally).
In particular, $N^\ast$  interpolates naturally between algebraic and homological equivalence, respectively, and it is multiplicative in a strong sense: $N^j\CH^i(X)\cdot \CH^h(X)\subset N^{j+h}\CH^{i+h}(X)$.

\begin{theorem} \label{thm:coniveau1}
Let $X$ be a smooth projective variety over $\C$ and let $i\geq 2$.
A homologically trivial torsion cycle $z\in \CH^i(X)_{\tors}$ has coniveau 1 if and only if Griffiths' Abel--Jacobi invariant $\lambda(z)\in H^{2i-1}(X,\Q/\Z)$ admits a lift to  $N^1H^{2i-1}(X,\Q)$. 
\end{theorem}

The above theorem uses that the torsion subgroup of Griffiths' intermediate Jacobian $J^{2i-1}(X)$ identifies canonically to the image of $H^{2i-1}(X,\Q)$ in $ H^{2i-1}(X,\Q/\Z)$, cf.\ \cite[p.\ 116]{bloch-compositio}.
The subgroup $N^1H^{2i-1}(X,\Q)\subset H^{2i-1}(X,\Q)$ consists of those classes that vanish away from a divisor. 

The case $i=2$ may be deduced from \cite[\S 18]{MS}; the case  $i\geq 3$ is new.
Torsion cycles to which the above theorem applies are constructed e.g.\ in \cite{totaro-JAMS,SV,Sch-Griffiths}. 

\begin{corollary} \label{cor:n-torsion}
Let $X$ be a smooth projective variety over $\C$.
Then for any positive integer $n$, the $n$-torsion subgroup of $\CH^i(X)_{\tors}/N^1\CH^i(X)_{\tors}$ is finite. 
\end{corollary}

We illustrate the above corollary in the case of cycles of codimension $3$. 
In this case the coniveau filtration is of the form $N^2\subset N^1\subset N^0\subset \CH^3(X)$.
The above corollary shows that the $n$-torsion in $\CH^3(X)$ is finite modulo $N^1$.
In contrast, it is shown in \cite{schoen-product,Rosenschon-Srinivas,totaro-annals}  that the $n$-torsion in $N^2\CH^3(X)$ can be infinite (the torsion classes constructed there are all algebraically trivial, hence contained in $N^2$).  
Similarly, using the theory developed in this paper we show in \cite{Sch-Griffiths} that at least for $n$ even, the $n$-torsion in $\CH^3(X)_{\tors}/N^2$ is in general infinite. 
It follows that $\CH^3(X)_{\tors}/N^0$ and $N^0\CH^3(X)_{\tors}/N^1$ are the only graded pieces of $\CH^3(X)_{\tors}$ whose torsion is always finite.
In this sense, Corollary \ref{cor:n-torsion} is optimal for $i=3$.  
We also note that $ \CH^i(X)/N^1\CH^i(X)$  
is for $i=2$  in general not a finitely generated group, see \cite{clemens,voisin-duke} (it is conceivable that the same holds all $i\geq 2$). 

Another immediate consequence of Theorem \ref{thm:coniveau1} is as follows.

\begin{corollary} \label{cor:jannsen}
Let $X$ be a smooth projective variety over $\C$.
A homologically trivial torsion cycle $z\in \CH^i(X)_{\tors}$ with trivial Abel--Jacobi invariant has coniveau 1.
\end{corollary}

Corollary \ref{cor:jannsen} proves a torsion analogue of a conjecture of Jannsen (going back to a question of Esnault) who writes in \cite[p.\ 227, (5)]{jannsen-3}  that ``cycles in the kernel of the Abel-Jacobi map should be homologous to zero on a divisor, at least modulo torsion''. 
Jannsen shows that his conjecture follows from deep motivic conjectures:  the existence of a Bloch--Beilinson filtration $F^\ast$ (see e.g.  \cite[\S 2.1]{jannsen-3}) together with the standard conjecture B  imply $F^{j}\subset N^{j-1}$ at least rationally, see \cite[p.\ 226, (4)]{jannsen-3}.  
Jannsen's conjecture generalizes a conjecture of Nori \cite{nori}, which predicts that the  transcendental Abel--Jacobi map on codimension 2 cycles modulo algebraic equivalence is injective,  see \cite[p.\ 468]{totaro-JAMS}.

Nowadays essentially all deep conjectures in the theory of algebraic cycles on smooth complex projective varieties are formulated rationally. 
For instance,  Hodge  originally formulated his famous conjecture integrally, but when Atiyah and Hirzebruch showed that it fails for torsion cycles \cite{AH}, it became clear that one should phrase it rationally.
Nonetheless, investigating instances where the Hodge  conjecture may hold integrally remained an active field of research, see e.g.\ \cite{voisin-IHC,CTV,BW,perry}. 
Similarly, it is natural to investigate to which extent other cycle conjectures may hold integrally, or on torsion cycles,  see e.g.\ \cite[\S 8]{totaro-JAMS}.
In view of the many torsion counterexamples to the integral Hodge conjecture (see e.g.\ \cite{AH,totaro-JAMS,SV,Benoist-Ottem-Enriques}), it may be surprising that the integral Jannsen conjecture holds by Corollary  \ref{cor:jannsen}  unconditionally on torsion cycles.

The proof of Theorem \ref{thm:coniveau1} uses a homological interpretation of the transcendental Abel--Jacobi invariant on torsion cycles that does not require the smoothness assumption on $X$. 
The main point of this observation is that it allows for $0\leq j\leq i-2$ to define higher Abel--Jacobi invariants
$$
\overline \lambda_{j,tr}^i:N^j\CH^i(X)_{\tors}\longrightarrow \overline{J}^i_{j,tr}(X)_{\tors},
$$
by applying suitable Abel--Jacobi mappings on closed subschemes of $X$.
Here,
$$
\overline{J}^i_{j,tr}(X)_{\tors}:=\lim_{\substack{\longrightarrow \\ Z\subset X}} H^{BM}_{2d-2i+1}(Z_{\cx},\Q/\Z)/N^1H^{BM}_{2d-2i+1}(Z_{\cx},\Q),
$$
where $d:=\dim X$,  $Z\subset X$ runs through all closed subschemes with $j=\dim X-\dim Z$ and $N^\ast$ denotes Grothendieck's coniveau filtration on the respective Borel--Moore homology group. 

\begin{theorem} \label{thm:green}
Let $X$ be a separated scheme of finite type over $\C$.
Then for all $0\leq j \leq i-2$, we have
$$
 N^{j+1}\CH^i(X)_{\tors}=\ker \left(\overline \lambda_{j,tr}^i:N^j\CH^i(X)_{\tors}\to  \overline{J}^i_{j,tr}(X)_{\tors}\right) .
$$ 
\end{theorem}

Since $N^{i-1}\CH^i(X)=\CH^i(X)_{\alg}$ is divisible, the above theorem implies that the torsion subgroup of $A^i(X):=\CH^i(X)/\sim_{\alg}$ admits a finite filtration (by coniveau) such that the graded pieces are determined by higher Abel--Jacobi invariants.
This should be compared to Green's conjecture \cite{green,voisin-annals},  which predicts that  rational Chow groups of smooth complex projective varieties  carry a finite filtration (expected to be the Bloch--Beilinson filtration) such that the graded quotients are determined by higher Abel--Jacobi invariants.

Theorem \ref{thm:green} admits the following $\ell$-adic analogue, concerning the $\ell$-power torsion subgroup $\CH^i(X)[\ell^\infty]$ of $\CH^i(X):=\CH_{\dim X-i}(X)$.

\begin{theorem} \label{thm:green-ell-adic-arbitrary-field}
Let $X$ be a separated scheme of finite type over a field $k$ and let $\ell$ be a prime invertible in $k$.
Assume that $k$ contains all  $\ell$-power roots of unity.
Then for all $0\leq j\leq i-2$, we have
$$
N^{j+1}\CH^i(X)[\ell^\infty]=\ker \left(\overline \lambda_{j,tr}^i:N^j\CH^i(X)[\ell^\infty]\to \overline{J}^i_{j,tr}(X)[\ell^\infty] \right).
$$ 
\end{theorem}
 
The coniveau filtration on  $\CH^i(X)[\ell^\infty]$ as well as the (higher) Abel--Jacobi invariants are defined analogues to the case of complex schemes above, where we replace ordinary Borel--Moore homology by its $\ell$-adic pro-\'etale analogue \cite{BS}, see Proposition \ref{prop:proetale-coho-arbitrary-field} and Definition \ref{def:bar-J^i_j,tr} below.

The above theorem proves analogues of Theorems \ref{thm:coniveau1} and \ref{thm:green} over any field $k$ that contains all $\ell$-power roots of unity. 
This includes in particular an $\ell^{\infty}$-torsion version of Jannsen's conjecture over any field that contains all $\ell$-power roots of unity, such as (function fields over) algebraically closed fields.

\subsection{Refined unramified cohomology and algebraic cycles}
Let $X$ be a separated scheme of finite type over a field $k$.
We consider the 
increasing filtration
$$
 F_0X\subset F_1X\subset \dots \subset F_{\dim X}X= X,
\ \ \ \text{where}\ \ \ 
F_jX:= \{x\in X\mid\codim(x)\leq j\},
$$
and $\codim(x):=\dim X-\dim (\overline{\{x\}})$. 
Each $F_jX$ may be seen as a pro-object in the category of schemes.
For a given (co-)homology functor $H^i(-,A(n))$ that admits pullbacks along open immersions of schemes of the same dimension, the (co-)homology of $F_jX$ is defined as direct limit over all open subsets $U\subset X$ with $F_jX\subset U$. 
We then define the associated $j$-th refined unramified cohomology by
\begin{align} \label{def:H_jnr-intro}
H^i_{j,nr}(X,A(n)):=\im (H^i(F_{j+1}X,A(n))\to H^i(F_jX,A(n))) .
\end{align}
The Gysin sequence (see (\ref{eq:exact-couple})) shows that the case $j=0$ corresponds to classical unramified cohomology.

\subsubsection{Complex schemes} 
Let now $k=\C$.
For an abelian group $A$,  let  
\begin{align} \label{eq:Borel-Moore}
H^{i}(X, A(n)):=H^{BM}_{2d_X-i}(X_{\cx},A(d_X-n)),  
\end{align} 
where
$H^{BM}_\ast$ denotes Borel--Moore homology of the underlying analytic space $X_{\cx}$ and $A(m)=A\otimes_\Z (2 \pi i)^{m}\Z$ denotes the $m$-th Tate twist.
Restriction maps as required above exist in this setting and so we get refined unramified cohomology groups $H^i_{j,nr}(X,A(n))$.

By \cite[\S 19.1]{fulton}, there is a cycle class map   
\begin{align*}
\cl_X^i:\CH^i(X)\longrightarrow H^{2i}(X,\Z(i)) ,
\end{align*}
where $\CH^i(X):=\CH_{d_X-i}(X)$.
We let  $\Griff^i(X):=\ker(\cl_X^i)/\sim_{\alg} $.  

If $X$ is smooth and equi-dimensional, then $H^{BM}_{2d_X-i}(X_{\cx},A)\cong H^{i}_{sing}(X_{\cx}, A)$ and the above map agrees
 with the usual cycle class map in singular cohomology.

\begin{theorem} \label{thm:main-2-intro} 
Let $X$ be a separated scheme of finite type over $\C$ and define refined unramified cohomology as in (\ref{def:H_jnr-intro}) with cohomology theory  in (\ref{eq:Borel-Moore}). 
\begin{enumerate}
\item There are canonical isomorphisms\label{item:Tors(coker)} 
$$
\coker(\cl_X^i) _{\tors}\cong \frac{H^{2i-1}_{i-2,nr}(X,\Q/\Z(i))}{H^{2i-1}_{i-2,nr}( X,\Q(i) )},\ \ \ 
\ \ \ \Griff ^i(X)\cong \frac{H^{2i-1}_{i-2,nr}(X,\Z(i)) }{H^{2i-1} (X,\Z(i)) } .
$$
\item There is a transcendental Abel--Jacobi map \label{item:Abel-Jacobi}
\begin{align*} 
\lambda_{tr}^i: \Griff^i(X)_{\tors} \longrightarrow  \frac{H^{2i-1}(X,\Q/\Z(i))}{N^{i-1}H^{2i-1}(X,\Q(i) )}.
\end{align*}
If $X$ is a smooth projective variety, this agrees with Griffiths' transcendental Abel--Jacobi map \cite{griffiths} restricted to torsion cycles.   
Its kernel 
is  isomorphic to
$$
\ker(\lambda_{tr}^i)\cong  H^{2i-2}_{i-3,nr} (X,\Q/\Z(i)) /G^{i}H^{2i-2}_{i-3,nr} (X,\Q/\Z(i) ) 
$$
and its image is given by $\im(\lambda_{tr}^i)=N^{i-1}H^{2i-1}(X,\Q /\Z (i))_{\operatorname{div}}/N^{i-1}H^{2i-1}(X,\Q (i) ) $.
\end{enumerate}
\end{theorem}
 In the above theorem,  
$N^\ast$ denotes Grothendieck's coniveau filtration and $H^\ast(X,\Q/\Z(i))_{\operatorname{div}}\subset H^\ast(X,\Q/\Z(i))$ denotes the divisible subgroup. 
 Moreover,
 $G^{i}H^{2i-2}_{i-3,nr} (X,\Q /\Z(i) )$ denotes the subspace of $H^{2i-2}_{i-3,nr} (X,\Q /\Z(i) $ generated by classes  that admit a lift $\alpha\in H^{2i-2}(F_{i-2}X,\Q /\Z(i) )$ whose image $\delta(\alpha)\in H^{2i-1}(F_{i-2}X, \Z(i) )$ via the Bockstein map lifts to $H^{2i-1}(X, \Z(i) )$,  cf.\ Definition \ref{def:G} and Lemma \ref{lem:Gi} below.

The above theorem contains the aforementioned results from \cite{BO,CTV,Voi-unramified,Ma} as the special case where $i=2$ in item (\ref{item:Tors(coker)}) and $i=3$ in (\ref{item:Abel-Jacobi}), and  where $X$ is a smooth projective variety. 
Item  (\ref{item:Tors(coker)}) uses Hilbert 90, but not the Bloch--Kato conjecture. 
Item (\ref{item:Abel-Jacobi})  uses the Bloch--Kato conjecture in degree 2, proven by Merkurjev--Suslin,  but not in higher degrees.  

Item (\ref{item:Tors(coker)}) in Theorem \ref{thm:main-2-intro} leads to new results on the integral Hodge conjecture for uniruled varieties. 
Indeed, Voisin \cite{voisin-IHC} proved that the integral Hodge conjecture holds for smooth complex projective threefolds $X$ that are uniruled  (i.e.\ $Z^\ast(X)_{\tors}=0$) and conjectured that it should fail for codimension two cycles on rationally connected varieties of dimension at least four.
This has later been proven in \cite{CTV}  ($\dim X\geq 6$) and in full generality in \cite{Sch-JAMS}.
Taking products $X\times \CP^n$ with $\CP^n$, the examples in \cite{CTV,Sch-JAMS} also yield counterexamples to the integral Hodge conjecture on unirational varieties for cycles of codimension greater than two.
However,  in some sense these non-algebraic Hodge classes should still be regarded as degree four classes, because they are Gysin pushforwards of non-algebraic degree four Hodge classes on a subvariety of $X\times \CP^n$ (namely $X\times \{pt.\}$).  

The tools of this paper allow us to go further by studying the integral Hodge conjecture for Hodge classes (of arbitrary degree)  
 in the following strong sense.

\begin{corollary} \label{cor:IHC}
For any integer $i\geq 2$, there is a smooth uniruled complex projective variety $X$ such that the integral Hodge conjecture fails for codimension $i$-cycles on $X$ in a way that cannot be explained by the failure on proper subvarieties of $X$ in the following sense:
There is a class $\alpha\in  \coker(\cl_X^i) _{\tors}$ such that for any closed subscheme $Z\subset X$ of codimension $j\geq 1$, the class $\alpha$ is not in the image of the natural map $ \coker(\cl_Z^{i-j}) _{\tors} \to  \coker(\cl_X^i) _{\tors}$.
\end{corollary}
 
In the above corollary we may take $X=Y\times E$, where $Y$ is a certain unirational variety of dimension $3i$ and  $E$ is an elliptic curve,  see Theorem \ref{thm:IHC:example} below.
The problem of finding a unirational variety $X$ with the property stated in the corollary remains open for $i\geq 3$.

\subsubsection{Arbitrary ground fields}
Theorem \ref{thm:main-2-intro} admits an $\ell$-adic analogue over any field $k$ in which $\ell$ is invertible. 
The corresponding (co-)homology functor will be the $\ell$-adic pro-\'etale Borel--Moore cohomology, see \cite{BS} and Proposition \ref{prop:proetale-coho-arbitrary-field} below.
For instance, 
\begin{align} \label{eq:Borel-Moore-l-adic}
H^{i}(X, \Z_\ell(n)):=H^{i-2d_X} (X_{\proet},\pi_X^!\widehat \Z_\ell(n-d_X)), 
\end{align}
where $\pi_X:X\to \Spec k$ is the structure map and $d_X=\dim X$.
We construct in Section \ref{subsec:cli} cycle class maps
$$
\cl_X^i:\CH^i(X)_{\Z_\ell} \longrightarrow H^{2i}(X,\Z_\ell(i)),\ \ \text{where}\ \ \CH^i(X)_{\Z_\ell}:=\CH_{d_X-i}(X)\otimes_{\Z} \Z_\ell ,
$$ 
which for $X$ smooth and equi-dimensional coincide with  Jannsen's cycle class in continuous \'etale cohomology, see Lemma \ref{lem:cl=Jannsen}. 
Let $N^\ast$ be the associated coniveau filtration on $\CH^i(X)_{\Z_\ell} $ and put
$$
A_0^i(X)_{\Z_\ell}:=N^0\CH^i(X)_{\Z_\ell}/N^{i-1}\CH^i(X)_{\Z_\ell}.
$$ 
This is the $\ell$-adic Griffiths group of homologically trivial  $\ell$-adic cycles modulo algebraic equivalence if $k$ is algebraically closed, and it is the group of homologically trivial $\ell$-adic cycles modulo rational equivalence if $k$ is finitely generated, see \cite[Lemmas 5.7 and 5.8]{jannsen-3}  and Lemma \ref{lem:AiX} below.
We denote the torsion subgroup of $A_0^i(X)_{\Z_\ell}$ by $A_0^i(X)[\ell^\infty]$.

\begin{theorem} \label{thm:singular}
Let $k$ be a field and let $\ell$ be a prime invertible in $k$.
Let $X$ be a separated scheme of finite type over $k$ and define refined unramified cohomology as in (\ref{def:H_jnr-intro}) with cohomology theory in (\ref{eq:Borel-Moore-l-adic}).
\begin{enumerate}
\item There are canonical isomorphisms  
\label{item:coker-proet}
$$ 
\coker(\cl_X^i) [\ell^\infty] \cong \frac{H^{2i-1}_{i-2,nr}(X,\Q_\ell /\Z_\ell (i))}{H^{2i-1}_{i-2,nr}( X,\Q_\ell(i) ) } \ \ \text{and}\ \ A^i_0(X)_{\Z_\ell}\cong \frac{H^{2i-1}_{i-2,nr}(X,\Z_\ell(i))}{H^{2i-1}(X,\Z_\ell(i))}.
$$
\item \label{item:lambda-thm:singular}
There is a transcendental Abel--Jacobi map  
$$
\lambda_{tr}^i: A^i_0(X)[\ell^\infty] \longrightarrow H^{2i-1}(X,\Q_\ell/\Z_\ell(i))/ N^{i-1}H^{2i-1}(X,\Q_\ell(i) ) .
$$
If  $X$ is a smooth projective variety and $k$ is algebraically closed, then this map is induced by Bloch's Abel--Jacobi map on torsion cycles \cite{bloch-compositio}.
Its kernel is isomorphic to 
$$
\ker(\lambda_{tr}^i)\cong  H^{2i-2}_{i-3,nr} (X,\Q_\ell /\Z_\ell (i)) /G^{i}H^{2i-2}_{i-3,nr} (X,\Q_\ell/\Z_\ell(i) )
$$  
and its image is given by$\im(\lambda_{tr}^i)=N^{i-1}H^{2i-1}(X,\Q_\ell/\Z_\ell(i))_{\operatorname{div} }/N^{i-1}H^{2i-1}(X,\Q_\ell(i) )$.
\end{enumerate}
\end{theorem}

The filtration $N^\ast$ is the coniveau filtration and $G^\ast$ is defined similarly as in Theorem \ref{thm:main-2-intro},  see Definition \ref{def:G} and Lemma \ref{lem:Gi} below.
Moreover, $H^{2i-1}(X,\Q_\ell/\Z_\ell(i))_{\operatorname{div} } \subset H^{2i-1}(X,\Q_\ell/\Z_\ell(i))$ denotes the image of
$H^{2i-1}(X,\Q_\ell (i))\to H^{2i-1}(X,\Q_\ell/\Z_\ell(i))$.

The computation of $\ker(\lambda_{tr}^i)$ uses as before Merkurjev--Suslin's theorem.  
The Bloch--Kato conjecture is not used otherwise (in particular not in item (\ref{item:coker-proet})).

The first isomorphism in item (\ref{item:coker-proet}) generalizes a result of Kahn \cite{Kahn} who proved it for $i=2$ and  $X$ smooth projective.  

\subsection{Comparison to Bloch--Ogus theory and Kato homology} \label{subsec:intro:BO}

Let $X$ be an algebraic scheme over a field $k$ and let $H^i(-,A(n))$ be one of the (co)-homology theories considered above. 
For a point $x\in X$ with closure $Z_x:=\overline{\{x\}}\subset X$, we let $H^i(x,A(n))=H^i(F_0Z_x,A(n))$.  
The Gysin sequence (i.e.\ long exact sequence of pairs), yields in the colimit a long exact sequence
\begin{small}
\begin{align} \label{eq:exact-couple} 
   H^{p+q-1}(F_{p}X,A(n)) \stackrel{f}\to   H^{p+q-1}(F_{p-1}X,A(n))  \stackrel{\del}\to  \bigoplus_{x\in X^{(p)}}H^{q-p}(x,A(n-p)) \stackrel{\iota_\ast} \to    H^{p+q}(F_{p}X,A(n)) ,
\end{align} 
\end{small}
see Lemma \ref{lem:les} below.
Note that the image of $f$ agrees by definition with the refined unramified cohomology group  $H^{p-1+q}_{p-1,nr}(X,A(n))$, which thus coincides with the kernel of the residue map $\del$ above.
This shows in particular that $H^q_{0,nr}(X,A(n))$ corresponds to traditional unramified cohomology. 

The above sequence gives rise to an exact couple $D_1\stackrel{f}\to D_1\stackrel{\del}\to E_1\stackrel{\iota_\ast}\to D_1$, where  
$$
D_1^{p,q}=H^{p-1+q} (F_{p-1}X,A(n)),\ \ \ \ \ \text{and} \ \ \ \ \  
E_1^{p,q}=\bigoplus_{x\in X^{(p)}} H^{q-p}(x,A(n-p)),
$$
and 
$f$, $\del$, and $\iota_\ast$ have bi-degrees $(-1,1)$, $(0,0)$, and $(1,0)$, respectively.
The associated spectral sequence $E_1^{p,q}\Longrightarrow H^{p+q}(X,A(n))$ is convergent.
The derived couple has  the form
$
D_2 \to D_2\to E_2\to D_2
$, where
$$
D_2=\bigoplus_{p,q} D_2^{p,q},\ \ \ \ \ D_2^{p,q}=H^{p-1+q}_{p-1,nr}(X,A(n)) ,
$$
is the direct sum of all refined unramified cohomology groups of $X$.

It follows from Lemma \ref{lem:del=delta} below that $d_1:E_1\to E_1$ agrees with the differential of the coniveau spectral sequence from \cite[\S 3]{BO}, but see also  Remark \ref{rem:refined-intro-2} below.
Hence, $E_2$ agrees with the second page of the coniveau spectral sequence and the derived couple $
D_2 \to D_2\to E_2\to D_2
$
shows that $E_2$ is up to extensions determined by $D_2$,  hence by refined unramified cohomology, see Proposition \ref{prop:higher-unramified} below.
In the special case where $X$ is smooth and equi-dimensional,  the Gersten conjecture proven in \cite{BO} (see also \cite{CTHK}) thus implies that the  cohomologies $H^p(X_{Zar},\mathcal H^q )$  of the Zariski sheaf associated to $U\mapsto H^q(U,A(n))$ are up to extensions determined by refined unramified cohomology.

Without any smoothness assumption on $X$,  but under the condition that the ground field $k$ has finite cohomological dimension $c$,  the derived couple $D_2\to D_2\to E_2\to D_2$ yields for $A=\Z/\ell^r$ canonical isomorphisms 
$$
E_2^{p,d+c} \stackrel{\sim}\longrightarrow H^{p+d+c}_{p,nr}(X,\mu_{\ell^r}^{\otimes n}) ,\ \ \ \ \ d=\dim X,
$$
see Corollary \ref{cor:kato} below.
By definition,  $E_2^{p,d+c}$ agrees with the Kato homology of $X$ (see e.g.\ \cite{kato,KeSa,tian}),  and so we find that the latter is in fact a special instance of refined unramified cohomology, cf.\ Remark \ref{rem:kato} below. 
We remark that for  $X$ smooth projective, Kato homology as well as traditional unramified cohomology are stable birational invariants of $X$,  see e.g.\ \cite{CTO,CTV,tian}.

\begin{remark}\label{rem:refined-intro-2} 
One of the key differences of this paper compared to previous work is the observation that for our purposes,  the couple $D_1\to D_1\to E_1\to D_1$  is  better suited than the couple $D_1'\stackrel{f'}\to D_1'\to E_1\to D_1'$, used in \cite[\S 3]{BO} to define the coniveau spectral sequence.
(Both couples stem from the long exact sequence of triples, but applied to different geometric situations.) 
Moreover, we will not pass to the coniveau spectral sequence (as done e.g.\ in  \cite{bloch-compositio,CTV,Kahn,Voi-unramified,Ma}), but work directly with the above couple, which contains more information. 
\end{remark}

\begin{remark} 
Assume that $X$ is smooth and equi-dimensional.
The main result of \cite{BO} (see also \cite{CTHK}) is that the map $f':D_1'\to D_1'$ is zero locally on $X$ with the exception of only some trivial bidegrees; as a consequence, $E_2^{p,q}=0$  locally on $X$ for all $p\neq 0$.
In contrast, the image of $f:D_1\to D_1$ is refined unramified cohomology and this invariant in general does not vanish locally on $X$.
In fact, 
the local vanishing of $E_2^{p,q}$ for $p\neq 0$ implies that $f:D^{p+1,q-1}_2\to D^{p,q}_2$ is an isomorphism locally on $X$ for all $p\geq 1$.
It follows that the Zariski sheaf $\mathcal H^i_{j,nr}$ associated to $U\mapsto H^i_{j,nr}(U,A(n))$ does not depend on $j\geq 0$,  hence agrees with  $\mathcal H^i $   for all $j\geq 0$, which is in general nonzero.  
(This shows in particular that at least in the smooth case, 
 refined unramified cohomology contains no interesting local information.) 
\end{remark}

\subsection{Homology or cohomology?} \label{subsec:homology-coho}
The results above relied on a twisted Borel--Moore homology theory $H^{BM}_\ast(-,A(n))$ with corresponding Borel--Moore cohomology theory $H^i(X,A(n)):=H^{BM}_{2d_X-i}(X,A(d_X-n))$, see  (\ref{eq:Borel-Moore}) and (\ref{eq:Borel-Moore-l-adic}).
We will collect the properties of this functor that are crucial for us in Section \ref{sec:axioms} below.
In sheaf theoretic terms, 
Borel--Moore cohomology will in practice be the (hyper-)cohomology of some complex of sheaves.
If $X$ is smooth and equi-dimensional, Poincar\'e duality identifies this complex to a locally constant sheaf.
(For instance,  $\ell$-adic Borel--Moore pro-\'etale cohomology is given by  $H^i(X,\Z_\ell(n))=\RR ^i\Gamma(X_{\proet},\pi_X^! \widehat \Z_\ell(n-d_X)[-2d_X])$, where $\pi_X:X\to \Spec k$ and $\pi_X^!\cong \pi_X^\ast(d_X)[2d_X]$ if $X$ is smooth and equi-dimensional of dimension $d_X$, see Section \ref{sec:proet} below.)
%
The resulting theory thus coincides on smooth equi-dimensional algebraic schemes with ordinary cohomology, but it differs in general.
In particular, as long as one is interested only in smooth equi-dimensional schemes, $H^i(X,A(n))$ can be identified with ordinary cohomology in all our applications.
(Working with singular schemes is however important for the proof of several of the main results such as Theorems \ref{thm:coniveau1}, \ref{thm:green}, and \ref{thm:green-ell-adic-arbitrary-field} even if the total space is smooth projective.)

The functoriality properties of Borel--Moore cohomology differ from those of ordinary cohomology: the latter has arbitrary pullbacks but no pushforwards, while the former  admits only pullbacks along \'etale maps, but has proper pushforwards (which shift the degree), see Section \ref{sec:axioms} below.
The situation is similar to the distinction between ordinary singular homology and Borel--Moore homology (i.e.\ finite versus locally finite singular chains)  in topology, which have also different functoriality properties; they agree for compact analytic spaces but differ in general. 

Instead of exploiting the notion of Borel--Moore cohomology, we could of course use the formula $H^i(X,A(n))=H^{BM}_{2d_X-i}(X,A(d_X-n))$ to write everything in terms of Borel--Moore homology, which may be preferred by some readers.
The reason we use Borel--Moore cohomology and wrote this paper cohomologically is that in the important special case where $X$ is smooth and equi-dimensional, $H^i(X,A(n))$ will coincide with ordinary cohomology in all our applications.
This has in particular the advantage that the formulas that we prove for singular varieties and in arbitrary codimension reduce in the special case of smooth projective varieties to those in  \cite{CTV,Kahn,Voi-unramified,Ma}, which motivate this paper.
Moreover, the applications of the theory in \cite{Sch-Griffiths,Sch-preparation} concern  smooth projective varieties and use the identification of Borel--Moore cohomology with ordinary cohomology.
This allows one to make use of cup products, which will be crucial (and which requires a cohomological formulation).
Writing this paper homologically would thus make it significantly harder to read those applications. 

After all it is a matter of formal manipulations to
 rewrite this paper homologically, but note that it will not be enough to just use $H^i(X,A(n))=H^{BM}_{2d_X-i}(X,A(d_X-n))$, one should also change the indices in the filtration $F_\ast X$ to make the indices in the resulting formulas in Theorems \ref{thm:main-2-intro} and \ref{thm:singular} appealing.
Unfortunately, the translation between the homological and the cohomological notation is tedious, so that we restrict ourselves to only one version here.

While only a matter of notation, we do believe that the notion of Borel--Moore cohomology may be useful also in future.

\section{Notation} \label{sec:notation}

A field is said to be finitely generated, if it is finitely generated over its prime field.
An algebraic scheme is a separated scheme of finite type over a field.
A variety  is an integral algebraic scheme. 
An open subset of a scheme will always refer to a Zariski open subset, unless specified otherwise.
The dimension of an algebraic scheme is the maximum of the dimensions of its irreducible components.

For an algebraic scheme $X$, we write $X_{(i)}$ for the set of all points $x\in X$ with $\dim(\overline {\{x\}})=i$.
We then define $X^{(i)}:=X_{(d_X-i)}$, where $d_X=\dim X$.
That is, $x\in X$ lies in $X^{(i)}$ if and only if $ \dim X-\dim(\overline {\{x\}})=i$.
Note that this is slightly non-standard, as it does not imply that the codimension of $x$ defined locally in $X$ is $i$, but it has the advantage that the Chow group $\CH^i(X):=\CH_{d_X-i}(X)$ (see \cite{fulton}) is the quotient of $\bigoplus_{x\in X^{(i)}}[x]\Z$ by rational equivalence, where $[x]\Z $ denotes the free $\Z$-module with generator $[x]$.
We refer to \cite[\S 10.3]{fulton} for the definition of algebraic equivalence of cycles on algebraic schemes. 

Whenever $G$ and $H$ are abelian groups (or $R$-modules for some ring $R$) so that there is a canonical map $H\to G$ (and there is no reason to confuse this map with a different map), we write $G/H$ as a short hand for $\coker(H\to G)$.
For an abelian group $G$, we denote by $G[\ell ^r]$ the subgroup of $\ell^r$-torsion elements, and by $G[\ell^\infty]:=\bigcup_{r}G[\ell ^r]$ the subgroup of elements that are $\ell^r$-torsion for some $r\geq 1$.
We further write $\Tors(G)$ or $G_{\tors}$ for the torsion subgroup of $G$.

Let $I$ be a directed index set and let  $(G_i)_{i\in I}$ be a direct system of abelian groups. 
We then denote by 
$$
\lim_{\substack{\longrightarrow\\ i\in I}}G_i
$$
the direct limit of this system.
Synonymously, we sometimes also call this the (filtered) colimit of $(G_i)_{i\in I}$ and denote it by $\colim G_i$.

\section{Warm-up: a simple proof of the Colliot-Th\'el\`ene--Voisin theorem} \label{sec:CTV}
Let $X$ be a smooth complex variety. 
In this section we present as a warm-up a proof of the formula
\begin{align} \label{eq:CTV-theorem}
\coker(\cl^2_X) _{\tors}\cong  H^{3}_{nr}(X,\Q/\Z )/H^{3}_{nr}( X,\Q ) ,
\end{align}
which is due to Colliot-Th\'el\`ene--Voisin \cite{CTV}.
Their original proof relied on  Voevodsky's proof of the Bloch--Kato conjecture; 
Kahn \cite{Kahn}  later showed that Bloch--Kato in degree 2, i.e.\ the Merkurjev--Suslin theorem, suffices to prove the result.
Both  approaches use the Gersten conjecture proven by Bloch--Ogus, which identifies unramified cohomology with the global sections of a certain Zariski sheaf, see \cite{BO}.
In this section we explain a simpler argument that does not need Bloch--Kato in any degree and which does not make use of the Gersten conjecture.
The proof presented here generalizes easily to give the result for arbitrary codimension and in fact on possibly singular schemes,  see Theorem \ref{thm:IHC} below.
 This is the starting point of the more general theory presented in the body of the paper. 
 
To fix notation  in this section, we denote  by $H^i(X,A)$ singular cohomology of the underlying analytic space $X_{\cx}$ with coefficients in an abelian group  $A$.
This coincides with Borel--Moore cohomology as considered  in the rest of this paper, because $X$ is smooth and irreducible (hence equi-dimensional).

As above, we define $H^i(F_jX,A)$ as the direct limit over $H^i(U,A)$ where $U\subset X$ runs through all (Zariski) open subsets whose complement has codimension at least $j+1$.
The unramified cohomology of $X$ is defined by $H^i_{nr}(X,A)=\im(H^i(F_1X,A)\to H^i(F_0X,A))$.
The Gysin sequence implies that this agrees with the definition given in \cite[Theorem 4.1.1(a)]{CT} (cf.\ Lemma \ref{lem:les} below).
In other words, an element $[\alpha]\in H^i_{nr}(X,A)$ is represented by a class $\alpha \in H^i(U,A)$ for some open subset $U\subset X$ whose complement has codimension at least 2 (such open subsets are called ``big'') and two such representatives yield the same element in $H^i_{nr}(X,A)$ if they coincide on some dense open subset of $X$.

Our proof of (\ref{eq:CTV-theorem}) relies on the following lemma. 

\begin{lemma} \label{lem:intro:Psi}
The  natural restriction  map is an isomorphism
$$
f:\frac{H^{3} (F_1X,\Q/\Z )}{H^{3} ( F_1X,\Q )}\stackrel{\sim}\longrightarrow \frac{H^{3}_{nr}(X,\Q/\Z )}{H^{3}_{nr}( X,\Q )}.
$$
\end{lemma}

\begin{proof}  
Since $f$ is clearly surjective, it suffices to prove that it is injective.

Note that  $H^{3} ( F_1X,\Q )\to H^{3}_{nr}( X,\Q )$ is surjective by definition.
Hence, in order to prove the injectivity of $f$ it suffices to show the following:
Let $U\subset X$ be a big open subset and let $\alpha\in H^3(U,\Q/\Z)$ such that 
$$
\alpha|_V=0\in  H^3(V,\Q/\Z)
$$
for some dense open subset $V\subset U$.
Then we need to show that up to removing a codimension 2 subset from $U$, the class $\alpha$ lifts to $H^3(U,\Q)$. 
Equivalently,  we need to show that the image $\delta(\alpha)\in H^4(U,\Z)_{\tors}$ via the Bockstein map (associated to $0\to \Z\to \Q\to \Q/\Z\to 0$) vanishes after removing a codimension 2 subset from $U$.

Up to removing a codimension 2 subset from $U$, we may assume that $D:=U\setminus V$ is smooth of pure codimension 1 in $U$.
Since $\alpha|_V=0$, the Gysin sequence shows that there is a class $\xi\in H^1(D,\Q/\Z)$ with $\alpha=\iota_\ast \xi$, where $\iota_\ast:H^1(D,\Q/\Z)\to H^3(U,\Q/\Z)$ denotes the Gysin pushforward.
Identifying the respective cohomology groups via Poincar\'e duality with Borel--Moore homology,  it follows directly from the definitions that the Bockstein map is compatible with $\iota_\ast$. 
It thus suffices to show that
$$
 \delta (\xi) \in  H^2(D, \Z)_{\tors}
$$
vanishes after removing a codimension 1 subset of $D$.
This in turn is   a well-known consequence of Hilbert 90, see \cite[end of Lecture 5]{bloch},  
which concludes the proof of the lemma. 
\end{proof}

By the above lemma, it suffices to construct an isomorphism
\begin{align} \label{eq:intro:Phi}
g: \coker(\cl^2_X) _{\tors}\longrightarrow H^{3} (F_1X,\Q/\Z )/H^{3} ( F_1X,\Q ) .
\end{align}
Here we note that both sides in (\ref{eq:intro:Phi}) remain unchanged if we remove from $X$ a closed codimension 3 subset (this is obvious for the right hand side and it follows from the Gysin sequence and purity for the left-hand side).
We will allow ourselves to perform such shrinkings in what follows (this could   be avoided if we were using Borel--Moore cohomology so that we can work with singular schemes).
Let $\alpha\in H^4(X,\Z)$ such that some multiple $n\alpha=\cl_X^2(z)$ is algebraic.
Let $Z:=\supp z$.
Up to removing $Z^{\sing}$ from $X$, we may assume that $Z$ is smooth.
The Gysin sequence then yields
\begin{align} \label{eq:intro:Gysin1}
H^0(Z,\Z)\stackrel{\iota_\ast}\longrightarrow
H^4(X,\Z)\longrightarrow H^4(U,\Z) \stackrel{\del}\longrightarrow H^1(Z,\Z) ,
\end{align}
where $\del$ denotes the residue map and the pushforward $\iota_\ast$ corresponds to the cycle class map.
Since $n\alpha=\cl^2_X(z)\in \im(\iota_\ast)$, we find that $\alpha|_U\in H^4(U,\Z)_{\tors}$ is torsion.
The Bockstein map $\delta:H^3(U,\Q/\Z)\to H^4(U,\Z)$ induces an isomorphism
$$
\delta^{-1}: H^4(U,\Z)_{\tors}\stackrel{\cong}\longrightarrow H^3(U,\Q/\Z)/H^3(U,\Q) .
$$
The right hand side in the above isomorphism maps naturally to the right hand side in (\ref{eq:intro:Phi}) (in fact, the latter is just the direct limit of the former where one runs through all big open subsets $U\subset X$). 
The map $g$ above is then defined by
$$
g(\alpha):=[\delta^{-1}(\alpha|_U)].
$$
The Gysin sequence implies that this definition is well-defined, i.e.\ $g(\alpha)$ does not change if we add to $\alpha$ some algebraic class.

We aim to construct an inverse of $g$.
To this end, let  $\beta\in H^3(U,\Q/\Z)$ for some big open subset $U\subset X$.
The class $\alpha':=\delta(\beta)\in H^4(U,\Z)$ is a torsion class.
Let $Z=X\setminus U$.
Up to shrinking $X$, we may assume that $Z$ is smooth of pure codimension 2 in $X$.
Then we have an exact sequence as in (\ref{eq:intro:Gysin1}) and the (trivial) fact that $H^1(Z,\Z)$ is torsion-free implies that the torsion class $\alpha'\in H^4(U,\Z)$ lifts to a class $\alpha\in H^4(X,\Z)$.
The fact that $\alpha'$ is torsion implies that some multiple of $\alpha$ is algebraic and hence $[\alpha]\in \coker(\cl^2_X)_{\tors}$.
If $\beta$ lifts to a class in $H^3(U,\Q)$, then $\alpha'=0$ and so the above construction yields a map
$$
g': H^{3} (F_1X,\Q/\Z )/H^{3} ( F_1X,\Q )\longrightarrow \coker(\cl^2_X) _{\tors} ,\ \ [\beta]\mapsto [\alpha] .
$$
It follows from the construction that $g$ and $g'$ are inverse to each other.
Hence, $g$ is an isomorphism and the formula in (\ref{eq:CTV-theorem}) is proved.

 \section{Borel--Moore cohomology} \label{sec:axioms}
  
We list here properties of a functor, that we call Borel--Moore cohomology, which allows to run the  arguments from  Section \ref{sec:CTV} (and more). 
 Technically speaking, Borel--Moore cohomology will in all our applications agree up to shifts with Borel--Moore homology,  see also  Section \ref{subsec:homology-coho} above.
In practice and in terms of sheaf theory,  this means that Borel--Moore cohomology will be the hypercohomology of some complex of sheaves on some site; the complex in question has by Poincar\'e duality the property that it simplifies on smooth equi-dimensional schemes to a locally constant sheaf.
In other words,  
 on smooth equi-dimensional varieties, Borel--Moore cohomology will agree with ordinary cohomology.
However, on singular spaces, Borel--Moore cohomology and ordinary cohomology differ: 
for Borel--Moore cohomology we require pullbacks only along open immersions of equi-dimensional schemes (and not along arbitrary morphisms), while
 we require pushforwards along proper morphisms of possibly singular schemes (and not only for smooth equi-dimensional schemes).

\begin{definition} \label{def:V}
Let $\VV$ be a category of Noetherian schemes such that the morphisms are given  by open immersions $U\hookrightarrow X$ of schemes with $\dim (U)=\dim (X)$.
We call $\VV$ constructible, if 
for any $X\in \VV$, the following holds: 
\begin{itemize}
\item if $Y\hookrightarrow X$ is an open or closed immersion, then $Y\in \VV$;
\item if $X\in \VV$ is reduced, then the normalization of $X$ is also in $\VV$.
\end{itemize} 
\end{definition}

\begin{definition} \label{def:Borel--Moore-cohomology}
Let $\VV$ be a constructible category of Noetherian schemes as in Definition \ref{def:V}.
Let $R$ be a ring and let $\mathcal A\subset \Mod_R$ be a full subcategory of $R$-modules with $R\in \mathcal A$.
A twisted Borel--Moore cohomology theory on $\VV$  with coefficients in $\mathcal A$ is a family of 
contravariant 
 functors
\begin{align} \label{eq:functor}  
\mathcal V \longrightarrow \Mod_R,\ \ X\mapsto   H_{BM}^i(X,A(n)) \ \ \ \ \ \ \text{with $i,n\in \Z$ and $A\in \mathcal A$}
\end{align} 
that are covariant in $A$ and such that the following holds, where we write for simplicity
$$ H^i(X,A(n)):= H_{BM}^i(X,A(n)) , $$
\begin{enumerate}[label=\textbf{P\arabic*}]
\item \label{ax:pushforward}
For $X,Y\in \VV$ and any proper morphism $f:X\to Y$ of schemes of relative codimension $c=\dim Y-\dim X$,  there are functorial pushforward maps 
$$
f_\ast : H^{i-2c}(X,A(n-c))\longrightarrow H^{i }(Y,A(n )),
$$ 
 compatible with pullbacks along morphisms  in $\VV$. 
\item \label{ax:Gysin} For any pair $(X,Z)$ of schemes in $\VV$ with a closed immersion $Z\hookrightarrow X$ of codimension $c=\dim(X)-\dim(Z)$  and with complement $U$ with $\dim(X)=\dim(U)$, there is a Gysin exact sequence
\begin{align*} 
\dots\longrightarrow 
H^i(X,A(n))\stackrel{r}\longrightarrow H^i(U,A(n))\stackrel{\del}\longrightarrow H^{i+1-2c}(Z,A(n-c))\stackrel{\iota_\ast}\longrightarrow H^{i+1}(X,A(n)) 
 \longrightarrow \dots
\end{align*}
where $r$ is induced by functoriality,  $\del$  is called residue map and $\iota_\ast$ is induced by proper pushforward from (\ref{ax:pushforward}). 
The Gysin sequence is functorial with respect to pullbacks along  open immersions $f:V\hookrightarrow X$ with $\dim V= \dim X$,  $\dim(V\cap Z)=\dim Z$, and $\dim (V\setminus(V\cap Z))=\dim V$, giving rise to a commutative diagram
$$
\xymatrix{
H^i(X,A(n))\ar[r]^-{r}\ar[d]^{f^\ast}&H^i(U,A(n))\ar[r]^-{\del}\ar[d]^{f^\ast}& H^{i+1-2c}(Z,A(n-c))\ar[d]^{f^\ast}\ar[r]^-{\iota_\ast} &H^{i+1}(X,A(n)) \ar[d]^{f^\ast}\\
H^i(V,A(n))\ar[r]^-{r} &H^i(V\cap U,A(n))\ar[r]^-{\del} & H^{i+1-2c}(V\cap Z,A(n-c)) \ar[r]^-{\iota_\ast} &H^{i+1}(V,A(n))  
}
$$
for all $i$.
Similarly, if $f:X'\to X$ is proper with $Z'=f^{-1}(Z)$ and $\dim X'=\dim(X'\setminus Z')$, then the proper pushforward along $f$ induces for all $i$ a commutative diagram
$$
\xymatrix{
H^i(X',A(n))\ar[r]^-{r}\ar[d]^{f_\ast}&H^i(U',A(n))\ar[r]^-{\del}\ar[d]^{f_\ast}& H^{i+1-2c}(Z',A(n-c))\ar[d]^{f_\ast}\ar[r]^-{\iota_\ast} &H^{i+1}(X',A(n)) \ar[d]^{f_\ast}\\
H^i(X,A(n))\ar[r]^-{r} &H^i(U,A(n))\ar[r]^-{\del}& H^{i+1-2c}(Z,A(n-c))\ar[r]^-{\iota_\ast} &H^{i+1}(X,A(n)) .
}
$$
\item \label{ax:normalization}
For any $X\in \VV$ and  $x\in X$, the groups
\begin{align} \label{eq:Hi(kappa(x))}  
H^i(x,A(n)):=\lim_{\substack{\longrightarrow \\ \emptyset \neq V_x\subset \overline {\{x\}}}} H^i(V_x,A(n))  , 
\end{align}
where $V_x$ runs through all (Zariski) open dense subsets of $\overline{\{x\}}\subset X$ (with the reduced subscheme structure), satisfy
$
H^{i}(x,A(n))=0
$ for $i<0$.
Moreover, there are isomorphisms $ H^0(x,A(0))\cong A$ that are functorial in $A$, and for $A=R$ there is a distinguished class $[x]\in H^0(x,R(0))$ (called the fundamental class)  such that 
$
H^0(x,R(0))=[x] R
$ 
is freely generated by $[x]$.
\end{enumerate}
\end{definition}
 
 \begin{remark} 
 We warn the reader that  even if 
 $\Spec \kappa(x)\in \VV$, the cohomology of a point $x\in X$ in (\ref{eq:Hi(kappa(x))})
 may not agree with $H^i(\Spec \kappa(x) ,A(n))$. 
This phenomenon is not new but already present in \cite{BO} and in any other work where the Bloch-Ogus resolution for non-torsion coefficients is used. 
\end{remark}

In what follows,  we will usually write $\mu_{\ell^r}^{\otimes n}:=\Z/\ell^r(n)$.  

\begin{definition} \label{def:Borel--Moore-cohomology-l-adic}
Let $\VV$ be a constructible category of Noetherian schemes as in Definition \ref{def:V}.
Let $\ell$ be a prime and let $R=\Z_\ell$.
A twisted Borel--Moore cohomology theory on $\VV$ with coefficients in $\mathcal A\subset \Mod_{\Z_\ell}$ (see Definition \ref{def:Borel--Moore-cohomology}) is called $\ell$-adic, if  $\Z_\ell$, $\Q_\ell$, $\Q_\ell/\Z_\ell$, and $\Z/\ell^r$ for all $r\geq 1$ are contained in $\mathcal A$, such that the following holds:
\begin{enumerate}[label=\textbf{P\arabic*}]
\setcounter{enumi}{3}
\item \label{ax:Ql/Zl,Ql} Functoriality in the coefficients induces  isomorphisms of functors
\begin{align*} 
 \lim_{\substack{\longrightarrow\\ r}} H^i(-,\mu_{\ell^r}^{\otimes n})\stackrel{\sim} \longrightarrow H^i(-,\Q_\ell/\Z_{\ell}(n)) \ \ \text{and}\ \ H^i(-,\Z_{\ell}(n))\otimes_{\Z_\ell}\Q_{\ell}\stackrel{\sim}\longrightarrow H^i(-,\Q_{\ell}(n)) .
\end{align*}  
\item For any $X\in \VV$, there is a long exact 
Bockstein sequence
\begin{align*} 
\dots \longrightarrow H^i(X,\Z_\ell(n))\stackrel{\times \ell^r}\longrightarrow H^i(X,\Z_\ell(n)) \longrightarrow H^i(X,\mu_{\ell^r}^{\otimes n}) \stackrel{\delta}\longrightarrow & H^{i+1}(X,\Z_\ell(n)) \stackrel{\times \ell^r}\longrightarrow \dots
\end{align*}
where $H^i(X,\Z_\ell(n)) \to H^i(X,\mu_{\ell^r}^{\otimes n})$ is given by functoriality in the coefficients and where  $\delta$ is called the Bockstein map. 
This sequence is functorial with respect to proper pushforwards and pullbacks along morphisms in $\VV$.\label{ax:Bockstein} 
\item    
For any $X\in \VV$ and $x\in X$, there is a map $\epsilon : \kappa(x)^\ast\to H^1( x,\Z_\ell(1))$ such that Hilbert 90 holds in the sense that
the map
$\overline 
\epsilon:\kappa(x)^\ast\to H^1( x,\mu_{\ell^r}^{\otimes 1})
$
 induced by reduction modulo $\ell^r$ 
is surjective. 
Moreover, for $X\in \VV$ integral with generic point $\eta$, there is a unit $u\in \Z_\ell$ such that for any regular point $x\in X^{(1)}$, the natural composition
$$
\kappa(\eta)^\ast \stackrel{ \epsilon}\longrightarrow H^1(\eta,\Z_{\ell}(1))\stackrel{\del}\longrightarrow H^0( x ,\Z_{\ell}(0))=[x] \Z_{\ell},
$$
where $\del$ is induced by (\ref{ax:Gysin}) and the last equality comes from (\ref{ax:normalization}),
sends $f$ to $[x](u\cdot \nu_x(f)) $.
Here,  $\nu_x$ denotes the valuation on $\kappa(\eta)$ induced by $x$.\label{ax:Hilbert90} 
\end{enumerate}
\end{definition}

Let $X\in \VV$ be integral and let $U\subset X$ be a big open subset, i.e.\ $\dim(X\setminus U)<\dim X-1$. 
Then $H^2(U,\Z_\ell(1))\cong H^2(X,\Z_\ell(1))$ by (\ref{ax:Gysin}) and  (\ref{ax:normalization}) (see Corollary \ref{cor:F^i} below).
Taking the direct limit over all $U$ and using $H^0(x,\Z_\ell(0))=[x]\Z_\ell$ from (\ref{ax:normalization}), we find that the proper pushforwards from (\ref{ax:pushforward}) induce a cycle class map
\begin{align}\label{eq:iota_ast}
\iota_\ast:\bigoplus_{x\in X^{(1)}}[x]\Z_\ell\longrightarrow  
H^2(X,\Z_\ell(1)) .
\end{align} 
The following two options are of particular interest:
\begin{enumerate}[label=\textbf{P7.\arabic*}]
\item If $X$ is integral and regular,  the kernel of (\ref{eq:iota_ast}) is given by $\Z_\ell$-linear combinations of  algebraically trivial divisors.\label{ax:homolog=alg-div}
\item If $X$ is integral and regular, the kernel of (\ref{eq:iota_ast}) is given by  $\Z_\ell$-linear combinations of principal divisors. \label{ax:homolog=ratl-div}  
\end{enumerate}

\begin{definition} \label{def:adapted-alg/ratl-equivalence}
Let $\VV$ be a constructible category of Noetherian schemes, see Definition \ref{def:V}.
An $\ell$-adic twisted Borel--Moore cohomology theory $H^\ast(-,A(n))$  on $\VV$ as in Definition \ref{def:Borel--Moore-cohomology-l-adic}  is adapted to algebraic equivalence, if (\ref{ax:homolog=alg-div}) holds, and it is adapted to rational equivalence, if (\ref{ax:homolog=ratl-div}) holds.
\end{definition}

In addition to $\ell$-adic theories, we will also need the following integral variant.
To this end,  we perform in each of the statements (\ref{ax:Ql/Zl,Ql})--(\ref{ax:Hilbert90}),  (\ref{ax:homolog=alg-div}), and  (\ref{ax:homolog=ratl-div})  the formal replacement of symbols:
$$
\Z_\ell \rightsquigarrow \Z, \ \ \Q_\ell \rightsquigarrow \Q, \ \ \ \ell^r \rightsquigarrow r,
$$
and denote the corresponding statements by (\ref{ax:Ql/Zl,Ql}')--(\ref{ax:Hilbert90}'),  (\ref{ax:homolog=alg-div}'), and (\ref{ax:homolog=ratl-div}'), respectively. 

\begin{definition}\label{def:Borel--Moore-integral}
Let $\VV$ be a constructible category of Noetherian schemes (see Definition \ref{def:V}) and let $R=\Z$. 
A twisted Borel--Moore cohomology theory on $\VV$ with coefficients in $\mathcal A\subset \Mod_{\Z}$ as in Definition \ref{def:Borel--Moore-cohomology} is called integral,  if  $\Z$, $\Q$,  $ \Q/\Z$, and $ \Z/r$  for all $r\geq 1$ are contained in $\mathcal A$, such that items (\ref{ax:Ql/Zl,Ql}')--(\ref{ax:Hilbert90}') hold.
The theory is adapted to algebraic (resp.\ rational) equivalence, if item (\ref{ax:homolog=alg-div}') (resp.\ (\ref{ax:homolog=ratl-div}')) holds true.
\end{definition}

\begin{remark}
It seems natural to add in (\ref{ax:pushforward}) the requirement that pushforwards are compatible with the functoriality in the coefficients.
We did not do so because in this paper we will only need this compatibility for the natural maps $\Z_\ell\stackrel{\times \ell}\to \Z_\ell$, $\Z_\ell\to \Z/\ell^r$, $\Z_\ell\to \Q_\ell$, and $\Q_\ell\to \Q_\ell/\Z_\ell$, where it follows from (\ref{ax:Ql/Zl,Ql}) together with the compatibility of the Bockstein sequence with proper pushforwards formulated in (\ref{ax:Bockstein}). 
\end{remark}

\section{Definition of refined unramified cohomology and simple consequences} 
\label{sec:def}

In this section, we fix  a constructible category $\VV$ of Noetherian schemes, see Definition \ref{def:V}.
We further fix  a ring $R$ and a twisted Borel--Moore cohomology theory $H^\ast(-,A(n))$ on $\VV$ with coefficients in a full subcategory $\mathcal A\subset \Mod_R$,  see Definition \ref{def:Borel--Moore-cohomology}. 
In particular, (\ref{ax:pushforward})--(\ref{ax:normalization}) hold true.

For  $X\in \VV$ we write $F_jX:=\{x\in X\mid \codim(x)\leq j \} $, where  $\codim(x):=\dim (X)- \dim(\overline{\{x\}})$.
We then define
$$
H^i(F_jX,A(n)):=\lim_{\substack{\longrightarrow \\ F_jX\subset U\subset X}} H^i(U,A(n)) ,
$$
where the direct limit runs through all open subschemes $U\subset X$ with $F_jX\subset U$.
Since the direct limit functor (i.e. filtered colimits) is exact, many of the properties of $H^\ast(X,A(n))$ remain true for $F_jX$ in place of $X$.  
Moreover, for $m\geq j$,  there are canonical restriction maps $H^i(F_mX,A(n)) \to  H^i(F_jX,A(n))$. 

\begin{definition}
The $j$-th refined unramified cohomology of $X\in \VV$ with respect to a twisted Borel--Moore cohomology theory $H^\ast(-,A(n))$ on $\VV$ with coefficients in a full subcategory $\mathcal A\subset \Mod_R$,  is given by
$$
H^i_{j,nr}(X,A(n)):= \im \left( H^i(F_{j+1}X,A(n)) \longrightarrow  H^i(F_jX,A(n))  \right) .
$$
\end{definition}

\subsection{Three filtrations}

Following Grothendieck, the coniveau filtration on $H^i(X,A(n))$ is defined by
\begin{align} \label{def:coniveau-grothendieck}
N^jH^i(X,A(n)):=\ker(H^i(X,A(n))\to H^i(F_{j-1}X,A(n))).
\end{align}
There is a similar coniveau filtration on refined unramified cohomology, defined as follows.

\begin{definition} \label{def:N^j}
Let $X\in \VV$.  
The coniveau filtration $N^\ast$ is for $h\leq j+1$ given by 
$$
N^hH^i(F_jX,A(n)):=\ker\left(H^i(F_jX,A(n))\to H^i(F_{h-1}X,A(n)) \right),\ \ \ \ \text{and}
$$
$$
N^hH^i_{j,nr}(X,A(n)):=H^i_{j,nr}(X,A(n))\cap N^hH^i(F_jX,A(n)).
$$ 
\end{definition}

Somewhat dually to the coniveau filtration, we have the following filtration, which is also decreasing.

\begin{definition}\label{def:F}
Let $X\in \VV$.
The decreasing filtration
 $F^\ast$ is  for $m\geq j$ given by:
$$
F^m  H^i(F_jX,A(n)):= 
 \im \left( H^i(F_mX,A(n)) \longrightarrow  H^i(F_jX,A(n))  \right),\ \ \ \ \text{and}
$$
$$
F^mH^i_{j,nr}(X,A(n)):=H^i_{j,nr}(X,A(n))\cap F^m  H^i(F_jX,A(n)).
$$
\end{definition}

Note that for $m\geq j+1$, $F^m  H^i(F_jX,A(n))=F^mH^i_{j,nr}(X,A(n))
$.
 
Related to $F^\ast$ there is another filtration that will be important for us, and which exists only for $A=\Z/\ell^r$ or $A=\Q_\ell/\Z_\ell$.
To define it,  note that exactness of the direct limit functor ensures that the Bockstein sequence in (\ref{ax:Bockstein}) holds for $F_jX$ in place of $X$.
In particular, there is a Bockstein map 
$$
\delta:H^i(F_{j}X,\mu_{\ell^r}^{\otimes n})\to H^{i+1}(F_{j}X,\Z_{\ell}(n)).
$$

\begin{definition}\label{def:G}
Let $X\in \VV$. 
The decreasing filtration $G^\ast$ is for $m\geq j$  given by:
%
%
$$
\alpha\in G^m H^i(F_jX,\mu_{\ell^r}^{\otimes n})\ \ \ \Longleftrightarrow\ \ \ \delta(\alpha)\in F^mH^{i+1}(F_jX,\Z_{\ell}(n)) .
$$ 
We then define 
$$
 G^m H^i_{j,nr}(X,\mu_{\ell^r}^{\otimes n}) :=\im(G^mH^i(F_{j+1}X,\mu_{\ell^r}^{\otimes n})\to H^i(F_jX,\mu_{\ell^r}^{\otimes n}) ) .
$$ 
Using the isomorphism in (\ref{ax:Ql/Zl,Ql}),  we finally let
$$
 G^m H^i_{j,nr}(X,\Q_{\ell}/\Z_\ell(n)):=\lim_{\substack{\longrightarrow \\ r}}  G^m H^i_{j,nr}(X,\mu_{\ell^r}^{\otimes n}) .
$$
\end{definition}

\begin{remark}
By definition,  $F^\ast$ and $N^\ast$ on $H^i_{j,nr}(X,A(n))$  are induced by the corresponding filtration on $H^i(F_jX,A(n))$.
We warn the reader that the corresponding assertion does not hold true for $G^\ast$. 
\end{remark}

 \subsection{Consequence of the Gysin sequence}

\begin{lemma} \label{lem:XsqcupY}
Let $X\sqcup Y\in \VV$ with $\dim X=\dim Y$ and let $A\in \mathcal A$.
Then the canonical map given by pullback is an isomorphism:
$$
H^i(X\sqcup Y,A(n)) \stackrel{\cong}\longrightarrow H^i(X, A(n)) \oplus H^i( Y,A(n)) .
$$
\end{lemma}
\begin{proof}
Let $i_X$ (resp.\ $i_Y$) denote the inclusions of $X$ (resp.\ $Y$) into $X\sqcup Y$.
By the Gysin sequence (\ref{ax:Gysin}), we have an exact sequence
$$
H^i(X, A(n))\stackrel{i_{X\ast}}\longrightarrow H^i(X\sqcup Y,A(n)) \stackrel{i_Y^\ast} \longrightarrow H^i( Y,A(n)) .
$$
Functoriality of this sequence with respect to proper pushforward and pullbacks along morphisms in $\VV$ shows that $i_{X\ast}$ and $i_Y^\ast$ admit splittings.
Hence, the above sequence is part of a short exact sequence that splits, which proves the lemma.
\end{proof}

\begin{lemma} \label{lem:les-general}
Let $X\in \VV$ and  $A\in \mathcal A$. 
Then for any $n\in \Z$ and $m,j\geq 0$, the Gysin sequence  in (\ref{ax:Gysin}) induces a long exact sequence
\begin{align*}
\dots \longrightarrow   
 H^i( F_{j+m}X,A(n)) \longrightarrow   H^i(F_{j-1}X,A(n)) \stackrel{\del}\longrightarrow   \lim_{\substack{\longrightarrow \\ Z\subset X\\
\codim(Z)=j}}  H^{i+1-2j}( F_mZ,A(n-j)) \stackrel{\iota_\ast} \longrightarrow   \dots ,
\end{align*}
where the direct limit runs through all closed reduced subschemes $Z\subset X$  of codimension $\codim(Z)=\dim X-\dim Z=j$.
\end{lemma}
\begin{proof}
This follows immediately from  (\ref{ax:Gysin}) by taking direct limits.
We explain the details for convenience of the reader.
Let $Z\subset X$ be closed with $\dim Z=\dim X-j$.
Let further $W\subset Z$ be closed of dimension $\dim W=\dim Z-m-1=\dim X-j-m-1$. 
By (\ref{ax:Gysin}), we get an exact sequence
\begin{align*}
\dots \longrightarrow   
 H^i( X\setminus W,A(n)) \longrightarrow   H^i(X\setminus Z,A(n)) \stackrel{\del}\longrightarrow H^{i+1-2j}( Z\setminus W,A(n-j)) \stackrel{\iota_\ast} \longrightarrow   \dots .
\end{align*}
We can now consider the index set $I$ that consists of pairs $(Z,W)$ of closed subsets  $W\subset Z\subset X$ with $\dim W=\dim Z-m-1=\dim X-j-m-1$. 
This is a directed set with respect to the preorder given by declaring $(Z,W)\leq (Z',W')$ if and only if $Z\subset Z'$ and $W\subset W'$.
Taking the direct limit over this index set, the above long exact sequence yields the sequence
\begin{align*}
\dots \longrightarrow   
 H^i( F_{j+m}X,A(n)) \longrightarrow   H^i(F_{j-1}X,A(n)) \stackrel{\del}\longrightarrow   \lim_{\substack{\longrightarrow \\  Z\subset X\\
\codim(Z)=j}}  H^{i+1-2j}( F_mZ,A(n-j)) \stackrel{\iota_\ast} \longrightarrow   \dots ,  
\end{align*}
which is exact because the direct limit functor is exact.
This proves the lemma. 
\end{proof}

\begin{lemma} \label{lem:les}
Let $X\in \VV$ and $A\in \mathcal A$.
Then for any $j,n\in \Z$, (\ref{ax:Gysin}) induces a long exact sequence
\begin{align*}
 \longrightarrow
   H^i(F_jX,A(n)) \longrightarrow   H^i(F_{j-1}X,A(n))  \stackrel{\del}\longrightarrow \bigoplus_{x\in X^{(j)}}H^{i+1-2j}(x,A(n-j)) \stackrel{\iota_\ast} \longrightarrow    H^{i+1}(F_jX,A(n))  ,
\end{align*}
where $\iota_\ast$ (resp.\ $\del$) is induced by the pushforward (resp.\ residue) map from the Gysin exact sequence (\ref{ax:Gysin}).
\end{lemma}
\begin{proof} 
Using additivity from Lemma \ref{lem:XsqcupY}, this identifies to the special case $m=0$ in Lemma \ref{lem:les-general}.
\end{proof}

\begin{corollary} \label{cor:restr-H_nr}
Let $X\in \VV$.
Then for any $n\in \Z$ and $j,m\geq 0$, the following sequence is exact
$$
 \lim_{\substack{\longrightarrow \\ Z\subset X\\
\codim(Z)=j}}  H^{i-2j}_{m,nr}(Z,A(n-j)) \stackrel{\iota_\ast} \longrightarrow
 H^i_{j+m,nr}( X,A(n)) \longrightarrow   H^i_{j-1,nr}(X,A(n)) ,
$$
where the direct limit runs through all closed reduced subschemes $Z\subset X$  of codimension $\codim(Z)=\dim X-\dim Z=j$.
\end{corollary}
\begin{proof}
The composition of the two arrows in the corollary is zero by Lemma \ref{lem:les-general}.
Conversely, assume that $\alpha\in  H^i_{j+m,nr}( X,A(n)) $ maps to zero in $H^i_{j-1,nr}(X,A(n))$.
By Lemma \ref{lem:les-general}, $\alpha=\iota_\ast \xi$ for some $\xi\in H^{i-2j} (F_mZ,A(n-j))$ and some $Z\subset X$ of codimension $j$.
Since $\alpha$ is unramified, Lemma \ref{lem:les} shows that 
\begin{align} \label{eq:iota-del}
\iota_\ast (\del \xi)=\del (\iota_\ast \xi)=0\in \bigoplus_{x\in X^{(j+m+1)}}H^{i-2j-2m-1}(x,A(n-j-m-1)) ,
\end{align}
where the first equality uses that the Gysin sequence is functorial with respect to proper pushforwards (see (\ref{ax:Gysin})), so that $\iota_\ast$ and $\del$ commute. 
But this implies that the class
$$
\del \xi\in \bigoplus_{x\in Z^{(m+1)}}H^{i-2j-2m-1}(x,A(n-j-m-1))
$$
vanishes, as the above right hand side is a subgroup of the right hand side of (\ref{eq:iota-del}), and $\iota_\ast$ identifies to the inclusion.
Hence,  Lemma \ref{lem:les} implies $\xi\in H^{i-2j}_{m,nr} (Z,A(n-j)) $, as we want.
This concludes the proof of the corollary.
\end{proof}

\begin{corollary} \label{cor:F^i} 
Let $X\in \VV$ and $A\in \mathcal A$. 
 Then $H^i(F_jX,A(n))\cong H^i(X,A(n))$ for all $j\geq \lceil i/2\rceil$.
 \end{corollary}
 \begin{proof}
Since $H^i(x,A(n))$ vanishes for $i<0$ by  (\ref{ax:normalization}), Lemma \ref{lem:les} implies
$$
H^i(F_jX,A(n))\cong H^i(F_{j-1}X,A(n))
$$
for all $j$ with  $j>\lceil i/2\rceil$.
This proves   the corollary by induction on $j$, because 
$
H^i(F_jX,A(n))=H^i(X,A(n))
$ 
for $j\geq \dim(X)$. 
\end{proof}
 
 \begin{corollary}  \label{cor:F^i-2} 
Let $X\in \VV$ and $A\in \mathcal A$. 
Assume that there is a non-negative integer $c$, such that 
 for any $X\in \VV$ and $x\in X_{(j)}$, $H^i(x,A(n))=0$ for $i>j+c$ and all $n$.
 Then $H^i(F_jX,A(n))=0$ for all $i>\dim X+j+c$ and all $n$.  
\end{corollary}
\begin{proof}
Our assumption implies by Lemma \ref{lem:les} that $H^i(F_jX,A(n))\cong H^i(F_{j-1}X,A(n)) $ for all $j$ with $i>j+\dim X+c$.
Hence, $H^i(F_jX,A(n))\cong H^i(F_{0}X,A(n))$ for all $j$ with $i>j+\dim X+c$.
But $H^i(F_{0}X,A(n))=0$ for all $i>\dim X+c$ by Lemma \ref{lem:XsqcupY} and our assumption, because $F_0X$ is the union of the generic points of the maximal-dimensional components of $X$.   
This proves the corollary.
\end{proof}

The following lemma identifies the differential $d_1$ on the $E_1$-page of the coniveau spectral sequence of Bloch--Ogus \cite[\S 3]{BO} with the composition $\del\circ \iota_\ast$.

\begin{lemma}\label{lem:del=delta}
Let $X\in \VV$ and $A \in \mathcal A$.
Let $w\in X^{(p-1)}$ with closure $W\subset X$ and let $\tau:\tilde W\to W$ be the normalization with generic point $\eta_{\tilde W}\in \tilde W$.
Then the following diagram commutes for all integers $i$ and $n$  
$$
\xymatrix{
H^i(w,A(n))=H^i( \eta_{\tilde W} ,A(n))\ar[r]^-{\del}\ar@{^{(}->}[d]&\bigoplus_{\tilde w\in \widetilde W^{(1)}}H^{i-1}( \tilde w ,A(n-1)) \ar[d]^{\tau_\ast}\\
\bigoplus_{x\in X^{(p-1)}} H^i(x,A(n)) \ar[r]^-{\del \circ \iota_\ast}& \bigoplus_{x\in X^{(p)}}H^{i-1}(x,A(n-1)) ,
}
$$
where the  vertical arrow on the left is the natural inclusion, the vertical arrow on the right is induced by the proper pushforward maps from (\ref{ax:pushforward}), the upper horizontal arrow is induced by the residue map in (\ref{ax:Gysin}) and the lower horizontal arrow is given by 
$$
\bigoplus_{x\in X^{(p-1)}} H^i(x,A(n))\stackrel{\iota_\ast}\longrightarrow  H^{i+2p-2}(F_{p-1}X,A(n+p))\stackrel{\del}\longrightarrow \bigoplus_{x\in X^{(p)}}H^{i-1}(x,A(n-1)),
$$
where $\iota_\ast$ resp.\ $\del$ is the pushforward resp.\ residue map induced by  (\ref{ax:Gysin}).
\end{lemma}
\begin{proof}
Note that $W\in \VV$ and hence $\widetilde W\in \VV$, cf.\ Definition \ref{def:V}.
The lemma is thus a direct consequence of the functoriality of the Gysin sequence (\ref{ax:Gysin}) with respect to proper pushforwards (\ref{ax:pushforward}), as required in (\ref{ax:Gysin}).
\end{proof}

\subsection{Torsion-freeness of the cohomology of  points}

In this section we fix a prime $\ell$ and assume that the twisted Borel--Moore cohomology theory $H^\ast(-,A(n))$ on $\VV$ is $\ell$-adic, see Definition \ref{def:Borel--Moore-cohomology-l-adic}.
It is an observation of Bloch (see \cite[end of Lecture 5]{bloch}) that properties (\ref{ax:Bockstein}) and (\ref{ax:Hilbert90}) have the following important consequence.

\begin{lemma} \label{lem:torsionfree} 
Let $\VV$ be a constructible category of Noetherian schemes as in Definition \ref{def:V}.
Fix a prime $\ell$ and assume that  $H^\ast(-,A(n))$ is an $\ell$-adic twisted Borel--Moore cohomology theory on $\VV$ as in  Definition \ref{def:Borel--Moore-cohomology-l-adic}. 
Then for any  $X\in \VV$ and $x\in X$, $ H^i(x,\Z_\ell(i-1))$ is torsion-free for $1\leq i\leq 2$.
\end{lemma}
\begin{proof} 
Taking direct limits of abelian groups is exact, so that property (\ref{ax:Bockstein}) implies that 
$$
H^{i}(x,\Z_\ell(i-1))[\ell^r]\cong \coker(H^{i-1}(x,\Z_\ell(i-1))\longrightarrow H^{i-1}(x,\mu_{\ell^r}^{\otimes i-1})).
$$
This vanishes for $i=1$, as in this case we have by  (\ref{ax:Bockstein}) an exact sequence
$$
H^{0}(x, \Z_\ell(0)) \stackrel{\times \ell^r}\longrightarrow H^{0}(x, \Z_\ell(0)) \longrightarrow H^{0}(x, \mu_{\ell^r}^{\otimes 0}) 
$$
which by (\ref{ax:normalization}) identifies to $\Z_\ell \stackrel{\times \ell^r}\to \Z_\ell\to \Z/\ell$ and so the last arrow is surjective.  

By (\ref{ax:Hilbert90}), there is a surjection $\overline \epsilon:\kappa(x)^\ast\twoheadrightarrow H^{1}(x,\mu_{\ell^r}^{\otimes 1})$ which factors through $ H^{1}(x,\Z_\ell(1))$ and so the above cokernel also vanishes for $i=2$. 
This concludes the proof.
\end{proof}


\begin{remark}\label{rem:bloch-kato}
The above proof shows more generally that $H^{i+1}(x,\Z_\ell(i))$ is torsion-free if there are surjections $(\kappa(x)^\ast)^{\otimes i}\twoheadrightarrow H^{i}(x,\mu_{\ell^r}^{\otimes i})$ that factor through $H^{i}(x,\Z_{\ell}(i))$. 
In particular, $H^{i+1}(x,\Z_\ell(i))$ is torsion-free if a version of the Bloch--Kato conjecture holds in degree $i$ in the sense that there is a map $ K^M_i(\kappa(x))\to H^i(x,\Z_\ell(i))$ which induces isomorphisms  $K^M_i(\kappa(x))/\ell^r\cong H^i(x,\mu_{\ell^r}^{\otimes i})$. 
It follows from Voevodsky's proof of the Bloch--Kato conjecture \cite{Voe:Bloch-Kato} that the theories that we discuss in Proposition \ref{prop:proetale-coho-arbitrary-field} and \ref{prop:Betti-coho} below have this property.
\end{remark}

 \section{Examples of Borel--Moore cohomologies}
 
 In this section we discuss some examples of  functors that satisfy the properties from Section \ref{sec:axioms}.
 The results are certainly well-known to experts and we only include them for convenience of the reader.

\subsection{\texorpdfstring{$\ell$}{ell}-adic Borel--Moore pro-\'etale cohomology} \label{sec:proet} 

\subsubsection{Continuous \'etale cohomology of Jannsen} \label{subsec:jannsen}
Let $X$ be a scheme over a field $k$ and let $\Ab(X_{\et})^{\N}$ be the abelian category of inverse systems of abelian \'etale sheaves on the small \'etale site $X_{\et}$ of $X$.
This category has enough injectives (see \cite{jannsen}) and we may consider the left exact functor
$$
\varprojlim\circ \Gamma:\Ab(X_{\et})^{\N}\longrightarrow \Ab,\ \ (F_r)\mapsto \lim_{\substack{\longleftarrow\\ r}} \Gamma(X,F_r).
$$
Jannsen then defines the continuous \'etale cohomology groups 
$$
H^i_{cont}(X,(F_r)):=\RR^i (\varprojlim\circ \Gamma)((F_r)) .
$$ 
These groups are closely related to the corresponding \'etale cohomology groups via the following canonical short exact sequence (see \cite[\S 1.6]{jannsen}):
\begin{align} \label{eq:ses-lim-Jannsen}
0\longrightarrow \RR^1 \varprojlim  \ H^{i-1}(X_{\et},F_r)\longrightarrow H^i_{cont}(X_{\et},(F_r))\longrightarrow \varprojlim H^i(X_{\et},F_r)\longrightarrow 0 ,
\end{align}
where $\varprojlim$ denotes the inverse limit functor over $r$.

By  \cite[(3.27)]{jannsen}, we have the   following Kummer exact sequence in $\Ab(X_\et)^{\N}$:
\begin{align} \label{eq:kummer}
0\longrightarrow (\mu_{\ell^r})_r \longrightarrow (\mathbb G_m,\times \ell)_r \stackrel{\times \ell^r}\longrightarrow  (\mathbb G_m,\id)_r \longrightarrow 0 ,
\end{align}
where $\ell$ is a prime invertible in $k$.
Taking cohomology, the boundary map of the corresponding long exact sequence yields maps
\begin{align} \label{eq:epsilon-c1}
\epsilon:H^0(X,\mathbb G_m)\longrightarrow H^1_{cont}(X,\Z_\ell(1))\ \ \ \text{and}\ \ \ c_1:\Pic(X)\longrightarrow H^2_{cont}(X,\Z_\ell(1)),
\end{align}
where $H^i_{cont}(X,\Z_\ell(n)):=H^i_{cont}(X,(\mu_{\ell^r}^{\otimes n})_r)$.

\subsubsection{Pro-\'etale cohomology of Bhatt and Scholze}
For a scheme $X$ we denote by $X_\proet$ the pro-\'etale site of $X$ formed by weakly \'etale maps of schemes $U\to X$ (with $U$ of not too big cardinality), see \cite[Definition 4.1.1 and Remark 4.1.2]{BS}. 
Since every \'etale map is  weakly \'etale, there is a natural map of associated topoi 
\begin{align} \label{def:nu}
\nu:\Sh(X_{\proet})\longrightarrow \Sh(X_\et).
\end{align}
The  pullback $\nu^\ast:D^+(X_\et)\to D^+(X_\proet)$ on bounded below derived categories is fully faithful and the adjunction $\id\to \RR \nu_\ast \nu^\ast$ is an isomorphism, see \cite[Proposition 5.2.6]{BS}.
For a sheaf $F\in \Ab(X_\proet)$ of abelian groups on $X_\proet$, one defines
$$
H^i(X_\proet,F):=\RR^i \Gamma(X_\proet , F),
$$
where $\RR^i \Gamma$ denotes the $i$-th right derived functor of the global section functor $ F\mapsto \Gamma(X,F)$. 

If the transition maps in the inverse system $(F_r)\in \Ab(X_{\et})^{\N}$ are surjective, then there is a canonical isomorphism
\begin{align}\label{eq:H^i_cont=H^i_proet}
 H^i(X_\proet,\lim \nu^\ast F_r) \cong H^i_{cont}(X_{\et},(F_r)) ,
\end{align}
see   \cite[\S 5.6]{BS}.

\subsubsection{Constructible complexes in the pro-\'etale topology}\label{subsec:constructible-complex}
We present in this section some parts of the six functor formalism on constructible complexes of Bhatt and Scholze in the special case of algebraic schemes, i.e.\ separated schemes of finite type over a field, which suffices for our purposes.
In  Remark \ref{rem:BS-more-general} below we add some comments on the more general setting from \cite{BS}. 

Let $X$ be an algebraic scheme over a field $k$ and recall $\nu$ from (\ref{def:nu}).
For a prime $\ell$ invertible in $k$, let
\begin{align} \label{def:Zhat}
\widehat \Z_\ell(n):=\lim \nu^\ast \mu_{\ell^r}^{\otimes n} \in \Ab(X_\proet)
\end{align}
and write $\widehat \Z_\ell:=\widehat \Z_\ell(0)$.
Note that $ \widehat \Z_\ell$  is a sheaf of rings on $X_\proet$ and $\widehat \Z_\ell(n)$ are $\widehat \Z_\ell$-modules, which are in fact locally free (e.g.\ they are free on the pro-\'etale covering $X_{\overline k}\to X$).
We may then consider the derived category $D(X_\proet,\widehat \Z_\ell)$ of  the abelian category $\Mod(X_\proet,\widehat \Z_\ell)$ of sheaves of $\widehat \Z_\ell$-modules on $X_\proet$.
A complex $K\in D(X_\proet,\widehat \Z_\ell)$ is constructible, if it is complete, i.e.\ $K\stackrel{\cong}\to \RR \lim (K\otimes^{\mathbb L}_{\widehat \Z_\ell} \Z/\ell^r)$, and $K\otimes^{\mathbb L}_{\widehat \Z_\ell} \Z/\ell^r\cong \nu^\ast K_r$ for a constructible complex $K_r\in D(X_\et,\Z/\ell^r)$,  see \cite[Definition 6.5.1]{BS}.
The full subcategory spanned by constructible complexes is denoted by $D_{cons}(X_\proet,\widehat \Z_\ell)\subset D(X_\proet,\widehat \Z_\ell)$. 
Constructible complexes are bounded, see \cite[Lemma 6.5.3]{BS}.

For a morphism $f:X\to Y$ of algebraic schemes,  $\RR f_\ast$ respects constructibility and is right adjoint to $f^\ast_{comp}:D_{cons}(Y_\proet,\widehat \Z_\ell)\to D_{cons}(X_\proet,\widehat \Z_\ell)$, which is given by pullback followed by (derived) completion, see \cite[Lemma 6.7.2]{BS}. 
There is also a functor $\RR f_!:D_{cons}(X_\proet,\widehat \Z_\ell)\to D_{cons}(Y_\proet,\widehat \Z_\ell)$ (see \cite[Definition 6.7.6]{BS}) with a right adjoint $f^!: D_{cons}(Y_\proet,\widehat \Z_\ell )\to D_{cons}( X_\proet,\widehat \Z_\ell)$, see \cite[Lemma 6.7.19]{BS}.
If $f$ is proper, $\RR f_!=\RR f_\ast$ (by definition).

To explain the construction of $f^!$ in \cite{BS}, note that the pullback
\begin{align}\label{eq:nu*}
\nu^\ast: D_{cons}(X_\et,\Z/\ell^r)\stackrel{\cong}\longrightarrow D_{cons}(X_\proet,\Z/\ell^r)
\end{align}
is an equivalence (see paragraph after \cite[Definition 6.5.1]{BS}).
Using this, we will freely identify complexes on the two sides with each other.
For instance, we will freely identify $\mu_{\ell^r}^{\otimes n}$ on $Y_{\et}$ with its pullback $\nu^\ast \mu_{\ell^r}^{\otimes n}$ to $Y_\proet$.
Let now $K\in D_{cons}( Y_\proet,\widehat \Z_\ell)$ with truncation $K_r=K\otimes^{\mathbb L}_{ \Z_\ell}\Z/\ell^r$ and let $f_r^!:D_{cons}(Y_\et, \Z/\ell^r )\to D_{cons}( X_\et,  \Z/\ell^r)$ be the exceptional pullback on the \'etale site, induced by $f$, cf.\ \cite[Expos\'e XVIII]{SGA4.3}.
Since any constructible complex of sheaves of $\Z/\ell^r$-modules on $X_\proet$ is also a constructible complex of $\widehat \Z_\ell$-modules on $X_\proet$, we may by (\ref{eq:nu*}) identify $f_r^!K_r$ with an object in $D_{cons}( X_\proet,\widehat \Z_\ell)$.
By \cite[Lemma 6.7.18]{BS}, the natural reduction maps $K_r\to K_m$ for $m\leq r$ make $(f_r^!K_r)$ into a projective system and so, following Bhatt--Scholze (see \cite[Lemma 6.7.19]{BS}), one may define
\begin{align} \label{eq:f!K}
f^!K:=\RR \lim f_r^! K_r \in D_{cons}( X_\proet,\widehat \Z_\ell).
\end{align}
The above construction implies that many properties known from the \'etale site carry over to the pro-\'etale site. 

\begin{lemma} \label{lem:functors-proet}
Let $f:X\to Y$ be a morphism between algebraic schemes over a field $k$ and let $\ell$ be a prime invertible in $k$.
Then the following holds in $D_{cons}(X_\proet,\widehat \Z_\ell)$:
\begin{enumerate}
\item If $f$ is weakly \'etale or a closed immersion, then $ f_{comp}^\ast\cong f^\ast$; \label{item:f*=f-comp*}
\item If $f$ is \'etale, then $f^!\cong f^\ast \cong f^\ast_{comp}$; \label{item:f-etale}
\item If $g:Y\to Z$ is another morphism, then there is a natural isomorphism of functors $f^!g^!\stackrel{\cong}\longrightarrow (g\circ f)^! $;\label{item:f!g!=gf!} 
\item If $f$ is smooth of pure relative dimension $d$,  then 
there is a canonical isomorphism of functors $f_{comp}^\ast(d)[2d]\stackrel{\cong}\longrightarrow f^!$, where $f_{comp}^\ast(n):=f_{comp}^\ast(-\otimes_{\widehat \Z_\ell} \widehat \Z_{\ell}(n))$.\label{item:f!=f*(d)} 
\item  \label{item:f!=f*(d)-compatible} 
Let $f$ be smooth of pure relative dimension $d$.
Then for any \'etale map $j:U\to X$, the diagram 
$$
\xymatrix{
 (f\circ j)_{comp}^\ast(d)[2d] \ar[r]^-{\cong}&
(f\circ j)^!\\ 
 j_{comp}^\ast f_{comp}^\ast(d)[2d] \ar[r]^-{\cong} \ar[u]  &
  j_{comp}^\ast f^! \cong j^!f^! \ar[u]\\ 
}
$$
commutes, where the horizontal maps are induced by the canonical isomorphisms from  item (\ref{item:f!=f*(d)}) and the vertical arrows are induced by the canonical maps given by functoriality of $f^\ast_{comp}$ and $f^!$. 
\end{enumerate}
\end{lemma}
\begin{proof}
Item (\ref{item:f*=f-comp*}) follows from \cite[Remark 6.5.10]{BS}.
Item (\ref{item:f-etale}) follows from this together with the fact that for any $K\in D_{cons}(Y_\proet,\widehat \Z_\ell)$,
$$
f_{comp}^\ast K \stackrel{\cong}\longrightarrow \RR \lim f_r^\ast K_r=\RR \lim f_r^!K_r=
f^!K
$$
because $f_r^!=f_r^\ast$ since $f$ is \'etale, see \cite[XVIII, Proposition 3.1.8(iii)]{SGA4.3}. 

Let $K\in D_{cons}(Z_\proet,\widehat \Z_\ell)$ with truncations $K_r=K\otimes_{\widehat \Z_\ell}^{\mathbb L} \Z/\ell^r \in D_{cons}(Y_\et,\Z/\ell^r)$.
By  (\ref{eq:nu*}) and the construction of $f^!$ from (\ref{eq:f!K}),  there is a natural map
$$
f^!g^!K=\RR \lim f_r^!(\RR \lim g_r^! K_r)=\RR \lim (f_r^!\circ g_r^!(K_r))\stackrel{\cong}\longrightarrow \RR \lim ((g_r\circ f_r)^!K_r)=(g\circ f)^! K
$$
induced by the natural isomorphism $f_r^!\circ g_r^!\stackrel{\sim} \to (g_r\circ f_r)^!$, given by adjunction and $\RR (g_r)_!\RR (f_r)_!\stackrel{\cong}\to \RR (g_r\circ f_r)_!$. 
This proves (\ref{item:f!g!=gf!}). 

Let $K\in D_{cons}(Y_\proet ,\widehat \Z_\ell)$ with truncations $K_r=K\otimes_{\widehat \Z_\ell}^{\mathbb L} \Z/\ell^r \in D_{cons}(Y_\et,\Z/\ell^r)$.
Assume that $f$ is smooth of pure relative dimension $d$.
By Poincar\'e duality on the \'etale site,  there are canonical identifications $f_r^!= f_r^\ast(d)[2d]$, see \cite[XVIII, Th\'eor\`eme 3.2.5]{SGA4.3} (cf.\ \cite[\S 4.4]{verdier}). 
We thus get a canonical isomorphism
$$ 
f_{comp}^\ast K(d)[2d]\stackrel{\cong}\longrightarrow \RR \lim f_r^\ast K_r(d)[2d]=\RR \lim f_r^! K_r= f^!K.
$$ 
This holds functorially in $K$ and so we get an isomorphism $f_{comp}^\ast(d)[2d]\stackrel{\cong}\to f^!$, which proves (\ref{item:f!=f*(d)}). 

By item (\ref{item:f-etale}), $j_{comp}^\ast\cong j^!$ and so the commutativity of the diagram in item (\ref{item:f!=f*(d)-compatible}) follows from the fact that the isomorphism in (\ref{item:f!=f*(d)}) is compatible with respect to compositions of smooth maps.
The latter follows by construction of $f^!$ from the analogous result for constructible complexes on the \'etale site and hence from \cite[XVIII, diagram above Th\'eor\`eme 3.2.5]{SGA4.3}.
This concludes the proof of the lemma.
\end{proof}

Let $f:X\to Y$ be a morphism between algebraic $k$-schemes.
By adjunction, there are natural transformations
\begin{align} \label{eq:adjunction}
\Tr_f: \RR f_! f^!\longrightarrow \id \ \ \ \text{and}\ \ \ 
\theta_f: \id\longrightarrow \RR f_\ast f^\ast_{comp} 
\end{align}
between functors on $D_{cons}(Y_\proet,\widehat \Z_\ell)$.
For $K\in D_{cons}(Y_\proet,\widehat \Z_\ell)$, the above maps are defined by asking  (see \cite[Lemmas 6.7.2 and 6.7.19]{BS})   that the following diagrams commute:
\begin{align} \label{eq:adjunction-construction}
\xymatrix{
K \ar[d]\ar[rr]^-{\theta_f}& &\RR f_\ast f^\ast_{comp}K \ar[d]\\
 \RR \lim K_r\ar[rr]^-{\RR \lim(\theta_{f_r})}&&  \RR \lim (\RR f_{r\ast} f_r^\ast K_r) ,
}\ \ \ \ 
\xymatrix{
\RR f_!f^!K \ar[d]\ar[rr]^-{\Tr_f}& & K \ar[d]\\
\RR \lim  \RR f_{r!}f_r^!K_r\ar[rr]^-{\RR \lim(\Tr_{f_r})}&&  \RR \lim (  K_r) ,
}
\end{align}
where $K_r=K\otimes^{\mathbb L}_{\widehat \Z_\ell}\Z/\ell^r\in  D_{cons}(Y_\et, \Z/\ell^r)$, $f_r^\bullet$, $f_{r\bullet}$ denote the corresponding functors on $D_{cons}(Y_\et, \Z/\ell^r)$ (resp.\ $D_{cons}(X_\et, \Z/\ell^r)$) induced by $f$, and where as before  we identify $K_r$ with $\nu^\ast K_r$, using the equivalence (\ref{eq:nu*}).

For $K\in D_{cons}(Y_\proet,\widehat \Z_\ell)$, $\theta_f$ induces pullback maps
\begin{align} \label{eq:proet-pullback}
f^\ast:\RR^i \Gamma(Y_\proet, K)\longrightarrow \RR^i \Gamma(X_\proet, f_{comp}^\ast K).
\end{align}
If $f$ is proper, then $\RR f_!=\RR f_\ast$ and so $\Tr_f$ induces a pushforward map
\begin{align} \label{eq:proet-pushforward-complex}
f_\ast: \RR^i \Gamma(X_\proet, f^! K)\longrightarrow \RR^i \Gamma(Y_\proet, K).
\end{align} 

\begin{lemma} \label{lem:proet-Rlim-extension} 
Let $K\in D_{cons}(Y_\proet,\widehat \Z_\ell )$ with truncations $K_r=K\otimes^{\mathbb L}_{\widehat \Z_\ell}\Z/\ell^r\in  D_{cons}( Y_\et, \Z/\ell^r )$. 
Then for any morphism $f:X\to Y$ between algebraic $k$-schemes, the following diagram commutes with exact rows:
$$
\xymatrix{
0\ar[r]&
  \RR ^1\lim \RR^{i-1}\Gamma(X_{\et},f_r^\ast K_r)\ar[r]&  
  \RR ^i \Gamma(X_\proet ,f^\ast_{comp} K)\ar[r] & \lim \RR^{i} \Gamma(X_{\et},f _r^\ast K_r) \ar[r] & 0\\
0\ar[r]&  \RR ^1\lim \RR^{i-1}\Gamma (Y_{\et},K_r)\ar[r] \ar[u]^{\RR^1 \lim (f_r^\ast)}&  \RR^i \Gamma (Y_\proet , K)\ar[r] \ar[u]^{ f^\ast} & \ar[u]^{  \lim (f_r^\ast)} \lim \RR^{i} \Gamma (Y_{\et},K_r) \ar[r] & 0 .
}
$$
If  $f$ is proper, then the following commutes with exact rows as well:
$$
\xymatrix{
0\ar[r]&  \RR ^1\lim \RR^{i-1}\Gamma(X_{\et},f_r^! K_r)\ar[r] \ar[d]^{\RR^1 \lim ((f_r)_\ast)}&  
  \RR ^i \Gamma(X_\proet ,f^! K)\ar[r] \ar[d]^{ f_\ast} & \lim \RR^{i} \Gamma(X_{\et},f_r ^! K_r)\ar[d]^{  \lim ((f_r)_\ast)} \ar[r] & 0  \\
 0\ar[r]& \RR ^1\lim \RR^{i-1}\Gamma (Y_{\et},K_r)\ar[r] &  \RR^i \Gamma (Y_\proet , K)\ar[r] & \lim \RR^{i} \Gamma (Y_{\et},K_r)   \ar[r] & 0.
}
$$ 
\end{lemma}
\begin{proof}
The horizontal lines are parts of short exact sequences given by the composed functor spectral sequence of $\RR \lim\circ \RR \Gamma$, where we note that $\RR \lim$ has cohomological dimension $\leq 1$ on abelian groups.
The lemma follows thus immediately from the commutative diagrams in (\ref{eq:adjunction-construction}). 
\end{proof}

\begin{remark}\label{rem:BS-more-general}
Bhatt--Scholze's six functor formalism on constructible complexes of $\widehat \Z_\ell$-sheaves on the pro-\'etale site works more generally for quasi-excellent quasi-compact quasi-separated schemes over $\Z[1/\ell]$ and separated finitely presented maps between them, see \cite[\S 6.7]{BS}.
Lemmas \ref{lem:functors-proet} and \ref{lem:proet-Rlim-extension} remain true in this set-up (with the same proofs). 
\end{remark}

\begin{remark}
Related to Bhatt--Scholze's pro-\'etale theory \cite{BS}, there is Ekedahl's $\ell$-adic formalism \cite{Eke}, which leads to a six functor formalism (see \cite[Theorem 6.3]{Eke}) that is closer in spirit to Jannsen's continuous \'etale cohomology groups.   
We prefer to use Bhatt--Scholze's theory,  as it allows to work with actual sheaves on a site,  while Ekedahl's formalism as well as Jannsen's theory involve inverse systems of sheaves.
The resulting triangulated categories agree under suitable finiteness assumptions, see \cite[\S 5.5]{BS}.
\end{remark}
 
\subsubsection{Properties (\ref{ax:pushforward})--(\ref{ax:Hilbert90}),  (\ref{ax:homolog=alg-div}), and (\ref{ax:homolog=ratl-div})}
Let $k$ be a field and let $\ell$ be a prime that is invertible in $k$.
Let $\VV$ be the category whose objects are separated schemes of finite type over $k$ and where the morphisms are open immersions of schemes of the same dimension.
For $X\in \VV$ of dimension $d$ with structure morphism $\pi_X:X\to \Spec k$, we define 
\begin{align} \label{eq:H^i(X,mu-ell)}
H^i(X,\mu_{\ell^r}^{\otimes n}):= \RR^{i-2d} \Gamma(X_{\proet},\pi_X^! \mu_{\ell^r}^{\otimes n-d}) \in \Mod_{\Z_\ell},
\end{align} 
\begin{align}\label{eq:H^i(X,Z-ell)}
H^i(X,\Z_{\ell}(n)):=\RR^{i-2d} \Gamma(X_{\proet},\pi_X^!\widehat \Z_\ell(n-d)) \in \Mod_{\Z_\ell} , 
\end{align}
\begin{align}\label{eq:H^i(X,Q/Z-ell)}
H^i(X,\Q_\ell/\Z_{\ell}(n)):=\lim_{\substack{\longrightarrow \\ r}}H^i(X,\mu_{\ell^r}^{\otimes n}),\ \ \text{and}\ \ H^i(X,\Q_{\ell}(n)):=H^i(X,\Z_{\ell}(n))\otimes_{\Z_\ell} \Q_{\ell},
\end{align}
where $\widehat \Z_\ell(n)$ denotes the sheaf on $X_\proet$ defined in (\ref{def:Zhat}). 
By item (\ref{item:f-etale}) in Lemma \ref{lem:functors-proet} and (\ref{eq:proet-pullback}), the cohomology groups in (\ref{eq:H^i(X,mu-ell)})--(\ref{eq:H^i(X,Q/Z-ell)}) are contravariantly functorial with respect to morphisms in $\VV$ (in fact with respect to arbitrary \'etale maps $U\rightarrow X$  with $\dim U=\dim X$).

\begin{lemma}\label{lem:proet:coho=homo}
Assume that $X\in \VV$ is equi-dimensional and smooth over $k$.
Then there are  canonical isomorphisms
$$
H^i(X,\mu_{\ell^r}^{\otimes n})\cong  \RR^{i} \Gamma(X_{\proet},\mu_{\ell^r}^{\otimes n}) \cong H^i_{cont}(X_{\et},(\mu_{\ell^r}^{\otimes n},\id )_s)\cong H^i(X_{\et},\mu_{\ell^r}^{\otimes n}),
$$
and
$$
H^i(X,\Z_{\ell}(n))\cong  \RR^{i} \Gamma(X_{\proet},\widehat \Z_\ell(n)) \cong H^i_{cont}(X,(\mu_{\ell^r}^{\otimes n})_r) ,
$$ 
which are compatible with respect to pullbacks along open immersions.
\end{lemma}
\begin{proof} 
Since $\pi:X\to \Spec k$ is smooth of pure dimension $d$, there is a canonical isomorphism $(\pi_X)_{comp}^\ast(d)[2d]\stackrel{\cong}\to \pi_X^!$, see item  (\ref{item:f!=f*(d)})  in Lemma \ref{lem:functors-proet}.
This isomorphism is compatible with respect to pullbacks along open immersions by item  (\ref{item:f!=f*(d)-compatible})  in Lemma \ref{lem:functors-proet}.
This yields the first isomorphism in each row of the lemma.
The comparison to continuous \'etale cohomology follows from (\ref{eq:H^i_cont=H^i_proet}) and that to \'etale cohomology for finite coefficients from (\ref{eq:ses-lim-Jannsen}), because $\RR \lim^1$ vanishes on constant inductive systems.
This concludes the proof of the lemma.
\end{proof}
 
The main result of this section is the following. 
 
\begin{proposition}\label{prop:proetale-coho-arbitrary-field}
Let $k$ be a  field and let $\ell$ be a prime that is invertible in $k$.
Let $\mathcal A\subset \Mod_{\Z_\ell}$ be the full subcategory of $\Z_\ell$-modules containing  $\Z_\ell$,  $\Q_\ell$,  $\Q_\ell/\Z_\ell$, and $\Z/\ell^r$ for all $r\geq 1$. 
Let $\VV$ be the category of separated schemes of finite type over $k$ with morphisms given by open immersions $U\hookrightarrow X$ with $\dim U=\dim X$.
Let the cohomology functor (\ref{eq:functor}) be given by (\ref{eq:H^i(X,mu-ell)})--(\ref{eq:H^i(X,Q/Z-ell)}). 
Then  (\ref{ax:pushforward})--(\ref{ax:Bockstein}) from Definitions \ref{def:Borel--Moore-cohomology} and \ref{def:Borel--Moore-cohomology-l-adic} hold true.
If $k$ is perfect, then  (\ref{ax:Hilbert90}) holds true as well.
Moreover,
\begin{itemize}
\item property (\ref{ax:homolog=alg-div}) holds if $k$ is algebraically closed;
\item property (\ref{ax:homolog=ratl-div}) holds if $k$ is the perfect closure of a finitely generated field. 
\end{itemize}
\end{proposition}

In the terminology of Definitions \ref{def:Borel--Moore-cohomology}, \ref{def:Borel--Moore-cohomology-l-adic}, and \ref{def:adapted-alg/ratl-equivalence}, the above proposition says that $H^\ast(-,A(n))$ is a twisted Borel--Moore cohomology theory on $\VV$ that is $\ell$-adic if $k$ is perfect and it is adapted to algebraic equivalence if $k$ is algebraically closed,  while it is adapted to rational equivalence, if $k$ is the perfect closure of a finitely generated field. 

\begin{proof}[Proof of Proposition \ref{prop:proetale-coho-arbitrary-field}]
Item (\ref{ax:Ql/Zl,Ql}) is clear (by definition). Since the direct limit functor as well as $\otimes_{\Z_\ell}\Q_\ell$ is exact, it suffices to prove the remaining properties for $A=\Z/\ell^r$ and $A=\Z_\ell$.

\medskip
\noindent
\textbf{Step 1.} Item  (\ref{ax:pushforward}).
\smallskip

Let $X,Y\in \VV$ and let $f:X\to Y$ be  a proper morphism of schemes with $c=\dim Y-\dim X$.
The existence of the pushforward $f_\ast:H^{i-2c}(X,A(n-c))\to H^i(Y,A(n))$ follows from (\ref{eq:proet-pushforward-complex}) and item (\ref{item:f!g!=gf!}) in  Lemma \ref{lem:functors-proet}.
Functoriality in $f$ (i.e.\ $(f\circ g)_\ast=f_\ast g_\ast$) follows from the functoriality of the trace map (which by  (\ref{eq:adjunction})  may either be deduced from the corresponding statement on the \'etale site, or directly from item (\ref{item:f!g!=gf!}) in Lemma \ref{lem:functors-proet}). 
Compatibility of $f_\ast$ with pullbacks along open immersions may by Lemma \ref{lem:proet-Rlim-extension} be checked in the case where $A=\Z/\ell^r$  on the \'etale site of $X$, which is well-known (and holds in fact for arbitrary \'etale maps in place of open immersions), see \cite[(1.2.2) and \S 2.1]{BO}.
This proves (\ref{ax:pushforward}). 

\medskip
\noindent
\textbf{Step 2.} Item  (\ref{ax:Gysin}).
\smallskip

Let $X\in \VV$ and let $i:Z\hookrightarrow X$ be a closed immersion with complement $j:U\hookrightarrow X$.
Let $c=\dim X-\dim Z$. 
By \cite[Lemma 6.1.16]{BS}, there is an exact triangle
$$
\RR i_\ast i^! \pi_X^! \stackrel{\Tr_i}\longrightarrow \pi_X^!\stackrel{\theta_j}\longrightarrow \RR j_\ast j_{comp}^\ast \pi_X^!,
$$
where we used $j^\ast\cong j_{comp}^\ast$ (see Lemma \ref{lem:functors-proet}).
By Lemma \ref{lem:functors-proet}, $j_{comp}^\ast\cong j^! $, $\pi_Z^!\stackrel{\cong}\to  i^! \pi_X^! $ and $\pi_{U}^!\stackrel{\cong}\to j_{comp}^\ast\pi_X^! $.
Hence, the above triangle identifies to an exact triangle
$$
\RR i_\ast \pi_Z^!\longrightarrow \pi_X^!\longrightarrow \RR j_\ast \pi_U^! .
$$ 
Applying $\RR \Gamma(X_{\proet},-)$, the corresponding long exact sequence yields the Gysin sequence claimed in (\ref{ax:Gysin}).
The map $\iota_\ast$ in the Gysin sequence from item (\ref{ax:Gysin}) coincides by construction with the proper pushforward with respect to the inclusion $Z\hookrightarrow X$.
Functoriality with respect to open immersions follows by Lemma \ref{lem:proet-Rlim-extension} from the case $A=\Z/\ell^r$  on the \'etale site which is well-known (and holds in fact for arbitrary \'etale maps), see e.g.\ \cite[(1.1.2), Lemma 1.4, and (2.1)]{BO}.  
Similarly, functoriality with respect to proper pushforwards follows by Lemma \ref{lem:proet-Rlim-extension} from the case of finite coefficients $A=\Z/\ell^r$  on the \'etale site, which is well-known, see e.g.\ \cite[(1.2.4) and (2.1)]{BO}.
This proves  (\ref{ax:Gysin}).

\medskip
\noindent
\textbf{Step 3.} Item (\ref{ax:normalization}).
\smallskip

By the topological invariance of  the pro-\'etale topos (see \cite[Lemma 5.4.2]{BS}), we may replace $X$ be the base change to the perfect closure $k^{per}$ of $k$, and $x$ by the unique point in $X_{k^{per}}$ that lies over it via the natural map $X_{k^{per}}\to X$ (where we use that the latter is a universal homeomorphism).
After this reduction step, we may assume that $k$ is perfect.

Note that $H^i(x,A(n))$ in (\ref{eq:Hi(kappa(x))}) is defined as a direct limit where it suffices to run only through  the cohomology of regular (hence smooth, since $k$ is perfect) schemes, so that the vanishing $H^i(x,A(n))=0$ for $i<0$ as well as the canonical isomorphism $H^0(x,A(0))\cong A$ which is functorial in $A$ follows from Lemma \ref{lem:proet:coho=homo}. 
The fundamental class $[x]\in H^0(x,\Z_\ell(0))$ corresponds via the canonical isomorphism $H^0(x,\Z_\ell(0))\cong \Z_\ell$ to $1\in \Z_\ell$.
More precisely,  let  $U\subset \overline{\{x\}}$ be dense and smooth over $k$.
By Lemma \ref{lem:proet:coho=homo}, the canonical isomorphism in item (\ref{item:f!=f*(d)}) of Lemma \ref{lem:functors-proet}  induces a canonical isomorphism
$$
H^0(U,\Z_\ell(0))\cong H^0(U_\proet,\widehat \Z_\ell(0))
$$
which is compatible with respect to restrictions to open subsets.
The class of $H^0(U,\Z_\ell(0))$ induced by the unit section of the pro-\'etale sheaf $\widehat \Z_\ell(0) $  yields a canonical fundamental class $[U]\in  H^0(U,\Z_\ell(0))$ with $H^0(U,\Z_\ell(0))=[U]\Z_\ell$.
This class  is compatible  with respect to restrictions to open subsets (see item (\ref{item:f!=f*(d)-compatible}) in Lemma \ref{lem:functors-proet}), hence induces a canonical class $[x]\in H^0(x,\Z_\ell(0))$ in the limit.
This proves  (\ref{ax:normalization}).  
 
 \medskip
\noindent
\textbf{Step  4.} Item (\ref{ax:Bockstein}).
\smallskip

There is a canonical short exact sequence
$$
0\longrightarrow \widehat \Z_\ell(n) \stackrel{\times \ell^r}\longrightarrow \widehat \Z_\ell(n)\longrightarrow \mu_{\ell^r}^{\otimes n}\longrightarrow 0 
$$
of sheaves on $(\Spec k)_\proet$.
Applying $\pi_X^!$, we arrive at the exact triangle
\begin{align}\label{ses:coeff}
\pi_X^! \widehat \Z_\ell(n) \stackrel{\times \ell^r}\longrightarrow \pi_X^! \widehat \Z_\ell(n)\longrightarrow \pi_X^! \mu_{\ell^r}^{\otimes n}  .
\end{align}
The Bockstein sequence in (\ref{ax:Bockstein}) is (up to some shifts) the long exact sequence associated to this triangle after applying $\RR \Gamma(X_\proet,-)$. 
Functoriality of the exceptional pullback $\pi_X^!$ shows that the Bockstein sequence is functorial with respect to pullbacks along morphisms in $\VV$ and with respect to proper pushforwards from (\ref{ax:pushforward}). 
This proves (\ref{ax:Bockstein}).

\medskip
\noindent
\textbf{Step 5.} Item (\ref{ax:Hilbert90}) for $k$ perfect.
\smallskip

Let $X\in \VV$ and let $x\in X^{(1)}$. 
In the direct limit (\ref{eq:Hi(kappa(x))}) that defines $H^i(x,A(n))$, we may restrict ourselves to regular (hence smooth, as $k$ is perfect) dense open subsets $V_x\subset \overline{\{x\}}$, so that Lemma \ref{lem:proet:coho=homo} identifies $H^i(V_x,A(n))$ canonically with continuous \'etale cohomology. 
The map $\epsilon:\kappa(x)^\ast\to H^1(x,\Z_\ell(1))$ is then induced by the Kummer sequence in continuous \'etale cohomology, see (\ref{eq:epsilon-c1}).
Surjectivity of the reduction $\overline \epsilon:\kappa(x)^\ast\to H^1(x,\mu_{\ell^r})$ follows from the Kummer sequence in \'etale cohomology and  
 Grothendieck's Hilbert theorem 90, which implies that $$
\lim_{\substack{\longrightarrow \\ F_0X\subset U\subset X}} H^1_{\text{\'et}}(U,\mathbb G_m)\cong \lim_{\substack{\longrightarrow \\ F_0X\subset U\subset X}} \Pic(U)=0 .
 $$ 
 
 Let $X\in \VV$ be integral of dimension $d$ with generic point $\eta$ and let $x\in X^{(1)}$ be a regular point with closure $i:D\hookrightarrow X$.
 We claim that the composition $\del\circ   \epsilon: \kappa(\eta)^\ast \to  H^0(x ,\Z_{\ell}(0))=[x] \Z_{\ell}$ 
satisfies 
$$
\del\circ   \epsilon(f)=[x](- \nu_x(f)), 
$$
where  $\nu_x$ denotes the valuation on $\kappa(\eta)$ induced by $x$.
It suffices to check this after reduction modulo $\ell^r$ for $r\geq 1$.
It follows from \cite[p.\ 147, (cycle), Lemme 2.3.6]{SGA4.5} 
that the fundamental class $[x]\in H^0(x,\mu_{\ell^r}^{\otimes 0})$ that we defined above via Poincar\'e duality (i.e.\ via item (\ref{item:f!=f*(d)}) in Lemma \ref{lem:functors-proet}) is induced by the cycle class  
$$
\cl_{\ell^r}(D)\in H^2_D(X_\et,\mu_{\ell^r})=H^2(D_\et,i_r^!\mu_{\ell^r})\cong H^{2-2d}(D_\et,\pi_D^!\mu_{\ell^r}^{\otimes 1-d})
$$  
from \cite[p.\ 138, (cycle), D\'efinition 2.1.2]{SGA4.5}.
The claim in question follows therefore from the anticommutativity of the diagram in \cite[p.\ 138, (cycle), (2.1.3)]{SGA4.5}.
This concludes the proof of (\ref{ax:Hilbert90}).

%

\medskip

By Lemma \ref{lem:les} and Corollary \ref{cor:F^i} (which apply because we have proven (\ref{ax:pushforward})--(\ref{ax:Hilbert90}) already), the proper pushforward map from (\ref{ax:pushforward}) together with $H^0(x,\Z_\ell(0))=[x]\Z_\ell$ from (\ref{ax:normalization}) yields for any $X\in \VV$ a canonical map
$$
\iota_\ast : \bigoplus_{x\in X^{(1)}}[x]\Z_\ell \longrightarrow H^2(F_1X,\Z_\ell(1))\cong H^2(X,\Z_\ell(1)) ,
$$
as claimed in (\ref{eq:iota_ast}).
If $X$ is a smooth variety, then there is a canonical isomorphism $H^2(X,\Z_\ell(1)) \cong H^2_{cont}(X,\Z_\ell(1)) $ (see Lemma \ref{lem:proet:coho=homo}) and so we may compare $\iota_\ast$ to the first Chern class map $c_1$ from (\ref{eq:epsilon-c1}), as follows.

\begin{lemma} \label{lem:c_1}
Let $k$ be a field and let $X$ be a smooth $k$-variety. 
 For any Weil divisor $D\in \bigoplus_{x\in X^{(1)}} [x]\Z$, we have
\begin{enumerate}
\item
$
\iota_\ast D=c_1(\mathcal O_X(D)) ,
$
where $\iota_\ast$ is the cycle class map from (\ref{eq:iota_ast}) and $c_1$ is from (\ref{eq:epsilon-c1}).
\item $c_1(\mathcal O_X(D)) =0$ if and only if $\mathcal O_X(D)\in \Pic(X)$ is contained in the subgroup of $\ell$-divisible elements of $\Pic(X)$. 
\end{enumerate}
\end{lemma}
\begin{proof}
The first assertion is {\cite[Lemma 3.26]{jannsen}}.
For the second assertion, note that the Kummer sequence (\ref{eq:kummer}) yields an exact sequence
$$
H^1(X_{\et},  (\mathbb G_m,\times \ell)_r )\longrightarrow \Pic(X)\stackrel{c_1}\longrightarrow H^2_{cont}(X,\Z_\ell(1)).
$$
Functoriality of the extension in (\ref{eq:ses-lim-Jannsen}) shows that the image of the first map above is given by subgroup of $\ell$-divisible elements of $\Pic(X)$ (cf.\ \cite[Remark 6.15]{jannsen}), which concludes the proof of the lemma.
\end{proof}
\medskip

\noindent
\textbf{Step 6.} Item (\ref{ax:homolog=alg-div}) for $k$ algebraically closed.
\smallskip

Let us now assume that $k$ is algebraically closed and let $X$ be a regular (hence smooth, as $k=\overline k$) variety over $k$.
Note that (\ref{ax:homolog=alg-div})   is well-known in the case where $X$ is smooth projective, and we will deduce the general case from this statement in what follows.
We denote by $\NS(X)=\Pic(X)/\sim_{\alg}$ the group of divisors modulo algebraic equivalence on $X$. 
Since $k$ is algebraically closed,  the subgroup of algebraically trivial divisors in $\Pic(X)$ is $\ell$-divisible.
By Lemma \ref{lem:c_1}, the first Chern class map from (\ref{eq:epsilon-c1}) descends to a map 
$
c_1:\NS(X)\otimes \Z_\ell\to H^2(X,\Z_\ell(1))
$ 
and,  again by Lemma \ref{lem:c_1},  it suffices to show that this is injective. 
Let $\overline X$ be a projective normal compactification of $X$ and let $\tau:\overline X'\to \overline X$ be an alteration (i.e. \ a projective generically finite morphism with regular source) of degree prime to $\ell$, which exists by  \cite[Expos\'e X, Th\'eor\`eme 2.1]{ILO}.
Since $k$ is algebraically closed,   $\overline X'$ is smooth.
Putting $X':=\tau^{-1}(X)$, we get a commutative diagram
$$
\xymatrix{
\NS(\overline X')\otimes \Z_\ell \ar[rr]^{c_1}\ar[d]^{\text{restr.}} &   &  H^2(\overline X',\Z_\ell(1))  \ar[d]^{\text{restr.}} \\
\NS( X') \otimes \Z_\ell\ar[d]^{(\tau|_{X'})_\ast}\ar[rr]^{c_1} & & H^2(X',\Z_\ell(1))\ar[d]^{(\tau|_{X'})_\ast} \\
\NS( X)\otimes \Z_\ell \ar[rr]^{c_1}  &   &  H^2( X,\Z_\ell(1)) .
}
$$
We claim that it suffices to show that the horizontal arrow in the middle is injective.
To see this, let $\alpha\in \NS( X)\otimes \Z_\ell $.
Then $c_1(\tau^\ast\alpha)=\tau^\ast c_1(\alpha)$ and so injectivity of the horizontal arrow in the middle implies $c_1(\alpha)\neq 0$ unless $\tau^\ast\alpha=0 $ which in turn implies $\tau_\ast \tau^\ast \alpha=\deg \tau\cdot \alpha=0$ and so $\alpha=0$ since $\deg \tau$ is coprime to $\ell$. 

By the localization sequence,  the kernel of $\NS(\overline X') \otimes \Z_\ell\to \NS(X')\otimes \Z_\ell$ is generated by classes of divisors supported on $\overline X'\setminus X$.
Similarly, the Gysin sequence (see (\ref{ax:Gysin})) shows that the kernel of the restriction map $ H^2(\overline X',\Z_\ell(1)) \to  H^2(X',\Z_\ell(1)) $ is generated by the cycle classes of these divisors.
Since the first horizontal map in the above diagram is injective by \cite[p.\ 216, V.3.28]{milne}, while the restriction map $\NS(\overline X')\to \NS( X') $ is surjective (see \cite[Proposition 1.8]{fulton}), this shows that the horizontal arrow in the middle of the above diagram is injective, as we want. This proves (\ref{ax:homolog=alg-div}).

\medskip
\noindent
\textbf{Step 7.} Item (\ref{ax:homolog=ratl-div}) for $k$ the perfect closure of a finitely generated field.
\smallskip

We will use the following well-known lemma.

\begin{lemma}\label{lem:Chow-inseparable-extension}
Let $X$ be a separated scheme of finite type over a field $k$ of characteristic $p>0$ and let $E/k$ be a purely inseparable extension. 
Then the flat pullback map $\CH^i(X)[1/p]\to \CH^i(X_E)[1/p]$ is an isomorphism.
\end{lemma}
\begin{proof}
The argument is well-known; we recall it for convenience.
By a standard limit argument, it suffices to treat the case where $E/k$ is a finite extension of degree $p^s$ for some $s$.
Let $f:X_E\to X$ be the canonical map.
Then $f_\ast\circ f^\ast=p^s\cdot \id$ and so $f^\ast$ is injective after inverting $p$. 
Since $f$ is a universal homeomorphism,  we have that for any subvariety $Z\subset X_E$: $f^\ast f_\ast [Z]=m[Z]$ for some $m\geq 1$ and $f_\ast\circ f^\ast=p^s\cdot \id$ implies that $m$ must be a $p$-power.
Hence, $f^\ast$ is surjective after inverting $p$, as we want. 
\end{proof}

Let now $k$ be the perfect closure of a finitely generated field $k_0\subset k$,  and let $X$ be a regular (hence smooth) variety over $k$. 
By Lemma \ref{lem:c_1}, it suffices to show that the  map
\begin{align}  \label{eq:c1-Zell}
c_1\otimes \Z_\ell:\Pic(X)\otimes_\Z \Z_\ell\longrightarrow H^2(X,\Z_\ell(1))
\end{align}
induced by $c_1$ from (\ref{eq:epsilon-c1})
is injective, where we note that the right hand side identifies to continuous \'etale cohomology by Lemma \ref{lem:proet:coho=homo}. 
Using the existence of prime to $\ell$ alterations, the same argument as in Step 6 reduces us to the case where $X$ is smooth projective over $k$.
At this point the argument is similar to \cite[Remark 6.15]{jannsen}.

Since $X$ is defined over some finitely generated field, we may assume (up to enlarging $k_0$) that $X=X_0\times_{k_0}k$ for some smooth $k_0$-variety $X_0$.

Assume for the moment that $X_0$ is geometrically integral. 
By Grothendieck's theorem, the Picard functor on $X_0$ is then represented by the Picard scheme $\operatorname{\textbf{Pic}}_{X_0/k}$, see e.g.\ \cite[Theorem 9.4.8]{kleiman}.
In particular, $\Pic(X_0)$ is given by the group of $k_0$-rational points of $\operatorname{\textbf{Pic}}_{X_0/k}$.
The quotient of  $\operatorname{\textbf{Pic}}_{X_0/k}$ by the identity component is always a finitely generated group scheme (the N\'eron-Severi group).
Moreover, the identity component is an abelian variety over $k_0$ and since $k_0$ is finitely generated, its group of $k_0$-rational points is finitely generated by N\'eron's Mordell--Weil theorem  \cite{neron}.
It follows that  $\Pic (X_0)$ is a finitely generated abelian group.

In general, $X_0$ will split into a finite union of geometrically integral smooth projective varieties after a finite extension of the base field.
The above argument together with a pull and push argument then shows that in general, $\Pic(X_0)$ contains an $n$-torsion subgroup $T$ for some $n\geq 1$ such that $Q:=\Pic(X_0)/T$ is a finitely generated abelian group.
We consider the short exact sequence $0\to T\to \Pic(X_0)\to Q\to 0$.
Since $Q$ is finitely generated and $T$ is $n$-torsion, this sequence remains exact if we apply either the $\ell$-adic completion functor or  $\otimes_\Z \Z_\ell$.
Comparing the two resulting short exact sequences, we find that 
\begin{align} \label{eq:lim-Pic-X_0}
\lim_{\substack{\longleftarrow\\ r}}  (\Pic(X_0)/\ell^r) \stackrel{\cong}\longrightarrow \Pic(X_0)\otimes_\Z \Z_\ell.
\end{align}
  
The usual Kummer sequence on the \'etale site yields compatible injections $\Pic(X_0)/\ell^r\hookrightarrow H^2(X_0,\mu_{\ell^r}^{\otimes 1})$, i.e.\ an injection of projective systems.
Applying the inverse limit functor, this yields by (\ref{eq:lim-Pic-X_0}) an injection  $\Pic(X_0)\otimes_\Z\Z_\ell\hookrightarrow \lim H^2(X_0,\mu_{\ell^r}^{\otimes 1})$.
By   (\ref{eq:ses-lim-Jannsen}) and the construction of $c_1$ in (\ref{eq:epsilon-c1}), this injection factors through 
$$
c_1\otimes\Z_\ell:\Pic(X_0)\otimes_\Z\Z_\ell\longrightarrow  H^2(X_0,\Z_\ell(1))
$$
and so the latter must be injective as well.
It follows that (\ref{eq:c1-Zell}) is injective, because $c_1$ is functorial with respect to pullbacks, and the canonical pullback maps yield isomorphisms $\Pic(X_0)\otimes_\Z \Z_\ell\cong  \Pic(X)\otimes_\Z \Z_\ell $ (see Lemma \ref{lem:Chow-inseparable-extension}) and $H^2(X_0,\Z_\ell(1))\cong  H^2(X,\Z_\ell(1))$ (see \cite[Lemma 5.4.2]{BS}), since $k/k_0$ is purely inseparable by assumption.
This concludes the proof of (\ref{ax:homolog=ratl-div}) and hence finishes the proof of the proposition.
\end{proof} 

\subsection{Borel--Moore cohomology of complex analytic spaces}
If $X$ is a complex algebraic scheme with underlying analytic space $X_{\cx}$ and $A$ is an abelian group, then one may (and we will) define its Borel--Moore homology $H^{BM}_i(X_{\cx},A)$ analogous to singular homology with values in $A$, but with locally finite chains instead of finite ones, see \cite[Theorem V.12.14 and Corollary V.12.21]{bredon}. 
An alternative sheaf theoretic definition of the same group can be found in  \cite{BM},  \cite[Chapter V]{bredon}; a definition in terms of relative singular cohomology is given in \cite[Example 19.1.1]{fulton} and the references therein.
If $X$ is smooth and equi-dimensional of dimension $d_X$, then $H^{BM}_i(X_{\cx},A)\cong H^{2d_X-i}_{sing}(X_{\cx},A)$ by Poincar\'e duality, see \cite[Chapter V, Section 9]{bredon}.

\begin{proposition}\label{prop:Betti-coho}
Let $\VV$ be the category whose objects are separated schemes of finite type over $\C$ and whose morphisms are given by open immersions of schemes of the same dimension.
Let further $\mathcal A=\Mod_{\Z}$ and put $A(n):=A\otimes_{\Z}(2\pi i)^n\Z$ for all $A\in \mathcal A$ and $n\in \Z$. 
Let then  
$$
H^i(X,A(n)):=H^{BM}_{2d_X-i}(X_{\cx},A(d_X-n) ),
$$
where the right hand side denotes Borel--Moore homology of the underlying analytic space, and where $d_X=\dim X$.
Then $H^\ast(-,A(n))$ defines an integral twisted Borel--Moore cohomology theory that is adapted to algebraic equivalence, see Definition \ref{def:Borel--Moore-integral}. 
\end{proposition}
\begin{proof} 
Property (\ref{ax:pushforward}') follows from covariant functoriality of Borel--Moore homology with respect to proper maps, and 
item (\ref{ax:Gysin}') is a consequence of the long exact sequence of pairs in Borel--Moore homology, see e.g.\ \cite[\S 19.1]{fulton} and the references therein.
If $X$ is smooth and integral, then $H^i(X,A(n))\cong H^i_{sing}(X_{\cx},A(n))$.
In particular, $H^0(X,A(0))\cong A$ and there is a canonical class 
$$
[X]\in  H^0_{sing}(X_{\cx},\Z(0))=\Z,
$$ 
which corresponds to $1\in \Z$.
This proves (\ref{ax:normalization}'),  as in the direct limit (\ref{eq:Hi(kappa(x))}), it suffices to run through smooth integral varieties $V_{x}\subset \overline{\{x\}}$.

It remains to prove that (\ref{ax:Ql/Zl,Ql}')--(\ref{ax:Hilbert90}') and (\ref{ax:homolog=alg-div}') hold, where we recall that these properties are formally   deduced from   (\ref{ax:Ql/Zl,Ql})--(\ref{ax:Hilbert90}) and (\ref{ax:homolog=alg-div}) by the  replacement of symbols
$
\Z_\ell \rightsquigarrow \Z$,  $\Q_\ell \rightsquigarrow \Q$, and $ \ell^r \rightsquigarrow r$.

Item (\ref{ax:Ql/Zl,Ql}') is clear and item (\ref{ax:Bockstein}') follows from the long exact sequence associated to the coefficient sequence $0\to \Z(1)\stackrel{\times r}\to \Z(1)\to \Z/r(1)\to 0$; functoriality of the Bockstein sequence with respect to pullbacks and pushforwards in Borel--Moore (co-)homology are well-known and left to the reader.

For property (\ref{ax:Hilbert90}'), note that in the direct limit (\ref{eq:Hi(kappa(x))}) it suffices to run through regular (Zariski) open subsets $V:=V_x$ of the closure of $x$ in $X$.
In this case, $H^i(V ,A(n))$ identifies to singular cohomology and so the exponential sequence yields a map
$
H^0(V_{\cx} ,\mathcal O_{V_{\cx}}^\ast)\to H^1(V ,\Z(1)).
$
Taking direct limits,  and using that algebraic functions are holomorphic,  we get a map
$
\epsilon: \kappa(x)^\ast\to H^1(x,\Z(1)).
$
For any positive integer $r$, this induces by reduction modulo $r$ a map
$$
\overline \epsilon:\kappa(x)^\ast\to H^1(x,\Z/r(1))
$$
and we need to prove that this is surjective.

Consider the following commutative diagram of sheaves on $V_{\cx}$ (cf.\ \cite[\S 3.1]{CTV}):
$$
\xymatrix{
0\ar[r]& \Z(1)\ar[r]\ar[d]^{e^{\frac{1}{r}-}} & \mathcal O_{V_{\cx}}\ar[d]^{e^{\frac{1}{r}-}}\ar[r]^{e^{-} }& \mathcal O^\ast_{V_{\cx}}\ar[d]^{=}\ar[r] & 0  \\
0\ar[r]& \mu_{r}\ar[r]  & \mathcal O^\ast_{V_{\cx}} \ar[r]^{(-)^r}& \mathcal O^\ast_{V_{\cx}} \ar[r] & 0 .
}
$$
The rows in the above diagram are exact and we get a boundary map $\beta:H^0(V_{\cx}, \mathcal O^\ast_{V_{\cx}})\to H^1(V,\mu_r)$.
Taking the direct limit over all (Zariski) open dense $V\subset\overline{\{x\}} $ and restricting $\beta$ to algebraic functions, we get a map
$\overline \beta: \kappa(x)^\ast\to H^1(x,\mu_r)$.
Commutativity of the above diagram shows that  $\overline \beta$ identifies to $\overline \epsilon$ under the isomorphism $\Z/r(1)\to \mu_r$, $1\otimes (2\pi i)\mapsto e^{\frac{2\pi i}{r}}$.
It thus suffices to show that $\overline \beta$ is surjective.
This follows by comparing the sequence above with the Kummer sequence $0\to \mu_r\to \mathbb G_m\to \mathbb G_m\to 0$ on the \'etale site $V_\et$ and using that
$H^1(V_{\cx},\mu_r)\cong H^1(V_{\et},\mu_r)$ (see e.g.\ \cite[p. \ 117,  III.3.12]{milne}) and
$$
\lim_{\substack{\longrightarrow \\ \emptyset \neq V\subset \overline{\{x\}}}} H^1(V_{\et},\mathbb G_m)=0 ,
$$
because $
H^1(V_{\et},\mathbb G_m)\cong \Pic(V)
$ by Grothendieck's Hilbert 90 theorem.
We have thus shown that $\overline \epsilon$ is surjective.
Finally,  let $X\in \VV$ integral with generic point $\eta$ and a regular point $x\in X^{(1)}$.
We claim that   the natural composition
$$
\kappa(\eta)^\ast \stackrel{ \epsilon}\twoheadrightarrow H^1(\eta,\Z(1))\stackrel{\del}\longrightarrow H^0(x ,\Z(0))=[x]\Z ,
$$
where $\del$ is induced by (\ref{ax:Gysin}'), maps $f$ to $[x](-\nu_x(f))$, where $\nu_x$ denotes the valuation on $\kappa(\eta)$ induced by $x$.
It suffices to check this modulo an arbitrary prime power, which, thanks to the comparison between \'etale cohomology and singular cohomology with finite coefficients (see \cite[p.\ 117,  III.3.12]{milne}), follows from  Step 5 in the proof of Proposition \ref{prop:proetale-coho-arbitrary-field}. 
This concludes the proof of (\ref{ax:Hilbert90}').


Finally, property (\ref{ax:homolog=alg-div}') is well-known in the case where $X$ is smooth projective and follows for arbitrary smooth $X$ by choosing a smooth compactification (using resolution of singularities) by a similar argument as in Step 6 of the proof of Proposition \ref{prop:proetale-coho-arbitrary-field}.
This concludes the proof of the proposition.
\end{proof}

\section{Comparison theorems to algebraic cycles} \label{sec:comparison-thm}

In this section, we fix a prime $\ell$ and an $\ell$-adic twisted Borel--Moore cohomology theory $H^\ast(-,A(n))$ on a constructible category of Noetherian schemes $\VV$ with coefficients in a full subcategory $\mathcal A\subset \Mod_{\Z_\ell}$,  see Definition  \ref{def:Borel--Moore-cohomology-l-adic}.
In particular, (\ref{ax:pushforward})--(\ref{ax:Hilbert90}) hold true.
The main result is that this set-up allows to compute several cycle groups efficiently.

\subsection{\texorpdfstring{$\ell$}{ell}-adic Chow groups} \label{subsec:CHi-Ai}
We use the notation $\CH^i(X)_{\Z_\ell}:=\CH^i(X)\otimes_\Z \Z_\ell$.

\begin{lemma} \label{lem:CHi}
For any $X\in \VV$, there is a canonical isomorphism
$$
\CH^i(X)_{\Z_\ell}\cong \frac{\bigoplus_{x\in X^{(i)}} [x]\Z_\ell }{\im \left( \bigoplus_{x\in X^{(i-1)}} \kappa(x)^\ast \otimes_{\Z} \Z_\ell \stackrel{\epsilon} \longrightarrow  \bigoplus_{x\in X^{(i-1)}} H^1(x,\Z_\ell(1))\stackrel{\del\circ \iota_\ast}\longrightarrow  \bigoplus_{x\in X^{(i)}} [x]\Z_\ell \right)   } ,
$$  where $\epsilon$  is induced by the map from (\ref{ax:Hilbert90}), and where $\del\circ \iota_\ast$ denotes the composition
$$
\bigoplus_{x\in X^{(i-1)}} H^1(x,\Z_\ell(1)) \stackrel{\iota_\ast }\longrightarrow   H^{2i-1}(F_{i-1}X,\Z_\ell(1))\stackrel{\del}\longrightarrow \bigoplus_{x\in X^{(i)}} H^0(x,\Z_\ell(0))= \bigoplus_{x\in X^{(i)}}  [x]\Z_\ell ,
$$
where $\iota_\ast$ and $\del$ are induced by (\ref{ax:Gysin}) and the last equality uses (\ref{ax:normalization}).
\end{lemma}
\begin{proof}
We recall our convention that for a Noetherian scheme $X$, $X^{(j)}$ denotes the set of points $x\in X$ of dimension $\dim(\overline{\{x\}})=\dim (X)-j$. 
In particular, $\CH^i(X)_{\Z_\ell}$ is the quotient of $\bigoplus_{x\in X^{(i)}} [x]\Z_\ell$ by the $\Z_\ell$-submodule generated by cycles that are given by the pushforward of a principal divisor on the normalization of some subvariety $W\subset X$ with $\dim W=\dim X-i+1$.
The lemma follows therefore directly from Lemma \ref{lem:del=delta} and the second part of (\ref{ax:Hilbert90}).
This concludes the proof.
\end{proof}

In view of Lemma \ref{lem:CHi}, it is natural to make the following definition.
\begin{definition}\label{def:Ai}
For $X\in \VV$, we define
$$
A^i(X)_{\Z_\ell}:=\frac{\bigoplus_{x\in X^{(i)}} [x]\Z_\ell }{\im \left(   \bigoplus_{x\in X^{(i-1)}} H^1(x,\Z_\ell(1))\stackrel{\del\circ \iota_\ast}\longrightarrow  \bigoplus_{x\in X^{(i)}} [x]\Z_\ell \right)  }.
$$ 
\end{definition}

By Lemma \ref{lem:CHi},  there is a canonical surjection $$
\CH^i(X)_{\Z_\ell}\twoheadrightarrow A^i(X)_{\Z_\ell}.
$$
We compute the kernel of this surjection in Lemma \ref{lem:Ai-vs-Ni-1} below.

\subsection{Cycle class maps and coniveau filtration} \label{subsec:cli}

By Corollary \ref{cor:F^i}, $H^{2i}(X,\Z_\ell(i))\cong H^{2i}(F_iX,\Z_\ell(i))$.
The Gysin sequence from
Lemma \ref{lem:les} yields therefore a map
$$
\iota_\ast:\bigoplus_{x\in X^{(i)}}[x]\Z_\ell\longrightarrow H^{2i}(X,\Z_\ell(i)) ,
$$
which is zero on the image of 
$
\del: \bigoplus_{x\in X^{(i-1)}} H^1(x,\Z_\ell(1))\to \bigoplus_{x\in X^{(i)}}[x]\Z_\ell .
$
It thus follows from Lemma \ref{lem:CHi} and Definition \ref{def:Ai} that there is a well-defined cycle class map
\begin{align} \label{eq:cli}
\cl^i_X:\CH^i(X)_{\Z_\ell}\longrightarrow H^{2i}(X,\Z_\ell(i))
\end{align}
which factors through the canonical surjection $\CH^i(X)_{\Z_\ell}\twoheadrightarrow A^i(X)_{\Z_\ell}$.
We then define
$$
\CH_0^i(X)_{\Z_\ell}:=\ker(\cl_X^i).
$$

Since the category $\VV$ is constructible,  $Z\in \VV$ for any closed subscheme $Z\subset X$.
Using this, we can define the coniveau filtration $N^\ast$ on $\CH^i(X)_{\Z_\ell}$ as follows.

\begin{definition}\label{def:NjCHi}
A class $z\in \CH^i(X)_{\Z_\ell}$ has coniveau $j$, i.e.\ $z\in N^j\CH^i(X)_{\Z_\ell}$, if and only if it is homologically trivial on a closed  subscheme of codimension $j$.
More precisely,  $z\in N^j\CH^i(X)_{\Z_\ell}$ if and only if 
there is  a closed  subscheme $\iota: Z\hookrightarrow X$ with $j=\dim X-\dim Z$ and a cycle $z'\in \CH^{i-j}_0(Z)_{\Z_\ell}$ with $z=\iota_\ast z'\in \CH^i(X)_{\Z_\ell}$.
\end{definition}
For the case when $X$ is not equi-dimensional, we
recall from Section \ref{sec:notation} that $X^{(i)}:=X_{(d-i)}$ and so $\CH^i(X)_{\Z_\ell}$ is the group of $\ell$-adic cycles of dimension $d-i$, where $d=\dim X$.

The coniveau filtration $N^\ast$ on $\CH^i(X)_{\Z_\ell}$ is of the following form
$$
N^{i}=0\subset N^{i-1}\subset N^{i-2}\subset \cdots \subset N^1\subset N^0=\CH_0^i(X)_{\Z_\ell}\subset \CH^i(X)_{\Z_\ell}.
$$
This definition is related to the groups $A^i(X)_{\Z_\ell}$ from Definition \ref{def:Ai}, as follows.

\begin{lemma}\label{lem:Ai-vs-Ni-1}  
Let $X\in \VV$. 
Then  $A^i(X)_{\Z_\ell}=\CH^i(X)_{\Z_\ell}/N^{i-1}\CH^i(X)_{\Z_\ell}$.
\end{lemma}
\begin{proof}
We need to show that a codimension $i$-cycle on $X$ has coniveau $i-1$  if and only if it is represented by a cycle in
\begin{align} \label{eq:lem:Ai-vs-Ni-1}
  \im \left(   \bigoplus_{x\in X^{(i-1)}} H^1(x,\Z_\ell(1))\stackrel{\del\circ \iota_\ast}\longrightarrow  \bigoplus_{x\in X^{(i)}} [x]\Z_\ell \right) .
\end{align}

For $Z\in \VV$ the Gysin sequence (\ref{ax:Gysin}) yields a residue map
\begin{align*} 
\del: H^1(F_0Z,\Z_\ell(1)) \longrightarrow \bigoplus_{z\in Z^{(1)}}[z]\Z_\ell .
\end{align*}
The compatibility of the Gysin sequence with proper pushforwards yields a commutative diagram
$$
\xymatrix{
 H^1(F_0Z^{\red},\Z_\ell(1)) \ar[r]\ar[d] & \bigoplus_{z\in (Z^{\red})^{(1)}}[z]\Z_\ell\ar[d]^{=}\\
  H^1(F_0Z,\Z_\ell(1)) \ar[r] & \bigoplus_{z\in Z^{(1)}}[z]\Z_\ell
}
$$
where $Z^{\red}$ denotes the reduced scheme that underlies $Z$.
If $Z\subset X$ is closed so that $U=X\setminus Z$ satisfies $\dim U=\dim X$, then the comparison of the Gysin sequences for the pairs $(X,Z)$ and $(X,Z^{\red})$ shows by the five lemma that the pushforward map  $H^i(Z^{\red},A(n))\to H^i(Z,A(n))$ is an isomorphism.
This argument remains valid if we replace $Z$ by some dense open subset $Z^\circ\subset Z$ and $X$ by $X\setminus(Z\setminus Z^\circ)$.
It follows that  in the above diagram the vertical arrow on the left is an isomorphism for all  $Z\subset X$ closed with $\dim (X\setminus Z)=\dim X$.  
The   image in (\ref{eq:lem:Ai-vs-Ni-1}) therefore agrees with 
$$
  \im \left(   \bigoplus_{Z\subset X} H^1(F_0Z,\Z_\ell(1))\stackrel{\del\circ \iota_\ast}\longrightarrow  \bigoplus_{x\in X^{(i)}} [x]\Z_\ell \right) ,
$$ 
where $Z\subset X$ runs through all closed subschemes $Z\subset X$ with $i-1=\dim X-\dim Z$ and $\iota:Z\hookrightarrow X$ denotes the inclusion.   

Let us now fix a subscheme $Z\subset X$  with $i-1=\dim X-\dim Z$.
Since the Gysin sequence (\ref{ax:Gysin}) is functorial with respect to proper pushforwards, we get from Lemma \ref{lem:les} a commutative diagram
$$
\xymatrix{
  H^1(F_0Z,\Z_\ell(1)) \ar[d]^{\iota_\ast} \ar[r]^-{\del}&   \bigoplus_{z\in Z^{(1)}} [z]\Z_\ell \ar@{^{(}->}[d]^{\iota_\ast}  \ar[r]^{\cl_Z^1} &   H^2(Z,\Z_\ell(1))\ar[d]^{\iota_\ast}  \\
H^{2i-1}(F_{i-1}X,\Z_\ell(i))\ar[r]^-{\del}& \bigoplus_{x\in X^{(i)}} [x]\Z_\ell  \ar[r]^{\cl_X^i} & H^{2i}(X,\Z_\ell(i)) ,
}
$$
with exact rows.
Exactness of the first row shows that the $\ell$-adic cycles that are homologically trivial on  $Z $  are exactly those in the image of $\del$ in the left upper corner.
This description together with the commutativity of the square on the left implies the lemma.
\end{proof}

By the above lemma,  $A^i(X)_{\Z_\ell}=\CH^i(X)_{\Z_\ell}/N^{i-1}\CH^i(X)_{\Z_\ell}$.
The following lemma computes this quotient (and hence $N^{i-1}\CH^i(X)_{\Z_\ell}$) in the cases where (\ref{ax:homolog=alg-div}) and  (\ref{ax:homolog=ratl-div}) hold, respectively; the result is essentially due to Jannsen, see \cite[Lemmas 5.7 and 5.8]{jannsen-3} where it is proven $\otimes \Q_\ell$. 

Before we state the next result, we refer the reader to \cite[\S 10.3]{fulton} for the definition of algebraic equivalence of algebraic cycles.

\begin{lemma}\label{lem:AiX}  
For $X\in \VV$, the following holds:
\begin{enumerate} 
\item If (\ref{ax:homolog=alg-div}) holds, then 
$$
N^{i-1}\CH^i(X)_{\Z_\ell}=\CH^i(X)_{\alg}\otimes_\Z \Z_\ell \ \ \text{and}\ \ A^i(X)_{\Z_\ell}=(\CH^i(X)/\sim_{\alg})\otimes_\Z \Z_\ell .
$$  
\item If (\ref{ax:homolog=ratl-div}) holds, then $N^{i-1}\CH^i(X)_{\Z_\ell}=0$ and $A^i(X)_{\Z_\ell}=\CH^i(X)\otimes_\Z \Z_\ell$.  
\end{enumerate} 
\end{lemma}

\begin{proof}
We aim to describe the image of 
$$
\bigoplus_{x\in X^{(i-1)}}H^1(x,\Z_\ell(1)) \stackrel{\iota_\ast}\longrightarrow  H^{2i-1}(F_{i-1}X,\Z_\ell(1))\stackrel{\del}\longrightarrow \bigoplus_{x\in X^{(i)}}[x]\Z_\ell .
$$
By Lemma \ref{lem:del=delta}, the image is generated by the images of the maps
$$
H^1(F_0W',\Z_\ell(1))\stackrel{\del}\longrightarrow \bigoplus_{w\in (W')^{(1)}}[w]\Z_\ell \stackrel{\tau_\ast}\longrightarrow  \bigoplus_{x\in X^{(i)}}[x]\Z_\ell,
$$
where $W\subset X$ runs through all closed subvarieties of codimension $\dim X-\dim W=i-1$ and $\tau:W'\to W$ denotes the normalization.
By Lemma \ref{lem:les}, the image of $\del$ above is given by the kernel of
\begin{align} \label{eq:iota-lem:AiX}
\iota_\ast:\bigoplus_{w\in (W')^{(1)}}[w]\Z_\ell \longrightarrow H^2(F_1W',\Z_\ell(1))\cong H^2(W',\Z_\ell(1)) ,
\end{align}
where the last isomorphism is due to Corollary \ref{cor:F^i}. 
Since $W'$ is normal, it is regular in codimension one and so we may in (\ref{eq:iota-lem:AiX}) up to shrinking $W'$ assume that $W'$ is regular.
The kernel of (\ref{eq:iota-lem:AiX}) then coincides with the $\Z_\ell$-module spanned by algebraically trivial divisors on $W'$ if (\ref{ax:homolog=alg-div}) holds, and it coincides with the $\Z_\ell$-module spanned by principal divisors if (\ref{ax:homolog=ratl-div}) holds. 
This description proves the lemma.
\end{proof}

\subsection{The cokernel of the cycle class map} \label{sec:Zi} 

\begin{definition}
For $X\in \VV$, we define  
$$
Z^i(X)_{\Z_\ell} := \coker\left(  \cl_X^i:\CH^i(X)_{\Z_\ell}\longrightarrow H^{2i}(X,\Z_\ell(i)) \right)  ,
$$
and $Z^i(X)[\ell^r]:=Z^i(X)_{\Z_\ell} [\ell ^r]$,
where the cycle class map $\cl_X^i$ is from (\ref{eq:cli}). 
\end{definition}  

The following result generalizes Colliot-Th\'el\`ene--Voisin's computation of the failure of the integral Hodge conjecture for codimension two cycles on smooth complex projective varieties from \cite{CTV}.
The argument follows the same lines as  in Section \ref{sec:CTV} above.

\begin{theorem}\label{thm:IHC}
For any $X\in \VV$, there are canonical isomorphisms:
\begin{align}
 Z^i(X) [\ell ^r]&\cong  
 H^{2i-1}_{i-2,nr}(X,\mu_{\ell^r}^{\otimes i}) /  H^{2i-1}_{i-2,nr}( X,\Z_{\ell}( i)), \label{eq:Zi(X)[ell]}
 \\
 Z^i(X) [\ell ^\infty]&\cong  
H^{2i-1}_{i-2,nr}(X,\Q_{\ell}/ \Z_\ell(i))/H^{2i-1}_{i-2,nr}( X,\Q_{\ell}( i)).  \label{eq:Zi(X)[ell-infty]}
\end{align} 
The image of   $ H^{2i}(X,\Z_\ell(i))[\ell^r]\to Z^i(X)  [\ell ^r]$  corresponds via the  isomorphism in (\ref{eq:Zi(X)[ell]}) to the subspace generated by the image of $H^{2i-1} (X,\mu_{\ell^r}^{\otimes i})\to H^{2i-1}_{i-2,nr}(X,\mu_{\ell^r}^{\otimes i}) $.
Similarly, the  image of   $ H^{2i}(X,\Z_\ell(i))[\ell^\infty]\to Z^i(X)  [\ell ^\infty]$  corresponds via the  isomorphism in (\ref{eq:Zi(X)[ell-infty]}) to the subspace generated by the image of 
$$
H^{2i-1} (X,\Q_{\ell}/\Z_\ell(i))\to H^{2i-1}_{i-2,nr}(X,\Q_{\ell}/\Z_\ell(i)) .
$$ 
\end{theorem} 
Before we turn to the proof of the above theorem we need the following:

\begin{lemma} \label{lem:IHC}
Let $X\in \VV$.
Then
 the natural map
$$
 \frac{  H^{2i-1} (F_{i-1}X,\mu_{\ell^r}^{\otimes i})}{ H^{2i-1} (F_{i-1}X,\Z_\ell(i))}\longrightarrow  \frac{F^{i-1} H^{2i-1} (F_{i-2}X,\mu_{\ell^r}^{\otimes i})}{ F^{i-1}   H^{2i-1} (F_{i-2}X,\Z_\ell(i))}=\frac{ H^{2i-1}_{i-2,nr}(X,\mu_{\ell^r}^{\otimes i})}{ H^{2i-1}_{i-2,nr}( X,\Z_{\ell}( i))} 
$$
is an isomorphism.
\end{lemma}
\begin{proof}
Since the above map is clearly surjective, it suffices to show that it is injective.
Since
$$
F^{i-1}H^{2i-1} (F_{i-2}X,\Z_\ell(i))=\im(H^{2i-1} (F_{i-1}X,\Z_\ell(i))\to H^{2i-1} (F_{i-2}X,\Z_\ell(i)) ),
$$ 
it thus suffices to show that 
any element 
$$
\alpha\in \ker \left(  H^{2i-1} (F_{i-1}X,\mu_{\ell^r}^{\otimes i}) \longrightarrow H^{2i-1} (F_{i-2}X,\mu_{\ell^r}^{\otimes i})  \right)  
$$
satisfies
$$
\delta(\alpha)=0\in   H^{2i} (F_{i-1}X,\Z_\ell(i)),
$$
which by (\ref{ax:Bockstein}) implies that $\alpha$ lifts to an integral class.
By Lemma \ref{lem:les},
$\alpha=\iota_\ast \xi$ for some 
$$
\xi\in \bigoplus_{x\in X^{(i-1)}}H^1(x,\mu_{\ell^r}^{\otimes 1}).
$$
Since $\delta$ commutes with $\iota_\ast$ by functoriality of the Bockstein sequence in (\ref{ax:Bockstein}), we find
$
\delta(\alpha)=\iota_\ast(\delta(\xi))
$.
On the other hand,
$$
\delta(\xi)\in \bigoplus_{x\in X^{(i-1)}}H^2(x,\Z_{\ell}(1))
$$
is $\ell^r$-torsion by property (\ref{ax:Bockstein}), while the above direct sum is torsion-free by Lemma \ref{lem:torsionfree}.
Hence, $\delta(\alpha)=0$, which concludes the proof of the lemma.
\end{proof}

\begin{proof}[Proof of Theorem \ref{thm:IHC}]
By Lemma \ref{lem:les}, we have an exact sequence
\begin{align*}
 \bigoplus_{x\in X^{(i)}}[x]\Z_\ell\stackrel{\iota_\ast}\longrightarrow 
 H^{2i} (X,\Z_\ell(i))\longrightarrow   H^{2i}(F_{i-1}X,\Z_\ell(i))\longrightarrow  \bigoplus_{x\in X^{(i)}} H^1(x,\Z_\ell(0)) .
\end{align*} 
By Lemma \ref{lem:torsionfree}, the last term in this sequence is torsion-free and so 
$$
Z^i(X) [\ell ^r]\cong  H^{2i} (F_{i-1}X,\Z_\ell(i))[\ell^r] .
$$
By property (\ref{ax:Bockstein}), the Bockstein map thus induces an isomorphism
\begin{align} \label{eq:Zi(X)-proof}
Z^i(X) [\ell ^r]\cong \frac{ H^{2i-1} (F_{i-1}X,\mu_{\ell^r}^{\otimes i})}{ H^{2i-1} (F_{i-1}X,\Z_\ell(i))} .
\end{align}
By Lemma \ref{lem:IHC}, we then get a canonical isomorphism
$$
Z^i(X) [\ell ^r]\cong\frac{ H^{2i-1}_{i-2,nr}(X,\mu_{\ell^r}^{\otimes i})}{ H^{2i-1}_{i-2,nr}( X,\Z_{\ell}( i))} ,
$$
which proves (\ref{eq:Zi(X)[ell]}).

Let now $\alpha\in H^{2i-1} (F_{i-1}X,\mu_{\ell^r}^{\otimes i})$ with image
$$
[\alpha] \in \frac{ H^{2i-1} (F_{i-1}X,\mu_{\ell^r}^{\otimes i})}{ H^{2i-1} (F_{i-1}X,\Z_\ell(i))} .
$$
By Corollary \ref{cor:F^i},  $  H^{2i-1}(F_{i}X,\mu_{\ell^r}^{\otimes i})  \cong  H^{2i-1}(X,\mu_{\ell^r}^{\otimes i})  $ and so $[\alpha]$
lifts to $F^{i} H^{2i-1}(F_{i-1}X,\mu_{\ell^r}^{\otimes i})  $ if and only if $\delta(\alpha)\in H^{2i} (F_{i-1}X,\Z_\ell(i))[\ell^r]$
lifts to an $\ell^r$-torsion class in $H^{2i} (X,\Z_\ell(i))$.
Hence,
 the image of the $\ell^r$-torsion classes $ H^{2i}(X,\Z_\ell(i))[\ell^r]$ inside $Z^i(X) [\ell ^r]$  correspond via (\ref{eq:Zi(X)-proof}) to the subspace 
 $$
 \frac{F^i H^{2i-1} (F_{i-1}X,\mu_{\ell^r}^{\otimes i})}{  H^{2i-1} (F_{i-1}X,\Z_\ell(i))}\subset  \frac{ H^{2i-1} (F_{i-1}X,\mu_{\ell^r}^{\otimes i})}{ H^{2i-1} (F_{i-1}X,\Z_\ell(i))} ,
 $$
 where we recall that $F^i H^{2i-1} (F_{i-1}X,\mu_{\ell^r}^{\otimes i})\subset H^{2i-1} (F_{i-1}X,\mu_{\ell^r}^{\otimes i})$ is  the image of $H^{2i-1} (X,\mu_{\ell^r}^{\otimes i}) $, see Corollary \ref{cor:F^i}.
 Combining this with the isomorphism in Lemma \ref{lem:IHC}, we find that the image of   $ H^{2i}(X,\Z_\ell(i))[\ell^r]\to Z^i(X)  [\ell ^r]$  corresponds via the  isomorphism in (\ref{eq:Zi(X)[ell]}) to the subspace generated by the image of $H^{2i-1} (X,\mu_{\ell^r}^{\otimes i})\to H^{2i-1}_{i-2,nr}(X,\mu_{\ell^r}^{\otimes i}) $, as claimed.

By (\ref{ax:Ql/Zl,Ql}),  the identity  (\ref{eq:Zi(X)[ell-infty]}) as well as the assertion on the image of $ H^{2i}(X,\Z_\ell(i))[\ell^\infty]\to Z^i(X)  [\ell ^\infty]$ follow from what we have proven above by taking direct limits over $r$. 
This concludes the proof of the theorem.
\end{proof}
 
The above proof has the following consequence, that we want to record here.

\begin{corollary}\label{cor:delta(extension)}
Let $X\in \VV$ and $\alpha\in H^{2i-1}(F_{i-1}X,\mu_{\ell ^r}^{\otimes i})$.
Then 
$$
 \delta(\alpha)\in \im( H^{2i}( X,\Z_{\ell}(i))\to H^{2i}( F_{i-1}X,\Z_{\ell}(i)) ) .
$$
\end{corollary}
\begin{proof}
By Lemma \ref{lem:les}, we have an exact sequence
\begin{align*}
 \bigoplus_{x\in X^{(i)}}[x]\Z_\ell\stackrel{\iota_\ast}\longrightarrow 
 H^{2i} (X,\Z_\ell(i))\longrightarrow   H^{2i}(F_{i-1}X,\Z_\ell(i))\longrightarrow  \bigoplus_{x\in X^{(i)}} H^1(x,\Z_\ell(0)) .
\end{align*}
For any
$
\alpha\in H^{2i-1}(F_{i-1}X,\mu_{\ell^r}^{\otimes i}) 
$,
the class $\delta(\alpha)\in H^{2i}(F_{i-1}X,\Z_\ell(i))$ is torsion and so Lemma \ref{lem:torsionfree} implies that it maps to zero in 
$
 \bigoplus_{x\in X^{(i)}} H^1(x,\Z_\ell(0)) $.
\end{proof}

\subsection{The $\ell$-adic Griffiths group} \label{subsec:Griffiths}

Recall from (\ref{eq:cli}) that the cycle class map $\cl_X^i:\CH^i(X)_{\Z_\ell}\to H^{2i}(X,\Z_\ell(i))$ factorizes through $A^i(X)_{\Z_\ell}$ from Definition \ref{def:Ai}.
We denote the induced cycle class map on $A^i(X)_{\Z_\ell}$ by $\widetilde \cl_X^i$.

\begin{definition} \label{def:Griff}
For $X\in \VV$, we define
$$
\Grifftilde^i(X)_{\Z_\ell}:=\ker \left(\widetilde \cl_X^i: A^i(X)_{\Z_\ell}\longrightarrow H^{2i}(X,\Z_\ell(i))\right).
$$ 
\end{definition}

By Lemma \ref{lem:AiX}, $\Grifftilde^i(X)_{\Z_\ell}$ coincides with the $\ell$-adic Griffiths group of homologically trivial $\Z_\ell$-cycles modulo algebraic equivalence if (\ref{ax:homolog=alg-div}) holds, while it is given by the kernel of the cycle class map $\CH^i(X)_{\Z_\ell}\to H^{2i}(X,\Z_\ell(i))$ if (\ref{ax:homolog=ratl-div}) holds.

Using the definition of $A^i(X)_{\Z_\ell}$ from Section \ref{subsec:CHi-Ai}, we get
\begin{align} \label{eq:Griff=def}
\Grifftilde^i(X)_{\Z_\ell}=\frac{\ker \left( \iota_\ast:\bigoplus_{x\in X^{(i)}}[x]\Z_\ell\longrightarrow H^{2i}(X,\Z_\ell(i))\right) }{\im \left( \del\circ \iota_\ast:  \bigoplus_{x\in X^{(i-1)}} H^1(x,\Z_\ell(1)) \longrightarrow  \bigoplus_{x\in X^{(i)}} [x]\Z_\ell \right) } .
\end{align} 
The following result is motivated by Bloch--Ogus' computation of the second Griffiths group of a smooth complex projective variety in \cite[(7.5)]{BO}.

\begin{proposition} \label{prop:Griff}
For $X\in \VV$, there is a canonical isomorphism
$$
\Grifftilde^i(X)_{\Z_\ell}\cong  H^{2i-1}_{i-2,nr}( X,\Z_\ell(i))/ H^{2i-1}(X,\Z_\ell(i))  .
$$
\end{proposition}
\begin{proof}
By Lemma \ref{lem:les}  and Corollary \ref{cor:F^i}, we have exact sequences
$$
H^{2i-1}(X,\Z_\ell(i))\longrightarrow H^{2i-1}(F_{i-1}X,\Z_\ell(i))\stackrel{\del}\longrightarrow \bigoplus_{x\in X^{(i)}} [x]\Z_\ell\stackrel{\iota_\ast}\longrightarrow H^{2i}(X,\Z_\ell(i))
$$
and
$$
\bigoplus_{x\in X^{(i-1)}}H^1(x,\Z_{\ell}(1))\stackrel{\iota_\ast}\longrightarrow H^{2i-1}(F_{i-1}X,\Z_{\ell}(i))  \longrightarrow  H^{2i-1}(F_{i-2}X,\Z_{\ell}(i)) .
$$
This shows by (\ref{eq:Griff=def}) that $\Grifftilde^i(X)_{\Z_\ell}$ 
is isomorphic to
$$
 \im \left( \frac{ H^{2i-1}(F_{i-1}X,\Z_\ell(i))}{H^{2i-1}(X,\Z_\ell(i))}\longrightarrow \frac{ H^{2i-1}(F_{i-2}X,\Z_\ell(i))}{H^{2i-1}(X,\Z_\ell(i))} \right) ,
$$
which proves the proposition by definition of $ H^{2i-1}_{i-2,nr}( X,\Z_\ell(i))$.
\end{proof}

Combining Theorem  \ref{thm:IHC} and Proposition \ref{prop:Griff}, we obtain the following.  

\begin{corollary}\label{cor:extension}
Let $X\in \VV$. Then there is a canonical short exact sequence  
$$
0\longrightarrow \Grifftilde^i(X)_{\Z_\ell}\otimes \Z/\ell^r\longrightarrow \frac{ H^{2i-1}_{i-2,nr}( X,\mu_{\ell^r}^{\otimes i} )}{H^{2i-1}(X,\mu_{\ell^r}^{\otimes i}) }\longrightarrow \frac{Z^{i}(X)[\ell^r]}{H^{2i}(X,\Z_\ell(i))[\ell^r]}\longrightarrow 0 .
$$
\end{corollary}
\begin{proof} 
By  Proposition \ref{prop:Griff}, there is an exact sequence
\begin{align*} 
\Grifftilde^i(X)_{\Z_\ell}\otimes \Z/\ell^r\longrightarrow \frac{H^{2i-1}_{i-2,nr}( X,\mu_{\ell^r}^{\otimes i} )}{H^{2i-1}(X,\mu_{\ell^r}^{\otimes i}) }\longrightarrow \frac{H^{2i-1}_{i-2,nr}( X,\mu_{\ell^r}^{\otimes i} )}{H^{2i-1}_{i-2,nr}( X,\Z_{\ell}(i) )\oplus H^{2i-1}(X,\mu_{\ell^r}^{\otimes i}) }\longrightarrow 0 .
\end{align*}
By Theorem \ref{thm:IHC}, $Z^{i}(X)[\ell^r]\cong  H^{2i-1}_{i-2,nr}( X,\mu_{\ell^r}^{\otimes i} )/ H^{2i-1}_{i-2,nr}( X,\Z_{\ell}(i))$.
One checks that the natural map $H^{2i-1}(X,\mu_{\ell^r}^{\otimes i})\to Z^{i}(X)[\ell^r]$ is induced by the Bockstein morphism and so its image coincides with the image of the natural map
$H^{2i}(X,\Z_\ell(i))[\ell^r]\to  Z^{i}(X)[\ell^r]$.
We thus get  an exact sequence
\begin{align} \label{eq:extension-proof}
\Grifftilde^i(X)_{\Z_\ell}\otimes \Z/\ell^r\longrightarrow \frac{H^{2i-1}_{i-2,nr}( X,\mu_{\ell^r}^{\otimes i} )}{H^{2i-1}(X,\mu_{\ell^r}^{\otimes i}) }\longrightarrow \frac{Z^{i}(X)[\ell^r]}{H^{2i}(X,\Z_\ell(i))[\ell^r]}\longrightarrow 0 
\end{align}
and it remains to show that the first arrow is injective.
For this, let $z\in \bigoplus_{x\in X^{(i)}}[x]\Z_\ell$ with $\iota_\ast z=0$ and let $[z]\in \Grifftilde^i(X)_{\Z_\ell}$.
By Lemma \ref{lem:les}, there is a class $\alpha\in  H^{2i-1}(F_{i-1}X,\Z_\ell(i))$ with $\del \alpha=z$.
The  map in question sends $[z]$ to the image of $\alpha$ in
$$
\frac{F^{i-1}  H^{2i-1}(F_{i-2}X,\mu_{\ell^r}^{\otimes i} )}{H^{2i-1}(X,\mu_{\ell^r}^{\otimes i}) }= \frac{H^{2i-1}_{i-2,nr}( X,\mu_{\ell^r}^{\otimes i} )}{H^{2i-1}(X,\mu_{\ell^r}^{\otimes i}) } .
$$
If this vanishes, then there is a class $\xi\in \bigoplus_{x\in X^{(i-1)}}H^1(x,\mu_{\ell^r}^{\otimes 1})$ such that
$$
\del(\overline \alpha+ \iota_\ast \xi)=0\in \bigoplus_{x\in X^{(i)}}[x]\Z/\ell^r,
$$
where $\overline \alpha$ denotes the image of $\alpha$ in $H^{2i-1}(F_{i-1}X,\mu_{\ell^r}^{\otimes i} ) $.
By property (\ref{ax:Hilbert90}), we can pick a lift $\xi'\in \bigoplus_{x\in X^{(i-1)}}H^1(x,\Z_{\ell}(1))$ of $\xi$ and find that
$$
\del(\alpha+\iota_\ast \xi')\in \bigoplus_{x\in X^{(i)}}[x]\Z_\ell
$$
is zero modulo $\ell^r$.
The above cycle and $z=\del \alpha$ have the same class in $A^i(X)_{\Z_\ell}$  and so $z$ has trivial image in $\Grifftilde^i(X)_{\Z_\ell} / \ell^r $.
This shows that the first map in (\ref{eq:extension-proof}) is injective and so the exact sequence from the corollary follows.
This concludes the proof.
\end{proof}

The following result gives a geometric interpretation of the extension
$$
E_{\ell^r}^i:= \frac{ H^{2i-1}_{i-2,nr}( X,\mu_{\ell^r}^{\otimes i} )}{H^{2i-1}(X,\mu_{\ell^r}^{\otimes i}) }
$$
from Corollary \ref{cor:extension}.

\begin{lemma} \label{lem:E}
For $X\in \VV$, there is a canonical isomorphism
$$
 E_{\ell^r}^i(X) \cong \ker(\bar \cl ^i_X:A^i(X)_{\Z_\ell}/\ell ^r\longrightarrow H^{2i}(X,\mu_{\ell^r}^{\otimes i}) ),
$$
where $\bar \cl ^i_X$ denotes by reduction modulo $\ell ^r$ of the cycle class map $\widetilde \cl_X^i:A^i(X)_{\Z_\ell}\to H^{2i}(X,\Z_\ell(i))$.
\end{lemma}
\begin{proof}
By Lemma \ref{lem:les}, we have exact sequences
$$
H^{2i-1}(X,\mu_{\ell^r}^{\otimes i}) \longrightarrow  H^{2i-1}(F_{i-1}X,\mu_{\ell^r}^{\otimes i})\stackrel{\del}\longrightarrow \bigoplus_{x\in X^{(i)}}[x]\Z/\ell^r \stackrel{\iota_\ast}\longrightarrow H^{2i}(X,\mu_{\ell^r}^{\otimes i})
$$
and
$$
\bigoplus_{x\in X^{(i-1)}}H^1(\kappa(x),\mu_{\ell^r}^{\otimes 1})\stackrel{\iota_\ast} \longrightarrow   H^{2i-1}(F_{i-1}X,\mu_{\ell^r}^{\otimes i}) \longrightarrow  H^{2i-1}(F_{i-2}X,\mu_{\ell^r}^{\otimes i}) .
$$
Combining these two sequences, we find that $E_{\ell^r}^i(X)$ is isomorphic to
$$
\coker \left( \del\circ \iota_\ast:\bigoplus_{x\in X^{(i-1)}}H^1(\kappa(x),\mu_{\ell^r}^{\otimes 1}) \longrightarrow  \ker \left(\iota_\ast: \bigoplus_{x\in X^{(i)}}[x]\Z/\ell^r \to  H^{2i}(X,\mu_{\ell^r}^{\otimes i}) \right) \right) .
$$
By (\ref{ax:Bockstein}) and (\ref{ax:Hilbert90}),  $H^1(\kappa(x),\Z_\ell(1))/\ell^r\cong H^1(\kappa(x),\mu_{\ell^r}^{\otimes 1})$  and so the above cokernel injects into
$$
A^i(X)_{\Z_\ell}/\ell ^r=
\coker\left(  \bigoplus_{x\in X^{(i-1)}} H^1(x,\Z_\ell(1))\stackrel{\del\circ \iota_\ast}\longrightarrow  \bigoplus_{x\in X^{(i)}} [x]\Z_\ell \right) /\ell^r ,
$$
cf.\ Definition \ref{def:Ai}.
Moreover, a class $z\in A^i(X)_{\Z_\ell}$ with reduction $[z]\in A^i(X)_{\Z_\ell}/\ell ^r$ satisfies $[z]\in E_{\ell^r}^i$ if and only if 
$$
\widetilde \cl_X^i(z)\in \ker(H^{2i}(X,\Z_\ell(i))\to H^{2i}(X,\Z_\ell(i))/\ell^r )
$$
where we use that $H^{2i}(X,\Z_\ell(i))/\ell^r \hookrightarrow H^{2i}(X,\mu_{\ell^r}^{\otimes i})$ by (\ref{ax:Bockstein}).
This concludes the proof of the lemma, because
$$
H^{2i}(X,\Z_\ell(i))/\ell^r \longrightarrow H^{2i}(X,\mu_{\ell^r}^{\otimes i})
$$
is injective by (\ref{ax:Bockstein}).
\end{proof}

\subsection{A transcendental Abel--Jacobi map on torsion cycles}  \label{subsec:AJ_tr} 
We write for simplicity $ \Grifftilde^i(X)[\ell^r]:=\Grifftilde^i(X)_{\Z_\ell}[\ell^r]$, 
where $ \Grifftilde^i(X)_{\Z_\ell}$ denotes the kernel of the cycle class map  $\widetilde \cl_X^i:A^i(X)_{\Z_\ell}\to H^{2i}(X,\Z_\ell(i))$, see Definition \ref{def:Griff}.

In this section we show that there is a canonical  map
\begin{align}   \label{def:lambda}
\lambda^i_{tr}:\Grifftilde^i(X)[\ell^\infty]\longrightarrow  \frac{H^{2i-1}(X,\Q_\ell/\Z_\ell(i))}{N^{i-1}H^{2i-1}(X,\Q_\ell(i) )},
\end{align}
where we recall that $N^jH^i(X,A(n))=\ker(H^i(X,A(n))\to H^i(F_{j-1}X,A(n)))$. 
Our result is motivated by Bloch's Abel--Jacobi map on  $\CH^i(X)[\ell^\infty]$ constructed in \cite{bloch-compositio} in the case where $X$ is smooth projective over an algebraically closed field.  
We compare the two constructions in Section \ref{sec:lambda} below.
We will see that for $k=\bar k$ and when $H^\ast$ denotes Borel--Moore pro-\'etal cohomology (see Proposition \ref{prop:proetale-coho-arbitrary-field}), then 
 the above map is the  transcendental Abel--Jacobi map on torsion cycles,  i.e.\ the smallest quotient of Bloch's map that descends to a map on 
 $$
 \Grifftilde^i(X)[\ell^\infty]=\frac{N^0\CH^i(X)_{\Z_\ell}}{N^{i-1}\CH^i(X)_{\Z_\ell}}[\ell^\infty].
 $$
 (The adjective transcendental stems from the fact that  $N^{i-1}\CH^i(X)_{\Z_\ell}=\CH^i(X)_{\alg}\otimes_{\Z}\Z_\ell$ is the space of algebraically trivial $\ell$-adic cycles in this case.) 
We will call the above map the transcendental Abel--Jacobi map, regardless of the ground field and the Borel--Moore cohomlogy theory chosen. 

Lemma \ref{lem:les} and Corollary \ref{cor:F^i}, we have an exact sequence
\begin{align} \label{eq:prop:Griff-tors-1}
H^{2i-1}(X,\Z_\ell(i))\longrightarrow  H^{2i-1}(F_{i-1}X,\Z_\ell(i))\stackrel{\del}\longrightarrow \bigoplus_{x\in X^{(i)}}[x]\Z_\ell\stackrel{\iota_\ast}\longrightarrow H^{2i}(X,\Z_\ell(i)) .
\end{align} 
Let now $[z]\in \Grifftilde^i(X)[\ell^r]$ for some $r$ and some $z\in \bigoplus_{x\in X^{(i)}}[x]\Z_\ell$.
Then $\iota_\ast z=\widetilde \cl_X^i(z)=0$ by the definition of $\widetilde \cl_X^i$  and so we may choose a lift $\alpha\in  H^{2i-1}(F_{i-1}X,\Z_\ell(i))$ via (\ref{eq:prop:Griff-tors-1}). 
This is well-defined up to classes that come from $ H^{2i-1}(X,\Z_\ell(i))$.
Since $[z]$ is $\ell^r$-torsion, (\ref{eq:Griff=def}) implies that 
$
\del(\ell^r\alpha-\iota_\ast \xi)=0
$
for some $\xi\in \bigoplus_{x\in X^{(i-1)}}H^1(x,\Z_\ell(1))$.
Hence there is a class $\beta\in H^{2i-1}(X,\Z_\ell(i))$ with
\begin{align} \label{eq:beta=l^r-alpha}
\beta=\ell^r\alpha-\iota_\ast \xi\in F^i H^{2i-1}(F_{i-1}X,\Z_\ell(i)) .
\end{align}
By (\ref{ax:Ql/Zl,Ql}), $\beta/\ell^r\in H^{2i-1}(X,\Q_\ell(i))$.
Using functoriality of $H^\ast(-,A(n))$ in the coefficients, we may consider the image of that class in $ H^{2i-1}(X,\Q_\ell/\Z_\ell(i))$ and  define (\ref{def:lambda}) via
\begin{align}   \label{def:lambda-2}
 \lambda^i_{tr}([z]):=  \left[\beta/\ell^r \right] .
\end{align} 

\begin{lemma}\label{lem:lambda-tr-well-defined}
The map $\lambda^i_{tr}$ given by (\ref{def:lambda-2})  is well-defined.
\end{lemma}
\begin{proof}
Let us first fix a  representative $z$ of $[z]$.
Then $\lambda^i_{tr}([z])$ does not depend on the choice of $\alpha$, as this would change $\beta$ by a class in $\ell^r\cdot H^{2i-1}(X,\Z_\ell(i) )$.
Also, the class $\xi$ is well-defined up to classes $\zeta\in \bigoplus_{x\in X^{(i-1)}}H^1(x,\Z_\ell(1))$ with $\del(\iota_\ast \zeta)=0$.
This changes $\beta$ by 
$$
\iota_\ast \zeta\in F^i H^{2i-1}(F_{i-1}X, \Z_\ell(i))\cong H^{2i-1}(X, \Z_\ell(i))
$$
and hence by a class in $ N^{i-1}H^{2i-1}(X,\Z_\ell(i) )$.
In particular,  
$$
  \left[\beta/\ell^r \right]\in \frac{H^{2i-1}(X,\Q_\ell/\Z_\ell(i))}{N^{i-1}H^{2i-1}(X,\Q_\ell(i) )}
$$
remains unchanged.

Finally, if we replace $z$ by a cycle $z'$ that represents the same class in $\Grifftilde^i(X)_{\Z_\ell}$, then, by (\ref{eq:Griff=def}), $z-z'=\del\iota_\ast \zeta$ for some $\zeta\in \bigoplus_{x\in X^{(i-1)}}H^1(x,\Z_\ell(1))$.
But then we can replace $\alpha$  by $\alpha-\iota_\ast \zeta$ and $\xi$ by $\xi-\ell^r\cdot \zeta$, so that the class $\beta$ does not change at all via this process.
This proves the lemma.
\end{proof} 

The following alternative description of $\lambda_{tr}^i$ is useful.

\begin{lemma} \label{lem:lambda_tr-alternative}
Let $X\in \VV$ and let $[z]\in \Grifftilde^i(X)_{\Z_\ell}$ be an $\ell^r$-torsion class.
By Proposition \ref{prop:Griff}, $[z]$ is represented by a class $\alpha\in H^{2i-1}_{i-2,nr}(X,\Z_\ell(i))$ such that $\ell^r\cdot \alpha$ lifts to a class $\beta\in H^{2i-1}(X,\Z_\ell(i))$. Then
$$
\lambda_{tr}^{i}([z])=[\beta/\ell^r]\in \frac{H^{2i-1}(X,\Q_\ell/\Z_\ell(i))}{N^{i-1}H^{2i-1}(X,\Q_\ell(i) )}.
$$
\end{lemma}
\begin{proof}
By the proof of  Proposition \ref{prop:Griff}, $\alpha$ lifts to a class $\alpha'\in H^{2i-1} (F_{i-1}X,\Z_\ell(i))$ such that $\del \alpha'=z\in \bigoplus_{x\in X^{(i)}}[x]\Z_\ell$ is a representative of $[z]\in \Grifftilde^i(X)_{\Z_\ell}$.
Since $[z]$ is $\ell^r$-torsion, the construction of $\lambda_{tr}^{i}$ shows that there is a class $\beta'\in H^{2i-1}(X,\Z_\ell(i))$ and $\xi\in \bigoplus_{x\in X^{(i-1)}}H^1(x,\Z_\ell(1))$ such that
$$
\beta'=\ell^r\cdot \alpha'+\iota_\ast \xi
$$
and
$$
\lambda_{tr}^{i}([z])=[\beta'/\ell^r]\in \frac{H^{2i-1}(X,\Q_\ell/\Z_\ell(i))}{N^{i-1}H^{2i-1}(X,\Q_\ell(i) )}.
$$
Since $\beta$ and $\beta'$ both restrict to the same class on $F_{i-2}X$, we find that
$$
\beta/\ell^r-\beta'/\ell^r\in N^{i-1}H^{2i-1}(X,\Q_\ell(i) ).
$$
Hence, $\lambda_{tr}^{i}([z])=[\beta/\ell^r]$, which concludes the proof of the lemma.
\end{proof}

\subsection{The image of the transcendental Abel--Jacobi map on torsion cycles} \label{subsec:im-lambda}
For $X\in \VV$ we write 
$$
H^i(X,\Q_\ell/\Z_\ell(n))_{div}:=\im( H^i(X,\Q_\ell(n))\to H^i(X,\Q_\ell/\Z_\ell(n))) ,
$$
which is a divisible subgroup of $H^i(X,\Q_\ell/\Z_\ell(n))$; if the torsion subgroup of $H^{i+1}(X,\Z_\ell(n))$ is finitely generated,  then $H^i(X,\Q_\ell/\Z_\ell(n))_{div}$ is in fact the maximal divisible subgroup of $H^i(X,\Q_\ell/\Z_\ell(n))$.
For $j\geq 0$, we consider the coniveau filtration
$$
N^jH^i(X,\Q_\ell/\Z_\ell(n))_{div}:=\ker(H^i(X,\Q_\ell/\Z_\ell(n))_{div}\to H^i(F_{j-1}X,\Q_\ell/\Z_\ell(n))) .
$$
The following result computes the image of $\lambda_{tr}^i$.
It is motivated by  the description of the image of the Abel--Jacobi map on $\CH^2(X)_{\tors}$ for smooth projective varieties   over algebraically closed fields, due to Bloch and Merkurjev--Suslin, see \cite[\S 18.4]{MS}. 

\begin{proposition}\label{prop:im-lambda}
For $X\in \VV$, the transcendental Abel--Jacobi map $\lambda^i_{tr}$ from Lemma \ref{lem:lambda-tr-well-defined} satisfies
$$
\im(\lambda^i_{tr})=\frac{N^{i-1}H^{2i-1}(X,\Q_\ell/\Z_\ell(i))_{div}}{N^{i-1}H^{2i-1}(X,\Q_\ell(i) )} . 
$$
\end{proposition}

\begin{proof}
Let $[z]\in  \Grifftilde^i(X)[\ell^\infty]$ and let $\alpha\in  H^{2i-1}(F_{i-1}X,\Z_\ell(i))$ with $\del \alpha=z$.
Assume that $\ell^r\cdot[z]=0$.
Then, as we have seen above, there are classes  $\xi\in \bigoplus_{x\in X^{(i-1)}}H^1(x,\Z_\ell(1))$ and $\beta\in H^{2i-1}(X,\Z_\ell(i))$ such that (\ref{eq:beta=l^r-alpha}) holds true.
By definition,
$ 
\lambda^i_{tr}([z])= [\beta/\ell^r]
$ 
is the class represented by the image of $\beta/\ell^r\in H^{2i-1}(X,\Q_\ell(i))$.
This shows in particular
$$
\im(\lambda^i_{tr})\subset  \frac{ H^{2i-1}(X,\Q_\ell/\Z_\ell(i))_{div}}{N^{i-1}H^{2i-1}(X,\Q_\ell(i) )}= \frac{ \im (H^{2i-1}(X,\Q_\ell(i))\to H^{2i-1}(X,\Q_\ell/\Z_\ell(i)) )}{N^{i-1}H^{2i-1}(X,\Q_\ell(i) )}.
$$
Next, note that $\iota_\ast \xi$ vanishes on $F_{i-2}X$ and so (\ref{eq:beta=l^r-alpha})  implies that 
$$
[\beta/\ell^r]=[\alpha]=0\in H^{2i-1}(F_{i-2}X,\Q_\ell/\Z_\ell(i))
$$
because $\alpha$ is an integral class.
Hence,
$$
\im(\lambda^i_{tr})\subset \frac{N^{i-1}H^{2i-1}(X,\Q_\ell/\Z_\ell(i))_{div}}{N^{i-1}H^{2i-1}(X,\Q_\ell(i) )} ,
$$
as we want.

Conversely, let 
$$
\gamma \in N^{i-1}H^{2i-1}(X,\Q_\ell/\Z_\ell(i))_{div} .
$$
By Lemma \ref{lem:les},
$$
\gamma \in \im\left( \bigoplus_{x\in X^{(i)}}H^1(x,\Q_\ell/\Z_\ell(1))\to H^{2i-1}(F_{i-1}X,\Q_\ell/\Z_\ell(i)) \right).
$$
Since $H^1(x, \Q_\ell(1))\to H^1(x,\Q_\ell/\Z_\ell(1))$ is surjective by (\ref{ax:Ql/Zl,Ql}) and (\ref{ax:Hilbert90}), we conclude (using again (\ref{ax:Ql/Zl,Ql})) that there is a class $\xi\in  \bigoplus_{x\in X^{(i)}}H^1(x, \Z_\ell(1))$ and a positive integer $r$ such that $\gamma$ lifts to the class
$$
\frac{1}{\ell^r}\cdot \iota_\ast \xi\in H^{2i-1}(F_{i-1}X,\Q_\ell (i)).
$$
Since $\gamma$ lifts to a rational class by assumption, (\ref{ax:Ql/Zl,Ql}) implies there is a class $\beta\in H^{2i-1}(X,\Z_\ell (i))$ and a positive integer $r'$ such that  $\gamma$ lifts to $\beta/\ell^{r'}\in H^{2i-1}( X,\Q_\ell (i))$.
Up to replacing $r$ and $r'$ by their maximum (and $\xi$ resp.\ $\beta$ by a suitable multiple), we may assume that $r=r'$.
Let then
$$
\alpha':=\beta/\ell^r-\iota_\ast \xi/\ell^r \in H^{2i-1}(F_{i-1}X,\Q_\ell(i)).
$$
Note that the image of $\alpha'$ in $H^{2i-1}(F_{i-1}X,\Q_\ell/\Z_\ell(i))$ vanishes.
By (\ref{ax:Ql/Zl,Ql}) and (\ref{ax:Bockstein}), it follows that $\alpha'$ lifts to a class $\alpha\in H^{2i-1}(F_{i-1}X,\Z_\ell(i))$.
We then find that there is a torsion class $\tau\in H^{2i-1}(F_{i-1}X,\Z_\ell(i))$ such that
$$
\ell^r \alpha=\beta-\iota_\ast \xi+\tau\in H^{2i-1}(F_{i-1}X,\Z_\ell(i)).
$$
Since $\tau$ is torsion,  there is a positive integer $s$ such that $\ell^s\cdot \tau=0$ and we get
$$
\ell^s\beta= \ell^{r+s} \alpha+\ell^s\iota_\ast \xi \in H^{2i-1}(F_{i-1}X,\Z_\ell(i)).
$$
Let $z:=\del \alpha$ with associated class $[z]\in \Grifftilde^i(X)$.
Then the above identity shows that $[z]$ is $\ell^{r+s}$-torsion with associated transcendental Abel--Jacobi invariant
$$
\lambda_{tr}^i([z])=[\ell^s\beta/\ell^{r+s}]=[\beta/\ell^r]=[\gamma]\in H^{2i-1}(X,\Q_\ell/\Z_\ell(i))/N^{i-1} H^{2i-1}(X,\Q_\ell(i)).
$$
Hence, $[\gamma]\in \im(\lambda_{tr}^i)$, which concludes the proof of the proposition.
\end{proof}

\subsection{The kernel of the transcendental Abel--Jacobi map on torsion cycles} \label{subsec:ker-lambda} 

\begin{definition}\label{def:T^i}
For $X\in \VV$, we define
$$
\mathcal T^i(X)[\ell^\infty]:=\ker\left( \lambda^i_{tr}: \Grifftilde^i(X)[\ell^\infty]\longrightarrow  \frac{H^{2i-1}(X,\Q_\ell/\Z_\ell(i))}{N^{i-1}H^{2i-1}(X,\Q_\ell(i) )} \right) 
$$
where $\lambda^i_{tr}$ is the transcendental Abel--Jacobi map defined in Section \ref{subsec:AJ_tr}.
We further let $\mathcal T^i(X)[\ell^r]\subset \mathcal T^i(X)[\ell^\infty]$ denote the subgroup of $\ell^r$-torsion elements.
\end{definition}

Recall the filtrations $F^\ast$ and $G^\ast$ from Definitions \ref{def:F} and \ref{def:G}. 

\begin{lemma}\label{lem:Gi}
Let $X\in \VV$.
Then
$$
G^{i} H^{2i-2}_{i-3,nr}( X,\mu_{\ell^r}^{\otimes i})  \subset H^{2i-2}_{i-3,nr}( X,\mu_{\ell^r}^{\otimes i}) 
$$
is the subspace of classes $\alpha\in H^{2i-2}_{i-3,nr}( X,\mu_{\ell^r}^{\otimes i}) $ that admit a lift $\alpha'\in H^{2i-2} (F_{i-2} X,\mu_{\ell^r}^{\otimes i})$ such that $\delta(\alpha')\in H^{2i-1} (F_{i-2} X,\Z_{\ell}(i))$ lifts to $H^{2i-1} ( X,\Z_{\ell}(i))$.
\end{lemma}
\begin{proof}
This is an immediate consequence of the definition and the fact that 
$$
F^iH^{2i-1} (F_{i-2} X,\Z_{\ell}(i))=\im (H^{2i-1} (X,\Z_{\ell}(i))\to H^{2i-1} (F_{i-2} X,\Z_{\ell}(i))) ,
$$
because $H^{2i-1} (F_{i} X,\Z_{\ell}(i))\cong H^{2i-1} ( X,\Z_{\ell}(i))$ by Corollary \ref{cor:F^i}.
\end{proof} 

The following result is motivated by \cite{Voi-unramified} and \cite{Ma}, where $\mathcal T^3(X)[\ell^\infty]$ is computed for smooth projective varieties over $k=\C$.

\begin{theorem}\label{thm:Griff_tors}
Let $X\in \VV$ and assume that for any $x\in X$,  $H^3(x,\Z_{\ell}(2))$ is torsion-free. 
Then there are canonical isomorphisms 
\begin{align*} 
\mathcal T^i(X)[\ell^\infty] \cong \frac{ H^{2i-2}(F_{i-2}X,\Q_\ell/\Z_\ell(i))  }{G^{i} H^{2i-2}(F_{i-2}X,\Q_\ell/\Z_\ell(i))  } \cong \frac{ H^{2i-2}_{i-3,nr}( X,\Q_\ell/\Z_\ell(i))  }{G^{i} H^{2i-2}_{i-3,nr}( X,\Q_\ell/\Z_\ell(i))  }.
\end{align*}
\end{theorem}

The above theorem will be deduced from the following two propositions below.

\begin{proposition} \label{prop:Griff_tors}
For any $X\in \VV$, there is a canonical isomorphism
$$
\mathcal T^i(X)[\ell^\infty] \cong \frac{  H^{2i-1}(F_{i-2}X,\Z_\ell(i))[\ell^\infty] }{ F^i H^{2i-1}(F_{i-2}X,\Z_\ell(i))[\ell^\infty]}.
$$
\end{proposition}
\begin{proof}
Let $[z]\in \Grifftilde^i(X)[\ell^r]$ for some $r$.
By construction and Lemma \ref{lem:lambda-tr-well-defined}, $\lambda^i_{tr}([z])=0$ if and only if for some  classes $\alpha \in  H^{2i-1}(F_{i-1}X,\Z_\ell(i))$, $\xi\in \bigoplus_{x\in X^{(i-1)}}H^1(x,\Z_\ell(1))$ and $\beta \in  H^{2i-1}( X,\Z_\ell(i))$ with $z=\del \alpha$  and 
\begin{align} \label{eq:beta}
\beta=\ell^r \alpha-\iota_\ast \xi \in H^{2i-1}(F_{i-1}X,\Z_\ell(i)),
\end{align}
 we have that $\beta/\ell^r\in H^{2i-1}(X,\Q_\ell/\Z_\ell(i))$ admits a lift
$$
\gamma\in N^{i-1}H^{2i-1}(X,\Q_\ell(i) )= \ker(H^{2i-1}(X,\Q_\ell(i) )\to H^{2i-1}(F_{i-2}X,\Q_\ell(i) ) ) .
$$
This means that
$$
\beta/\ell^r=\gamma+\epsilon\in H^{2i-1}(X,\Q_\ell(i) )
$$
for some $\epsilon\in H^{2i-1}(X,\Z_\ell(i) )$.
Replacing $\alpha$ by $\alpha-\epsilon$, we may assume that $\epsilon=0$ and so $\beta/\ell^r\in N^{i-1}H^{2i-1}(X,\Q_\ell(i) )$.

By Lemma \ref{lem:les}, there is an exact sequence
\begin{align*} 
\bigoplus_{x\in X^{(i-1)}}H^1(x,\Z_\ell(1))\stackrel{\iota_\ast}\longrightarrow  H^{2i-1}(F_{i-1}X,\Z_\ell(i))  \stackrel{f}\longrightarrow& H^{2i-1}(F_{i-2}X,\Z_\ell(i)) ,
\end{align*}
where $f$ denotes the canonical restriction map.
Since $\beta/\ell^r\in N^{i-1}H^{2i-1}(X,\Q_\ell(i) )$, we find that 
 $\lambda^i_{tr}([z])=0$ implies that the image $f(\alpha)\in  H^{2i-1}(F_{i-2}X,\Z_\ell(i)) $ of $\alpha$ is  torsion.
We claim that the map
$$
\varphi:\mathcal T^i(X)[\ell^\infty] \longrightarrow \frac{ H^{2i-1}(F_{i-2}X,\Z_\ell(i))[\ell^\infty] }{ F^i H^{2i-1}(F_{i-2}X,\Z_\ell(i))[\ell^\infty]},\ \ \ [z]\mapsto  \varphi([z]):=[f(\alpha)]
$$
 is well-defined.
 Here we emphasize that the class $\alpha\in H^{2i-1}(F_{i-1}X,\Z_\ell(i))$ used in the definition of $\varphi([z])$ is not an arbitrary representative that satisfies $\del \alpha=z$, but it is chosen in such a way that $\beta$ from (\ref{eq:beta}) satisfies  $\beta/\ell^r\in N^{i-1}H^{2i-1}(X,\Q_\ell(i) )$.
 
 To prove that $\varphi$ is well-defined,  let us first fix $z$.
 Then  the condition $\del \alpha =z$ shows that $\alpha$ is unique up to classes in $H^{2i-1}(X,\Z_\ell(i) )$ and so $[f(\alpha)]$ is independent of the choice of $\alpha$ for fixed $z$ as we quotient out $F^iH^{2i-1}(F_{i-2}X,\Z_\ell(i))[\ell^\infty]$ in the above formula.
If $z'$ and $z$ have the same class in $A^i(X)_{\Z_\ell}$, then, by (\ref{eq:Griff=def}), $z-z'=\del\iota_\ast \zeta$ for some $\zeta\in \bigoplus_{x\in X^{(i-1)}}H^1(x,\Z_\ell(1))$.
The class $\alpha':=\alpha -\iota_\ast \zeta$ then satisfies
 $\del(\alpha -\iota_\ast \zeta)=z'$ and, by (\ref{eq:beta}),
 $$
 \beta=\ell^r \alpha'-\iota_\ast (\xi-\zeta) \in H^{2i-1}(F_{i-1}X,\Z_\ell(i)).
 $$
 Since $\beta/\ell^r\in N^{i-1}H^{2i-1}(X,\Q_\ell(i) )$, we find that $\alpha'$ may be used to compute  $\varphi([z'])$, that is,  $\varphi([z'])=f(\alpha')$.
By exactness of the above sequence, $f(\alpha')=f(\alpha)$ and so $\varphi([z'])=\varphi([z])$.
This proves that $\varphi$ is well-defined, as claimed.
It remains to see that $\varphi$ is an isomorphism.

To see that $\varphi$ is injective, assume that in the above construction, $f(\alpha)$ lifts to a class in $H^{2i-1}(X,\Z_\ell(i))$.
Then, by Lemma \ref{lem:les}, there is a class $\epsilon\in \bigoplus_{x\in X^{(i-1)}}H^1(x,\Z_\ell(1))$ such that 
$$
\alpha-\iota_\ast \epsilon \in H^{2i-1}(F_{i-1}X,\Z_\ell(i))
$$ 
lifts to $H^{2i-1}(X,\Z_\ell(i))$.
Since  $\del \alpha $ and $\del (\alpha-\iota_\ast \epsilon )=0$ have the same image in $\Grifftilde^i(X)_{\Z_\ell}$, it follows that 
$$
 [z]=[\del \alpha ]=[\del (\alpha-\iota_\ast \epsilon ) ]=0 \in \Grifftilde^i(X)_{\Z_\ell},
$$ 
as we want.

Next, we claim that $\varphi$ is surjective.
For this, let $\gamma\in  H^{2i-1}(F_{i-2}X,\Z_\ell(i))[\ell^r]$.
Then 
$$
\del \gamma\in \bigoplus_{x\in X^{(i-1)}}H^2(x,\Z_\ell(1))
$$ 
is torsion and so it must vanish by Lemma \ref{lem:torsionfree}.
Hence, $\gamma=f(\alpha)$ for some $\alpha\in H^{2i-1}(F_{i-1}X,\Z_\ell(i))$.
The cycle $z=\del \alpha$ is then homologically trivial (i.e.\ lies in the kernel of $\cl^i_X$) by exactness of (\ref{eq:prop:Griff-tors-1}).
The class $[z] \in \Grifftilde^i(X)_{\Z_\ell}$ of $z$  is $\ell^r$-torsion,
because $f(\ell^r \alpha)=0$ and so $\ell^r\alpha=\iota_\ast \xi$ (and hence $\ell^r z=\del \iota_\ast \xi$) for some $\xi\in \bigoplus_{x\in X^{(i-1)}}H^1(x,\Z_\ell(1))$.
In particular, $0=\ell^r\alpha-\iota_\ast \xi$ and so $\lambda_{tr}^i([z])=0$ by construction in Section \ref{subsec:AJ_tr}.
Hence, $[z]\in \mathcal T^i(X)[\ell^r]$.
By definition of $\varphi$ above, we have $\varphi([z])=f(\alpha)=\gamma$.
This concludes the proof of the proposition.
\end{proof}

\begin{proposition}\label{prop:Griff-injective}
Let $X\in \VV$ and assume that for any $x\in X$,  $H^3(x,\Z_{\ell}(2))$ is torsion-free.
Then 
the natural map
$$
 \frac{   H^{2i-2}(F_{i-2}X,\mu_{\ell^r}^{\otimes i})  }{G^{i} H^{2i-2}(F_{i-2}X,\mu_{\ell^r}^{\otimes i})  }\longrightarrow \frac{ H^{2i-2}_{i-3,nr}( X,\mu_{\ell^r}^{\otimes i})  }{G^{i} H^{2i-2}_{i-3,nr}( X,\mu_{\ell^r}^{\otimes i})  }
$$
is an isomorphism.
\end{proposition}
\begin{proof}
The map in question is surjective by definition.  

Let now $\alpha\in H^{2i-2}(F_{i-2}X,\mu_{\ell^r}^{\otimes i})$ so that the image  $\alpha'\in  H^{2i-2}(F_{i-3}X,\mu_{\ell^r}^{\otimes i})$ of $\alpha$ is contained in $G^{i} H^{2i-2}_{i-3,nr}(X,\mu_{\ell^r}^{\otimes i}) $.
By Lemma \ref{lem:Gi}, this means that there is a lift $\alpha''\in  H^{2i-2}(F_{i-2}X,\mu_{\ell^r}^{\otimes i})$ of $\alpha'$ such that
$$
\delta(\alpha'')\in H^{2i-1}(F_{i-2}X,\Z_{\ell}(i)) 
$$
lifts to a class $\beta\in H^{2i-1}(X,\Z_{\ell}(i))$.
Then $\alpha-\alpha''$ lies in the kernel of 
$$
  H^{2i-2}(F_{i-2}X,\mu_{\ell^r}^{\otimes i})\longrightarrow  H^{2i-2}(F_{i-3}X,\mu_{\ell^r}^{\otimes i}).
$$
Lemma \ref{lem:les} thus implies that  
$$
\alpha-\alpha''=
\iota_\ast \xi\in H^{2i-2}(F_{i-2}X,\mu_{\ell^r}^{\otimes i})
$$
for some $\xi\in \bigoplus_{x\in X^{(i-2)}}H^2(x,\mu_{\ell^r}^{\otimes 2})$.
The class
$$
\delta(\iota_\ast \xi)=\iota_\ast(\delta(\xi))\in  \bigoplus_{x\in X^{(i-2)}}H^3(x,\Z_{\ell}(2))
$$
is torsion by property (\ref{ax:Bockstein}) and so it vanishes, because $H^3(x,\Z_{\ell}(2))$ is torsion-free by assumption.
This shows that
$ 
\delta(\alpha)=\delta(\alpha'') 
$. 
Since $\delta(\alpha'')$ extends to the class $\beta\in H^{2i-1}(X,\Z_{\ell}(i))$, the same holds for $\delta(\alpha)$ and so
$$
\alpha\in G^{i}  H^{2i-2}(F_{i-2}X,\mu_{\ell^r}^{\otimes i}) .
$$
This proves that the map in question is injective, as we want. 
\end{proof}

\begin{remark}
The torsion-freeness assumption in the proposition (resp.\ in Theorem \ref{thm:Griff_tors}) will in all applications be satisfied by Merkurjev--Suslin's theorem \cite{MS}, i.e.\ by the Bloch--Kato conjecture in degree 2, see Remark \ref{rem:bloch-kato}. 
\end{remark}

\begin{proof}[Proof of Theorem \ref{thm:Griff_tors}] 
We claim that the Bockstein map 
$$
\delta: H^{2i-2}(F_{i-2}X,\mu_{\ell^r}^{\otimes i} )\longrightarrow H^{2i-1}(F_{i-2}X,\Z_\ell(i)) 
$$ 
from property (\ref{ax:Bockstein}) induces an isomorphism
\begin{align}\label{eq:iso-Griff-tors}
 \frac{ H^{2i-2}(F_{i-2}X,\mu_{\ell^r}^{\otimes i})  }{G^{i} H^{2i-2}(F_{i-2}X,\mu_{\ell^r}^{\otimes i})  }\cong \frac{ H^{2i-1}(F_{i-2}X,\Z_\ell(i))[\ell^r] }{ F^i  H^{2i-1}(F_{i-2}X,\Z_\ell(i))[\ell^r]} .
\end{align}
By (\ref{ax:Bockstein}), the image of $\delta$ is  $  H^{2i-1}(F_{i-2}X,\Z_\ell(i))[\ell^r]$ and so it suffices to show that 
$$
\delta^{-1} \left( F^i H^{2i-1}(F_{i-2}X,\Z_\ell(i))[\ell^r] \right) =G^{i} H^{2i-2}(F_{i-2}X,\mu_{\ell^r}^{\otimes i}) ,
$$
which is exactly the definition of $G^i$ (see Definition \ref{def:G}).
This proves the above claim.
Taking direct limits and using (\ref{ax:Ql/Zl,Ql}), we get an isomorphism
$$
 \frac{ H^{2i-2}(F_{i-2}X,\Q_\ell/\Z_\ell(i))  }{G^{i} H^{2i-2}(F_{i-2}X,\Q_\ell/\Z_\ell(i))  }\cong \frac{ H^{2i-1}(F_{i-2}X,\Z_\ell(i))[\ell^\infty] }{ F^i  H^{2i-1}(F_{i-2}X,\Z_\ell(i))[\ell^\infty]}
$$
The first isomorphism in Theorem \ref{thm:Griff_tors} follows therefore from Proposition \ref{prop:Griff_tors}.
The second isomorphism follows from Proposition \ref{prop:Griff-injective} by taking direct limits.
This concludes the proof of the theorem.
\end{proof}

 \begin{corollary}   \label{cor:Ti_0[ell]}
For any $X\in \VV$  there are subgroups $\mathcal T_0^i(X)[\ell^r] \subset \mathcal T^i(X)[\ell^r]$ with $\mathcal T^i(X)[\ell^\infty]=\bigcup_r \mathcal T^i_0(X)[\ell^r]$ and a canonical isomorphism
$$
\mathcal T_0^i(X)[\ell^r] \cong \frac{ H^{2i-2}(F_{i-2}X,\mu_{\ell^r}^{\otimes i}) }{ G^{i} H^{2i-2}(F_{i-2} X,\mu_{\ell^r}^{\otimes i})}  \cong \frac{ H^{2i-2}_{i-3,nr}( X,\mu_{\ell^r}^{\otimes i}) }{ G^{i} H^{2i-2}_{i-3,nr}(X,\mu_{\ell^r}^{\otimes i})} .
$$  
\end{corollary} 
\begin{proof}
By Proposition \ref{prop:Griff_tors}, there is a canonical isomorphism
$$
\mathcal T^i(X)[\ell^\infty] \cong \frac{  H^{2i-1}(F_{i-2}X,\Z_\ell(i))[\ell^\infty] }{ F^i H^{2i-1}(F_{i-2}X,\Z_\ell(i))[\ell^\infty]}.
$$
Using this isomorphism, we define
$$
\mathcal T_0^i(X)[\ell^r] := \frac{  H^{2i-1}(F_{i-2}X,\Z_\ell(i))[\ell^r] }{ F^i H^{2i-1}(F_{i-2}X,\Z_\ell(i))[\ell^r]}.
$$
Hence, $\mathcal T_0^i(X)[\ell^r]\subset \mathcal T^i(X)[\ell^\infty]$ is a subgroup of $\ell^r$-torsion elements and so $\mathcal T_0^i(X)[\ell^r] \subset \mathcal T^i(X)[\ell^r]$.
Note also that $\mathcal T^i(X)[\ell^\infty]=\bigcup_r \mathcal T^i_0(X)[\ell^r]$.
The corollary thus follows from (\ref{eq:iso-Griff-tors}) and Proposition \ref{prop:Griff-injective}.
\end{proof}

For  the final result of this subsection, we will need the following definition, where for $X\in \VV$ we let
$$
\tilde \delta: H^i(X,\mu_{\ell^r}^{\otimes n})\longrightarrow  H^{i+1}(X,\mu_{\ell^r}^{\otimes n})
$$
be the composition of the Bockstein map $\delta$ from (\ref{ax:Bockstein}) with the reduction modulo $\ell^r$ map
$$
H^{i+1}(X,\Z_\ell(n))\longrightarrow H^{i+1}(X,\mu_{\ell^r}^{\otimes n})
$$
given by functoriality in the coefficients.
 
\begin{definition} \label{def:G-tilde}
For any $X\in \VV$, we define a decreasing filtration $\widetilde G^\ast$ 
on $ H^i(F_jX,\mu_{\ell ^r}^{\otimes n})$ by
$$
\alpha\in \widetilde G^m H^i(F_jX,\mu_{\ell ^r}^{\otimes n})\ \ \ \Longleftrightarrow\ \ \   \tilde \delta(\alpha)\in F^m H^{i+1}(F_jX,\mu_{\ell ^r}^{\otimes n}) .
$$ 
Moreover,
$$
\widetilde G^{m} H^{i}_{j,nr} ( X,\mu_{\ell^r}^{\otimes n}):=\im( \widetilde G^{m} H^{i} (F_{j+1}X,\mu_{\ell^r}^{\otimes n})\to  H^{i} (F_{j}X,\mu_{\ell^r}^{\otimes n})) .
$$  
\end{definition} 

It follows directly from the definition that $ 
G^m  H^i(F_jX,\mu_{\ell ^r}^{\otimes n}) 
\subset \widetilde G^m H^i(F_j X,\mu_{\ell ^r}^{\otimes n}) 
$.

\begin{proposition}\label{prop:refined:Ti/ell}
For any $X\in \VV$, the kernel of the canonical surjection
$$
\mathcal T_0^i(X)[\ell^r] \cong \frac{ H^{2i-2}(F_{i-2}X,\mu_{\ell^r}^{\otimes i}) }{ G^{i} H^{2i-2}(F_{i-2}X,\mu_{\ell^r}^{\otimes i})}  \twoheadrightarrow \frac{ H^{2i-2}(F_{i-2}X,\mu_{\ell^r}^{\otimes i}) }{\widetilde G^{i} H^{2i-2}(F_{i-2}X,\mu_{\ell^r}^{\otimes i})}
$$
is given by all classes in $ \mathcal T_0^i(X)[\ell^r] $ that are $\ell^r$-divisible in $A^i(X)_{\Z_\ell}$.
\end{proposition} 

\begin{proof}
By Proposition \ref{prop:Griff}, there is a canonical isomorphism
$$
\Grifftilde^i(X)_{\Z_\ell} \cong    H^{2i-1}_{i-2,nr}(X,\Z_\ell(i)) /H^{2i-1}(X,\Z_\ell(i)) . 
$$
The natural inclusion
$$
\mathcal T_0^i(X)[\ell^r] \hookrightarrow \Grifftilde^i(X)_{\Z_\ell} 
$$
corresponds via the isomorphism in Corollary \ref{cor:Ti_0[ell]} to the map
$$
\frac{   H^{2i-2}(F_{i-2}X,\mu_{\ell^r}^{\otimes i}) }{ G^{i} H^{2i-2}(F_{i-2}X,\mu_{\ell^r}^{\otimes i})} \longrightarrow \frac{ H^{2i-1}_{i-2,nr}( X,\Z_\ell(i)) }{H^{2i-1}(X,\Z_\ell(i))},
\ \ \ [\alpha]\mapsto [\delta (\alpha)] .
$$  
Here the fact that $\delta (\alpha)\in H^{2i-1}(F_{i-2}X,\Z_\ell(i))$ lies  in $F^{i-1} H^{2i-1}(F_{i-2}X,\Z_\ell(i))$ follows from Lemma \ref{lem:les}, because $\delta (\alpha)$ is torsion while $\bigoplus_{x\in X^{(i-1)}}H^2( x,\Z_\ell(i))$ is torsion-free by Lemma \ref{lem:torsionfree}.

Let now $[\alpha]\in \mathcal T_0^i(X)[\ell^r] $ with $\alpha\in H^{2i-2}(F_{i-2}X,\mu_{\ell^r}^{\otimes i}) $.
As we have seen above, the class
$$
\delta(\alpha)\in  H^{2i-1}(F_{i-2}X,\Z_\ell(i))
$$ 
 admits a lift
\begin{align} \label{eq:delta(alpha)'}
\delta(\alpha)'\in  H^{2i-1}(F_{i-1}X,\Z_\ell(i)).
\end{align}
By Lemma \ref{lem:les},  the lift is unique up to classes coming from $\bigoplus_{x\in X^{(i-1)}}H^1(\kappa(x),\Z_\ell(1))$.
Moreover, there is an exact sequence
$$
H^{2i-1}(X,\Z_\ell(i))\longrightarrow  H^{2i-1}(F_{i-1}X,\Z_\ell(i))\stackrel{\del}\longrightarrow \bigoplus_{x\in X^{(i)}}[x]\Z_\ell .
$$
The class
$$
\del(\delta(\alpha)')\in \bigoplus_{x\in X^{(i)}}[x]\Z_\ell
$$
is unique up to an element of the image of
$$
\del\circ \iota_\ast:\bigoplus_{x\in X^{(i-1)}}H^1(x,\Z_\ell(1))\longrightarrow  \bigoplus_{x\in X^{(i)}}[x]\Z_\ell .
$$
The cokernel of the above map is isomorphic to $A^i(X)_{\Z_\ell}$, see Definition \ref{def:Ai}.
Since $\del$ is trivial on classes that lift to $H^{2i-1}(X,\Z_\ell(i))$, we get a well-defined map
$$
\frac{H^{2i-2}(F_{i-2}X,\mu_{\ell^r}^{\otimes i} )}{G^i H^{2i-2}(F_{i-2}X,\mu_{\ell^r}^{\otimes i} )}\longrightarrow A^i(X)_{\Z_\ell},\ \ [\alpha]\mapsto [\del(\delta(\alpha)')] .
$$
This map identifies via the isomorphism in Corollary  \ref{cor:Ti_0[ell]} to the inclusion 
$\mathcal T_0^i(X)[\ell^r]\hookrightarrow A^i(X)_{\Z_\ell}$.

Let us first assume that 
$$
[\del(\delta(\alpha)')]\in A^i(X)_{\Z_\ell}
$$
is divisible by $\ell^r$.
Then up to a suitable choice of the lift $\delta(\alpha)'$, we may assume that $\del(\delta(\alpha)')$ is zero modulo $\ell^r$.
By Lemma \ref{lem:les}, there is an exact sequence
$$
H^{2i-1}(X,\mu_{\ell^r}^{\otimes i}) \longrightarrow   H^{2i-1}(F_{i-1}X,\mu_{\ell^r}^{\otimes i})\stackrel{\del}\longrightarrow \bigoplus_{x\in X^{(i)}}[x]\Z/\ell^r .
$$
We thus conclude that the reduction $\tilde \delta(\alpha)'$ modulo $\ell^r$ of  $ \delta(\alpha)'$ lifts to a class in $ H^{2i-1}(X,\mu_{\ell^r}^{\otimes i})$.
Since $\tilde \delta(\alpha)'$ is a lift of $\tilde \delta (\alpha)$, this implies $\alpha\in \widetilde G^i H^{2i-1}(F_{i-2}X,\mu_{\ell^r}^{\otimes i})$. 

Conversely, assume that $\alpha\in H^{2i-2}(F_{i-2}X,\mu_{\ell^r}^{\otimes i}) $ lies in $\widetilde G^i H^{2i-2}(F_{i-2} X,\mu_{\ell^r}^{\otimes i})$.
That is,
$$
\tilde \delta(\alpha)\in   H^{2i-1}(F_{i-2} X,\mu_{\ell^r}^{\otimes i}) 
$$ 
lifts to a class in $H^{2i-1}(X,\mu_{\ell^r}^{\otimes i})$. 
Consider the lift $\delta(\alpha)'\in   H^{2i-1}(F_{i-1} X,\Z_\ell(i))$ of $\delta(\alpha)$ from above.
The reduction $\overline{\delta(\alpha)'}\in   H^{2i-1}(F_{i-1} X,\mu_{\ell^r}^{\otimes i})$ modulo $\ell^r$ of the lift $\delta(\alpha)' $ is a lift of $\tilde \delta(\alpha)$.
Since $\tilde \delta(\alpha)$ lifts to $H^{2i-1}(X,\mu_{\ell^r}^{\otimes i})$,   Lemma \ref{lem:les} implies that there is a class
$$
\xi\in \bigoplus_{x\in X^{(i-1)}}H^1(x,\mu_{\ell^r}) ,
$$ 
such that 
$$
\del \left(  \overline{ \delta(\alpha)' }  -\iota_\ast \xi \right) =0\in \bigoplus_{x\in X^{(i)}}[x]\Z/\ell^r .
$$
Since $\xi$ lifts by (\ref{ax:Hilbert90}) to a class in $\bigoplus_{x\in X^{(i-1)}}H^1(\kappa(x),\Z_{\ell}(1))$, and because $\del \left(  \overline{ \delta(\alpha)' } \right) $ is the reduction modulo $\ell^r$ of $\del(\delta(\alpha)')$, we conclude that 
$$
[\del(\delta(\alpha)')]\in A^i(X)_{\Z_\ell}
$$
is zero modulo $\ell^r$.
Hence,
the class $[\alpha]\in \mathcal T_0^i(X)[\ell^r]$ is divisible by $\ell^r$ in $A^i(X)_{\Z_\ell}$, as we want.
This concludes the proof of the proposition.
\end{proof}

\subsection{More on the coniveau filtration on Chow groups} \label{subsec:coniveau}
Recall the coniveau filtration $N^\ast$ on $\CH^i(X)_{\Z_\ell}$ from Definition \ref{def:NjCHi}.
By Lemma \ref{lem:Ai-vs-Ni-1},
$$
A^i(X)_{\Z_\ell}=\CH^i(X)_{\Z_\ell}/N^{i-1}\CH^i(X)_{\Z_\ell}\ \ \text{and}\ \ A_0^i(X)_{\Z_\ell}=N^0\CH^i(X)_{\Z_\ell}/N^{i-1}\CH^i(X)_{\Z_\ell}.
$$
It follows that $N^\ast$ induces a filtration on $A^i_0(X)_{\Z_\ell}$, given by
$$
N^jA_0^i(X)_{\Z_\ell}=N^j\CH^i(X)_{\Z_\ell}/N^{i-1}\CH^i(X)_{\Z_\ell}
$$
for $0\leq j\leq i-1$.
Note that $N^{i-1}\Grifftilde^i(X)_{\Z_\ell}=0$.
We thus have a finite decreasing filtration of the form
$$
0=N^{i-1}\subset N^{i-2}\subset N^{i-3}\subset \dots N^1\subset N^0=\Grifftilde^i(X)_{\Z_\ell}.
$$

Let $f:X\to Y$ be a proper morphism of schemes $X,Y\in \VV$ and let $c=\dim Y-\dim X$.
The definition of $A^i(X)_{\Z_\ell}$ and $A^i_0(X)_{\Z_\ell}$ (see Definition \ref{def:Ai} and  \ref{def:Griff}) together with the functoriality of the Gysin sequence (\ref{ax:Gysin}) with respect to the pushforward maps from (\ref{ax:pushforward}) implies that there are natural pushforward maps
$$
f_\ast :A^i(X)_{\Z_\ell}\longrightarrow A^{i+c}(Y)_{\Z_\ell}\ \ \ \text{and}\ \ \ f_\ast :\Grifftilde^i(X)_{\Z_\ell}\longrightarrow \Grifftilde^{i+c}(Y)_{\Z_\ell} .
$$
Using these maps, we get the following description of the above filtration on $A^i_0(X)_{\Z_\ell}$.

\begin{lemma} \label{lem:NjAi}
Let $X\in \VV$. The coniveau filtration $N^\ast$ on $\Grifftilde^i(X)_{\Z_\ell}$ is given by
$$
N^j\Grifftilde^i(X)_{\Z_\ell} =\im\left(\lim_{\substack{\longrightarrow \\ Z\subset X}} \Grifftilde^{i-j}(Z)_{\Z_\ell}\to \Grifftilde^i(X)_{\Z_\ell} \right) ,
$$
where $Z\subset X$ runs through all closed subschemes with $\dim Z=\dim X-j$.
\end{lemma}

The following lemma shows that the coniveau filtration on algebraic cycles is surprisingly well-behaved.

\begin{lemma} \label{lem:N^jAi_0}
For $X\in \VV$, the canonical pushforward maps
$$
\lim_{\substack{\longrightarrow \\ Z\subset X}} N^0\CH^{i-j}(Z)_{\Z_\ell} \longrightarrow N^j\CH^i(X)_{\Z_\ell}
\ \ \ \text{and}\ \ \ 
\lim_{\substack{\longrightarrow \\ Z\subset X}} \Grifftilde^{i-j}(Z)_{\Z_\ell} \longrightarrow N^j\Grifftilde^i(X)_{\Z_\ell}
$$
are isomorphisms, where $Z\subset X$ runs through all closed   subschemes with $\dim Z=\dim X-j$.
\end{lemma}
\begin{proof} 
Both pushforward maps in question are  surjective by definition,  cf.\ Lemma \ref{lem:NjAi}.
Moreover, injectivity is trivial unless $0\leq j\leq i-1$, which we will assume from now on.

We first prove injectivity of the first map.
Let $z\in \CH^{i-j}(Z)_{\Z_\ell}$ be a cycle that is rationally equivalent to zero on $X$.
Then there is a closed subscheme $W\subset X$ with $i-1=\dim X-\dim W$ such that $z$ is rationally equivalent to zero on $Z\cup W$.
Since $j\leq i-1$, we find that the subscheme $Z\cup W$ appears in the direct limit in question, which settles the injectivity of the first map in the lemma.

Injectivity of the second map is similar. 
\end{proof}

For the following proposition, recall the definition of the coniveau filtration $N^\ast$ on refined unramified cohomology from Definition \ref{def:N^j}.

\begin{proposition} \label{prop:Nj-Ai}
For $X\in \VV$, the isomorphism from Proposition \ref{prop:Griff} induces an isomorphism
$$
N^j\Grifftilde^i(X)_{\Z_\ell} \stackrel{\cong}\longrightarrow \frac{N^jH^{2i-1}_{i-2,nr}(X,\Z_\ell(i))}{N^jH^{2i-1}(X,\Z_\ell(i))} .
$$
\end{proposition}
\begin{proof}
By Proposition \ref{prop:Griff}, there is a canonical isomorphism
$$
\Grifftilde^i(X)_{\Z_\ell}\cong H^{2i-1}_{i-2,nr}(X,\Z_\ell(i))/H^{2i-1}(X,\Z_\ell(i)).
$$
By Corollary \ref{cor:restr-H_nr},   for any $0\leq j\leq i-1$,  there is a canonical exact sequence
$$
\lim_{\longrightarrow}H^{2(i-j)-1}_{i-j-2,nr}(Z,\Z_\ell(i-j))\stackrel{\iota_\ast}\longrightarrow H^{2i-1}_{i-2,nr}(X,\Z_\ell(i))\longrightarrow  H^{2i-1}_{j-1,nr}(X,\Z_\ell(i)) ,
$$
where the direct limit runs through all closed reduced subschemes $Z\subset X$ of dimension $\dim Z=\dim X-j$.
Here the first map is induced by the pushforward map with respect to  $Z\hookrightarrow X$ and the second map is the canonical restriction map.

The above sequence induces a sequence
$$
\lim_{\longrightarrow}\frac{H^{2(i-j)-1}_{i-j-2,nr}(Z,\Z_\ell(i-j))}{H^{2(i-j)-1} (Z,\Z_\ell(i-j))} \stackrel{\iota_\ast}\longrightarrow \frac{H^{2i-1}_{i-2,nr}(X,\Z_\ell(i))}{ H^{2i-1}(X,\Z_\ell(i))}\longrightarrow \frac{ H^{2i-1}_{j-1,nr}(X,\Z_\ell(i))}{ H^{2i-1} (X,\Z_\ell(i))}  ,
$$
and one directly checks that this sequence remains exact.
By Proposition \ref{prop:Griff}, the first arrow in this sequence identifies to the natural map
$$
\lim_{\longrightarrow}\Grifftilde^{i-j}(Z)_{\Z_\ell}\stackrel{\iota_\ast}\longrightarrow\Grifftilde^i(X)_{\Z_\ell} . 
$$
It follows from the functoriality of the Gysin sequence with respect to proper pushforwards (see (\ref{ax:Gysin})) that this map agrees with the pushforward of cycles induced by $Z\hookrightarrow X$.
Hence the image of the above map is given by $N^j\Grifftilde^i(X)_{\Z_\ell}$.
The above exact sequence thus yields a canonical isomorphism
$$
N^j\Grifftilde^i(X)_{\Z_\ell}\cong \ker\left( \frac{H^{2i-1}_{i-2,nr}(X,\Z_\ell(i))}{ H^{2i-1}(X,\Z_\ell(i))}\longrightarrow \frac{H^{2i-1}_{j-1,nr}(X,\Z_\ell(i))}{ H^{2i-1} (X,\Z_\ell(i))} \right) .
$$
By definition of the coniveau filtration (see Definition \ref{def:N^j}), we thus get
$$
N^j\Grifftilde^i(X)_{\Z_\ell}\cong \im\left(N^jH^{2i-1}_{i-2,nr}(X,\Z_\ell(i)) \longrightarrow \frac{H^{2i-1}_{i-2,nr}(X,\Z_\ell(i))}{ H^{2i-1}(X,\Z_\ell(i))}  \right) .
$$
The kernel of the above map is given by the image of $N^jH^{2i-1}(X,\Z_\ell(i))$ and so 
$$
N^j\Grifftilde^i(X)_{\Z_\ell}\cong \frac{N^jH^{2i-1}_{i-2,nr}(X,\Z_\ell(i))}{N^jH^{2i-1}(X,\Z_\ell(i))}
$$
as we want. This concludes the proof of the proposition.
\end{proof}

\subsection{Higher transcendental Abel--Jacobi mappings} \label{subsec:higher-AJ}
The coniveau filtration $N^\ast$ on $\CH^i(X)_{\Z_\ell}$ induces a filtration $N^\ast$ on
 $$A^i_0(X)_{\Z_\ell}=N^0\CH^i(X)_{\Z_\ell}/N^{i-1}\CH^i(X)_{\Z_\ell}$$ 
 and hence on the torsion subgroup $A^i_0(X)[\ell^\infty]\subset A^i_0(X)_{\Z_\ell}$.
 The goal of this section is to show that the graded pieces of this filtration are detected by higher Abel--Jacobi invariants.
 To this end it will be convenient to consider 
 $$  
 \overline{J}^i_{tr}(X)[\ell^\infty]:= H^{2i-1}(X,\Q_\ell/\Z_\ell(i)) / N^1H^{2i-1}(X,\Q_\ell(i)) .
$$
Here we use a bar in our notation to emphasize that we are quotiening out $N^1H^{2i-1}(X,\Q_\ell(i)) $ and not $N^{i-1}H^{2i-1}(X,\Q_\ell(i))$, as in the construction of $\lambda_{tr}^i$ in Section \ref{subsec:AJ_tr}.
For  $i\geq 2$, we have $N^{i-1}H^{2i-1}(X,\Q_\ell(i))\subset N^1H^{2i-1}(X,\Q_\ell(i))$ and so $\lambda_{tr}^i$ induces a canonical map
\begin{align} \label{def:bar-lambda} 
\bar \lambda_{tr}^i:\Grifftilde^i(X)[\ell^{\infty}]\longrightarrow \overline{J}^i_{tr}(X)[\ell^{\infty}],
\end{align}
where we note that
 $A^i_0(X)_{\Z_\ell}=0$ for $i\leq 1$ by Lemma \ref{lem:Ai-vs-Ni-1}. 

\begin{definition} \label{def:bar-J^i_j,tr}
For $0\leq j\leq i$, we define the $j$-th higher transcendental $\ell^\infty$-torsion intermediate Jacobian  of $X$ by 
$$
 \overline{J}^i_{j,tr}(X)[\ell^\infty]:=\lim_{\substack{\longrightarrow \\ Z\subset X}} \overline{J}^{i-j}_{tr}(Z)[\ell^\infty],
$$
where $Z\subset X$ runs through all subschemes with $\dim Z=\dim X-j$.
\end{definition}
 
It follows from Lemma \ref{lem:N^jAi_0} that  the map $\bar \lambda_{tr}^{i-j}$ from (\ref{def:bar-lambda}), applied to the subschemes $Z\subset X$ with  $\dim Z=\dim X-j$,  yields in the limit  a canonical higher Abel--Jacobi map 
\begin{align}\label{def:bar-lambda_j,tr}
\bar \lambda_{j,tr}^i:N^j\Grifftilde^i(X)[\ell^{\infty}]\longrightarrow \overline{J}^i_{j,tr}(X)[\ell^{\infty}].
\end{align}
Note that $\overline{J}^i_{0,tr}(X)[\ell^{\infty}]=\overline{J}^i_{tr}(X)[\ell^{\infty}]$ and $\bar \lambda_{0,tr}^i=\bar \lambda_{tr}^i$.
The following theorem computes the kernel of $\bar \lambda_{j,tr}^i$

\begin{theorem} \label{thm:ker-lambda}
Let $X\in \VV$ and assume that the twisted $\ell$-adic Borel--Moore cohomology theory $H^\ast$ on $\VV$ has the property that for all 
$Z\subset X$ with $\dim Z=\dim X-j$, the group $H^{2(i-j)-1}(F_0Z,\Z_\ell(i-j))$ is torsion-free.
Then for any $j\geq 0$, we have $
\ker\left( \bar \lambda_{j,tr}^i\right)=N^{j+1}\Grifftilde^i(X)[\ell^{\infty}] .
$ 
\end{theorem}

\begin{proof} 
By Lemma \ref{lem:N^jAi_0} and the construction of $\bar \lambda_{j,tr}^i$ via direct limits, it suffices by induction to show that 
$$
\ker\left( \bar \lambda_{tr}^i\right)=N^{1}\Grifftilde^i(X)[\ell^{\infty}] .
$$
To this end, let $[z]\in \Grifftilde^i(X)[\ell^{\infty}]$ and let $\alpha\in H^{2i-1}_{i-2,nr}(X,\Z_\ell(i))$ be a representative of $[z]$ via the isomorphism in Proposition \ref{prop:Griff}.
Let $r\geq 1$ such that $\ell^r[z]=0\in \Grifftilde^i(X)[\ell^{\infty}]$.
Then $\ell^r\cdot \alpha$ lifts to a class $\beta\in H^{2i-1}(X,\Z_\ell(i))$.
By Lemma \ref{lem:lambda_tr-alternative},
$$
\bar\lambda_{tr}^i([z])=[\beta/\ell^r] \in \frac{H^{2i-1}(X,\Q_\ell/\Z_\ell(i))}{N^1H^{2i-1}(X,\Q_\ell(i))}.
$$
Assume now that $\bar\lambda_{tr}^i([z])=0$.
Since $H^{2i-1}(F_0X,\Z_\ell(i))$ is torsion free by assumption, the preimage of $N^1H^{2i-1}(X,\Q_\ell(i))$ via the natural map $ H^{2i-1}(X,\Z_\ell(i))\to H^{2i-1}(X,\Q_\ell(i))$ is given by $N^1H^{2i-1}(X,\Z_\ell(i))$.
Hence, (\ref{ax:Ql/Zl,Ql}) and the assumption $[\beta/\ell^r]=0 $ implies that there is a class $\beta'\in N^1H^{2i-1}(X,\Z_\ell(i))$ and a positive integer $r'$ such that
$$
\beta/\ell^r=\beta'/\ell^{r'} \in H^{2i-1}(X,\Q_\ell/\Z_\ell(i)).
$$
Up to replacing $r$ and $r'$ by their maximum and $\beta$, resp.\ $\beta'$ by a suitable multiple, we may assume that $r=r'$.
We then consider the class
$$
\gamma:=\beta-\beta'\in H^{2i-1}(X,\Z_\ell(i)).
$$
Since $r=r'$, $\gamma/\ell^r=0\in H^{2i-1}(X,\Q_\ell/\Z_\ell(i))$.
Hence there is a class $\delta \in H^{2i-1}(X,\Z_\ell(i))$ and a torsion class $\tau\in H^{2i-1}(X,\Z_\ell(i))$ with
$$
\gamma=\ell^r\delta+\tau .
$$
Since $\tau$ is torsion, there is a positive integer $s$ such that $\ell^s\tau=0$.
Hence,
$$
\ell^s\beta=\ell^s\beta'+\ell^{r+s}\delta
$$
is a lift of $\ell^{r+s}\alpha\in H^{2i-1}_{i-2,nr}(X,\Z_\ell(i))$.
Since $\beta'\in N^1$, we deduce that the image of $\ell^{r+s}\alpha$ in $ H^{2i-1}(F_0X,\Z_\ell(i))$ agrees with the image of $\ell^{r+s}\delta$.
Replacing $\alpha$ by $\alpha-\delta$ (which does not change the class $[z]$ that $\alpha$ represents, because $\del \delta=0$), we may assume that $\delta=0$ and we find that the image of $\ell^{r+s}\alpha$ in $ H^{2i-1}(F_0X,\Z_\ell(i))$ vanishes.
The latter is torsion-free by assumption and so we conclude that the image of $\alpha$ in $ H^{2i-1}(F_0X,\Z_\ell(i))$ vanishes.
By Proposition \ref{prop:Nj-Ai}, this implies $[z]\in N^{1}\Grifftilde^i(X)[\ell^{\infty}]$, as we want.
This concludes the proof of the theorem. 
\end{proof}

\begin{remark}
The torsion-freeness condition in the above theorem will in our applications be satisfied by the Bloch--Kato conjecture, proven by Voevodsky, see Remark \ref{rem:bloch-kato}.
\end{remark}

\subsection{The second piece of the coniveau filtration}
In \cite{Voi-unramified,Ma},  Voisin and Ma showed that $\mathcal T^3(X)[\ell^\infty]$ is related to unramified cohomology up to an error term given by the torsion subgroup of 
$H^{5}(X,\Z_\ell(3))/N^2H^{5}(X,\Z_\ell(3))$.
The next result shows that this error term is exactly what is captured by the $G^\ast$-filtration on traditional unramified $\Q_\ell/\Z_\ell$-cohomology from Definition \ref{def:G}.
In particular, the statement in Theorem \ref{thm:Griff_tors} specializes in the case of codimension three cycles on smooth complex projective varieties to the result in \cite{Voi-unramified,Ma}.

 \begin{proposition}\label{prop:gr_N}
For $X\in \VV$, there is a canonical surjection
$$
\varphi:\left( \frac{H^{i}(X,\Z_\ell(n))}{N^2H^{i}(X,\Z_\ell(n))}\right)_{\tors}  \twoheadrightarrow \frac{G^{\lceil i/2\rceil}H^{i-1}_{0,nr}(X,\Q_\ell/\Z_\ell (n))}{  H^{i-1}_{0,nr}(X,\Q_\ell(n))}
$$
which maps the image of $H^{i}(X,\Z_\ell(n))_{\tors}$ on the left onto the image of $H^{i-1}(X,\Q_\ell/\Z_\ell(n))$ on the right.
If $H^{i-2}(x,\Z_\ell (n))$ is torsion-free for all $x\in X^{(1)}$, then $\varphi$ is an isomorphism.  
 \end{proposition}
 
 \begin{remark} \label{rem:gr_N}
 By Remark \ref{rem:bloch-kato},  Voevodsky's proof of the Bloch--Kato conjecture implies that $H^{i-2}(x,\Z_\ell (i-3))$ is torsion-free for the cohomology theories in Proposition \ref{prop:proetale-coho-arbitrary-field} and \ref{prop:Betti-coho}, so that the surjection in the above proposition will be an isomorphism for $n=i-3$ in those cases, but we will not use this result in the remainder of this paper.
 \end{remark}
 
\begin{proof}[Proof of Proposition \ref{prop:gr_N}]
Recall that $N^2H^{i}(X,\Z_\ell(n))=\ker(H^{i}(X,\Z_\ell(n))\to H^{i}(F_{1}X,\Z_\ell(n)))$.
Hence, 
\begin{align}\label{eq:Hi/N2Hi}
\left( \frac{H^{i}(X,\Z_\ell(n))}{N^2H^{i}(X,\Z_\ell(n))}\right)_{\tors}\stackrel{\cong}\longrightarrow \Tors\left( F^{\lceil i/2\rceil}H^{i}(F_{1}X,\Z_\ell(n))) \right),
\end{align}
because $H^{i}(F_{\lceil i/2\rceil}X,\Z_\ell(n)))\cong H^{i}( X,\Z_\ell(n)))$ by Corollary \ref{cor:F^i}.

By (\ref{ax:Ql/Zl,Ql}) and exactness of the direct limit functor, the integral Bockstein sequence (\ref{ax:Bockstein}) yields in the limit $r\to \infty$ a Bockstein sequence
\begin{align*} 
\dots \longrightarrow H^i(X,\Z_\ell(n)) \longrightarrow H^i(X,\Q_\ell(n)) \longrightarrow H^i(X,\Q_{\ell}/\Z_\ell(n)) \stackrel{ \delta}\longrightarrow & H^{i+1}(X,\Z_\ell(n)) \longrightarrow \dots ,
\end{align*}
where by slight abuse of notation we denote the boundary map still by $\delta$.
By the description of $\Q_\ell$-cohomology in (\ref{ax:Ql/Zl,Ql}), the image of $ \delta$ agrees with the torsion subgroup of $H^{i+1}(X,\Z_\ell(n))$.
Using exactness of the direct limit functor once again, we find that the above sequence remains exact for $F_jX$ in place of $X$. 
By definition of $G^\ast$ in Definition \ref{def:G}, $ \delta$ induces therefore an exact sequence
$$H^{i-1} (F_1X,\Q_\ell (n))\longrightarrow 
G^{\lceil i/2\rceil}H^{i-1} (F_1X,\Q_\ell/\Z_\ell (n)) \longrightarrow
 \Tors( F^{\lceil i/2\rceil}H^{i}(F_{1}X,\Z_\ell(n))) ) \longrightarrow 0.
$$ 
From this we conclude a canonical isomorphism
$$
 G^{\lceil i/2\rceil}H^{i-1} (F_1X,\Q_\ell/\Z_\ell (n))/H^{i-1} (F_1X,\Q_\ell (n))\stackrel{\cong}\longrightarrow  \Tors( F^{\lceil i/2\rceil}H^{i}(F_{1}X,\Z_\ell(n))) )
$$
induced by $ \delta$. 
Combining this with (\ref{eq:Hi/N2Hi}), we get the surjection $\varphi$ as claimed in the proposition.

It remains to analyse the kernel of the canonical map
$$
\frac{G^{\lceil i/2\rceil}H^{i-1} (F_1X,\Q_\ell/\Z_\ell (n))}{H^{i-1} (F_1X,\Q_\ell (n))}\longrightarrow \frac{G^{\lceil i/2\rceil}H^{i-1} (F_0X,\Q_\ell/\Z_\ell (n))}{H^{i-1} (F_1X,\Q_\ell (n))} .
$$
To this end,  let $\alpha\in G^{\lceil i/2\rceil}H^{i-1} (F_1X,\Q_\ell/\Z_\ell (n))$ be a class that vanishes on $F_0X$.
By Lemma \ref{lem:les},
$$
\alpha=\iota_\ast \xi,\ \ \ \ \text{for some }\ \ \xi\in \bigoplus_{x\in X^{(1)}}H^{i-3}(x,\Q_\ell/\Z_\ell (n-1)).
$$
If $H^{i-2}(x,\Z_\ell (n))$ is torsion-free for all $x\in X^{(1)}$, then $\delta(\xi)=0$ and so $\delta(\alpha)=0$ by functoriality of the Bockstein sequence  (see (\ref{ax:Bockstein})).
Hence, $\alpha$ lifts to  a class in $H^{i-1}(F_1X,\Q_\ell(n))$ and so it vanishes in the above quotient.
This concludes the proof of the proposition.
\end{proof}

 \begin{remark}
 A version of Proposition \ref{prop:gr_N} has been proven independently by Ma \cite{ma2}.
 \end{remark}

\subsection{Comparison to Bloch--Ogus theory and to Kato homology}\label{subsec:higher-unramified}
In this section we  define
\begin{align} \label{def:Hi-mathcal H^j}
E_2^{j,i+j}(X,A(n)):= \frac{
\ker(\del \circ \iota_\ast :\oplus_{x\in X^{(j)}} H^{i}(x )\to \oplus_{x\in X^{(j+1)}} H^{i-1}(x )) }{\im(\del \circ \iota_\ast:\oplus_{x\in X^{(j-1)}} H^{i+1}(x )\to \oplus_{x\in X^{(j)} }H^{i}(x ) )},
\end{align}
where $H^\ast(x)$ is a short hand for $H^\ast(x,A(n-c))$, where $c=\codim(x)=\dim X-\dim (\overline{\{x\}})$.
If $X$ is smooth and equi-dimensional over a field $k$ and $H^\ast$ satisfies  the properties of Bloch--Ogus in \cite[\S 1]{BO} (see e.g.\ \cite[\S 2]{BO}), then $E_2^{j,i+j}(X,A(n))\cong H^j(X,\mathcal H^{i+j}_X(A(n)))$ identifies by \cite{BO} to the $j$-th cohomology of the Zariski sheaf associated to $U\mapsto H^{i+j}(U,A(n))$.

\begin{proposition}\label{prop:higher-unramified}
For any $X\in \VV$,  there is a canonical long exact sequence
$$
\ldots \to H^{i+2j-1}_{j-1,nr}( X,A(n))\to H^{i+2j-1}_{j-2,nr}( X,A(n))\to E_2^{j,i+j}(X,A(n)) \to  H^{i+2j}_{j,nr}(X,A(n)) \to \dots 
$$
\end{proposition}
 
\begin{proof}
The result follows, as explained in Section \ref{subsec:intro:BO}, from the derived couple associated to the couple from Lemma \ref{lem:les}.
With the aim of making the involved maps explicit, 
we spell out the argument  in some  detail in what follows.


Let $[\xi]\in E_2^{j,i+j}(X,A(n))$ with $\xi\in \oplus_{x\in X^{(j)}} H^{i}(x)$ and $\del \circ \iota_\ast (\xi)=0$.
By  Lemma \ref{lem:les}, the condition $\del \circ \iota_\ast (\xi)=0$ is equivalent to 
$$
\iota_\ast \xi\in F^{j+1}H^{2j+i}(F_jX) .
$$
If $\xi=\del \circ \iota_\ast (\zeta)$ for some $\zeta\in  \oplus_{x\in X^{(j-1)}} H^{i+1}(x)$, then
$$
\iota_\ast  \xi=\iota_\ast\circ  \del \circ \iota_\ast (\zeta)=0
$$
by the exactness of the Gysin sequence.
It follows that there is a well-defined map
\begin{align}\label{eq:refined-vs-higher-1}
E_2^{j,i+j}(X,A(n))\longrightarrow  H^{i+2j}_{j,nr}(X,A(n)) ,\ \ [\xi]\mapsto \iota_\ast \xi
\end{align}
Any class in the image of this map lies in the kernel of 
\begin{align}\label{eq:refined-vs-higher-2}
 H^{i+2j}_{j,nr}(X,A(n)) \longrightarrow H^{i+2j}_{j-1,nr}(X,A(n)) 
\end{align}
because $\iota_\ast \xi$ vanishes when restricted to $F_{j-1}X$ by Lemma \ref{lem:les}.
Conversely, any class $\alpha \in  H^{i+2j}_{j,nr}(X,A(n)) $ in the kernel of the above restriction map is by Lemma \ref{lem:les} of the form $\alpha=\iota_\ast \xi$ for some $\xi\in \oplus_{x\in X^{(j)}} H^{i}(x)$.
The fact that $\alpha\in H^{i+2j}_{j,nr}(X,A(n)) \subset  H^{i+2j} (F_jX,A(n))$ is unramified implies $\del\circ \iota_\ast (\xi)=0$, and so $\alpha$ lies in the image of (\ref{eq:refined-vs-higher-1}).
Hence, the composition of (\ref{eq:refined-vs-higher-1}) and (\ref{eq:refined-vs-higher-2}) is exact.

Let now $[\xi]\in E_2^{j,i+j}(X,A(n))$ with $\xi\in \oplus_{x\in X^{(j)}} H^{i}(x)$ and $\del \circ \iota_\ast (\xi)=0$ be a class in the kernel of (\ref{eq:refined-vs-higher-1}).
By the exactness of the Gysin sequence, this means that
$
 \xi=\del \alpha
$
for some $\alpha\in H^{i+2j-1}(F_{j-1}X,A(n))$.
Hence, the natural sequence
\begin{align} \label{eq:refined-vs-higher-3}
H^{i+2j-1}(F_{j-1}X,A(n))\stackrel{\del}\longrightarrow E_2^{j,i+j}(X,A(n)) \stackrel{\iota_\ast}\longrightarrow H^{i+2j}_{j,nr}(X,A(n)) 
\end{align}
is exact.
The image of
$$
\iota_\ast: \oplus_{x\in X^{(j-1)}} H^{i+1}(x,A(n-j+1))\longrightarrow H^{i+2j-1}(F_{j-1}X,A(n))
$$
lies in the kernel of the first map in (\ref{eq:refined-vs-higher-3}) by the definition in (\ref{def:Hi-mathcal H^j}).
By the Gysin sequence, it follows that (\ref{eq:refined-vs-higher-3}) descends to an exact sequence
\begin{align} \label{eq:refined-vs-higher-4}
H^{i+2j-1}_{j-2,nr}(X,A(n))\stackrel{\del}\longrightarrow E_2^{j,i+j}(X,A(n)) \stackrel{\iota_\ast}\longrightarrow  H^{i+2j}_{j,nr}(X,A(n)) .
\end{align}

Let $[\alpha]\in H^{i+2j-1}_{j-2,nr}(X,A(n))$ with $\alpha\in  H^{i+2j-1}(F_{j-1}X,A(n))$ and assume that
$$
\del \alpha=0\in E_2^{j,i+j}(X,A(n)).
$$
This means that there is a class $\zeta\in  \oplus_{x\in X^{(j-1)}} H^{i+1}(x)$ with
$ 
\del (\alpha-\iota_\ast \zeta)=0 
$. 
Hence, up to replacing $\alpha$ by $\alpha-\iota_\ast \zeta$, we may assume $\del\alpha=0$ and so 
$$
[\alpha]\in F^{j}H^{i+2j-1}(F_{j-2}X,A(n)).
$$
Conversely, any class in $F^{j}H^{i+2j-1}(F_{j-2}X,A(n))$ clearly maps to zero in $H^j(X,\mathcal H^{i+j}_X(A(n)))$.
Hence, the kernel of the first map in (\ref{eq:refined-vs-higher-4}) agrees with the image of the canonical restriction map
$$
H^{i+2j-1}_{j-1,nr}(X,A(n))\longrightarrow H^{i+2j-1}_{j-2,nr}(X,A(n)) .
$$
This concludes the proof of the proposition. 
\end{proof}

\begin{corollary}\label{cor:kato}
Let $c\geq 0$ be a non-negative integer.
Let $X\in \VV$ with $d:=\dim X$ and assume that for any $x\in X_{(j)}$,  $H^i(x,\mu_{\ell^r}^{\otimes n})=0$ for $i>j+c$.
Then there is a canonical isomorphism
$$
E_2^{j,d+c}(X,\mu_{\ell^r}^{\otimes n})
\stackrel{\sim}\longrightarrow H^{d+c+j}_{j,nr}(X,\mu_{\ell^r}^{\otimes n}).
$$ 
\end{corollary}
\begin{proof}
Our assumption implies by Corollary \ref{cor:F^i-2}:
$$
H^i_{j,nr}(X,A(n))=0\ \ \ \text{for all $j<i-d-c$.}
$$
The result in question is then an immediate consequence of Proposition \ref{prop:higher-unramified}.  
\end{proof}

\begin{remark} \label{rem:kato}
The condition in Corollary \ref{cor:kato} is satisfied for $c=0$ if $k=\C$ and the underlying cohomology theory is singular/\'etale cohomology. 
It is also satisfied if $k$ has finite cohomological dimension $c$ and the cohomology theory is twisted $\ell$-adic pro-\'etale cohomology, which for finite coefficients agrees with \'etale cohomology and so $H^i(x,\mu_{\ell^r}^{\otimes n})$ identifies by \cite[p.\ 88, III.1.16]{milne} to the Galois cohomology of the residue field $\kappa(x)$. 
In both cases,  $E_2^{j,d+c}(X,\mu_{\ell^r}^{\otimes n})$ coincides by definition with Kato homology of $X$,  see \cite{kato,KeSa,tian}.
Corollary \ref{cor:kato} thus shows that Kato homology is a special case of  refined unramified cohomology.
%
\end{remark}


%

\section{Comparison to Bloch's map}\label{sec:lambda}
In \cite{bloch-compositio}, Bloch constructed an Abel--Jacobi map on torsion cycles in the Chow group of smooth projective varieties over algebraically closed ground fields.
Bloch's map induces a transcendental Abel--Jacobi map on torsion-cycles in the Griffiths group of such varieties and we aim to show in this section that  
Bloch's map agrees with the map that we constructed in Section \ref{subsec:AJ_tr} (applied to the cohomology theory from Proposition \ref{prop:proetale-coho-arbitrary-field} in the case where $k=\bar k$ is algebraically closed and $X$ is smooth projective).

In contrast to Bloch's map,  the transcendental Abel--Jacobi map that we defined in Section \ref{subsec:AJ_tr} works for arbitrary algebraic schemes over a field.
This is crucial for the construction of the higher Abel--Jacobi maps in Section \ref{subsec:higher-AJ}. 

\subsection{\texorpdfstring{$\ell^r$}{ell}-torsion in Chow groups}
Fix a prime $\ell$ and an $\ell$-adic twisted Borel--Moore cohomology theory $H^\ast(-,A(n))$ on a constructible category of Noetherian schemes $\VV$ with coefficients in a full subcategory $\mathcal A\subset \Mod_{\Z_\ell}$, as in Definitions \ref{def:Borel--Moore-cohomology} and \ref{def:Borel--Moore-cohomology-l-adic}.
For $X\in \VV$, $x\in X$ there are isomorphisms $H^0(x,A(0))\cong A$ that are functorial in $A$.
Moreover, there is a distinguished class $[x]\in H^0(x,\Z_\ell(0))$ and we denote the image of that class in $H^0(x,\mu_{\ell^r}^{\otimes 0})$ by the same symbol, so that $H^0(x,\mu_{\ell^r}^{\otimes 0})=[x]\Z/\ell^r$.
For any $X\in \VV$, properties (\ref{ax:pushforward})--(\ref{ax:normalization}), (\ref{ax:Bockstein}), and (\ref{ax:Hilbert90}) thus imply the existence of the following commutative diagram with exact rows (cf.\ \cite[(2.1)]{bloch-compositio}): 
{\small
$$
\xymatrix{ 
& \bigoplus_{x\in X^{(i-1)}} \kappa(x)^\ast\ar[r]^{\times \ell^r} \otimes_\Z \Z_\ell \ar[d]^{\epsilon}& \bigoplus_{x\in X^{(i-1)}} \kappa(x)^\ast \otimes_\Z \Z_\ell \ar[r]\ar[d]^{\epsilon}& \bigoplus_{x\in X^{(i-1)}} \kappa(x)^\ast/(\kappa(x)^\ast)^{\ell^r} \ar[d]^{\cong}\ar[r]& 0\\ 
& \bigoplus_{x\in X^{(i-1)}} H^1( x,\Z_{\ell}(1)) \ar[r]^{\times \ell^r}\ar[d]^{\del\circ \iota_\ast}&\bigoplus_{x\in X^{(i-1)}}H^1(x,\Z_{\ell}(1)) \ar[r]\ar[d]^{\del \circ \iota_\ast}&\bigoplus_{x\in X^{(i-1)}} H^1( x,\mu_{\ell^r}^{\otimes 1})  \ar[d]^{\del \circ \iota_\ast} \ar[r]& 0 \\ 
 0\ar[r]& \ar[r] \ar[d] \bigoplus_{x\in X^{(i)}}[x]\Z_\ell\ar[r]^{\times \ell^r}& \bigoplus_{x\in X^{(i)}} [x]\Z_\ell\ar[r]\ar[d]  & \bigoplus_{x\in X^{(i)}}[x] \Z/\ell^r \ar[r]& 0\\
& \ar[d] \CH^i(X)_{\Z_\ell} \ar[r]^{\times \ell^r}&\ar[d] \CH^i(X)_{\Z_\ell}&  &\\
&  A^i(X)_{\Z_\ell}   \ar[r]^{\times \ell^r}&  A^i(X)_{\Z_\ell} & &
}
$$ }
The following is motivated by  \cite[\S 2]{bloch-compositio}.

\begin{lemma} \label{lem:snake}
For any $X\in \VV$, there are canonical isomorphisms
$$
 \phi_r:
\CH^i(X)[\ell^r] \stackrel{\cong}\longrightarrow  \frac{\ker\left(\del\circ \iota_\ast:\bigoplus_{x\in X^{(i-1)}} H^1( x,\mu_{\ell^r}^{\otimes 1})  \longrightarrow  \bigoplus_{x\in X^{(i)}} [x] \Z/\ell^r  \right) }{ \ker\left(\del\circ \iota_\ast\circ \epsilon:\bigoplus_{x\in X^{(i-1)}} \kappa(x)^\ast\otimes_\Z \Z_\ell \longrightarrow  \bigoplus_{x\in X^{(i)}}[x]\Z_\ell \right)  } 
$$
and
$$
\psi_r:
A^i(X)[\ell^r] \stackrel{\cong}\longrightarrow  \frac{\ker\left(\del\circ \iota_\ast:\bigoplus_{x\in X^{(i-1)}} H^1(x,\mu_{\ell^r}^{\otimes 1})  \longrightarrow  \bigoplus_{x\in X^{(i)}} [x]\Z/\ell^r \right) }{\ker\left(\del\circ \iota_\ast:\bigoplus_{x\in X^{(i-1)}}H^1( x,\Z_{\ell}(1)) \longrightarrow  \bigoplus_{x\in X^{(i)}}[x]\Z_\ell \right)} .
$$
\end{lemma}
\begin{proof}
Note that the first arrow in the third row of the above diagram is injective, while the last arrows in the first two rows are surjective by property (\ref{ax:Hilbert90}).
The result is therefore an immediate consequence of the snake lemma and the presentation of $ \CH^i(X)_{\Z_\ell}$ in Lemma \ref{lem:CHi}, respectively the definition of $A^i(X)_{\Z_\ell}$ in Definition \ref{def:Ai}.
\end{proof}

\subsection{The case of smooth projective varieties over algebraically closed fields}

In this section we assume that $k$ is an algebraically closed field, $\ell$ is a prime that is invertible in $k$ and $\VV$ denotes the category of separated schemes of finite type over $k$.
Let $\mathcal A\subset \Mod_{\Z_\ell}$ be the full subcategory spanned by $ \Z_\ell,\Q_\ell,\Q_\ell/\Z_\ell$ and $\Z/\ell^r$ for all $r\geq 1$.
We further fix the $\ell$-adic  twisted Borel--Moore cohomology theory on $\VV$ with coefficients in $\mathcal A$  given by Proposition \ref{prop:proetale-coho-arbitrary-field}, cf.\ Definitions \ref{def:Borel--Moore-cohomology} and \ref{def:Borel--Moore-cohomology-l-adic}.
We also note that (\ref{ax:homolog=alg-div}) holds true by Proposition \ref{prop:proetale-coho-arbitrary-field}, as $k$ is algebraically closed. 

If $X\in \VV$ is regular and equi-dimensional, then
$$
H^i(X,A(n))\cong H_{cont}^i(X_\et,A(n))\cong H^i(X_\et,A(n))
$$
where the first isomorphism comes from  Lemma \ref{lem:proet:coho=homo} and the second isomorphism uses that $k$ is algebraically closed,  so that continuous \'etale cohomology of algebraic schemes over $k$ coincides with usual \'etale cohomology,  as the $\RR^1 \lim$ term in (\ref{eq:ses-lim-Jannsen}) vanishes in this case by finiteness of the corresponding \'etale cohomology groups, cf.\ \cite{jannsen}.
(As usual,  \'etale cohomology with $\Z_\ell$-coefficients has in the above formula to be understood as inverse limit $\lim H^i(X_{\et},\mu_{\ell^r}^{\otimes n})$ and cohomology with $\Q_\ell$ or $\Q_\ell/\Z_\ell$ coefficients is as usual defined by asking that (\ref{ax:Ql/Zl,Ql}) holds.)

Bloch \cite{bloch-compositio} used Bloch--Ogus theory \cite{BO} and the Weil conjectures, proven by Deligne \cite{deligne}, to construct a map
\begin{align} \label{eq:lambda}
\lambda:\CH^i(X)[\ell^\infty] \longrightarrow H^{2i-1}(X,\Q_{\ell}/\Z_\ell(i)) 
\end{align}
which agrees  with the Abel--Jacobi map on homologically trivial cycles  in the case where $k=\C$, see \cite[Proposition 3.7]{bloch-compositio}.   
To give a description of Bloch's map in the present context, we need the following.

\begin{lemma} \label{lem:weil}
Let $k$ be an algebraically closed field and let $\ell$ be a prime that is invertible in $k$.
Let $X$ be a smooth projective variety over $k$.
Then the image of 
$$
\ker\left(\del\circ \iota_\ast\circ \epsilon:\bigoplus_{x\in X^{(i-1)}} \kappa(x)^\ast \longrightarrow  \bigoplus_{x\in X^{(i)}}[x]\Z_\ell \right) \otimes_\Z \Z_\ell 
$$
via the composition
$$
\bigoplus_{x\in X^{(i-1)}} \kappa(x)^\ast\otimes_\Z \Z_\ell  \stackrel{\epsilon} \longrightarrow  \bigoplus_{x\in X^{(i-1)}} H^1(x,\Z_\ell(1))\stackrel{\iota_\ast }\longrightarrow  H^{2i-1}(F_{i-1}X,\Z_\ell(i))  
$$ 
is torsion.
\end{lemma}
\begin{proof}
Our proof is similar to \cite[Lemma 2.4]{bloch-compositio}  
but we avoid Bloch--Ogus theory.
 
Let $\xi\in \bigoplus_{x\in X^{(i-1)}} \kappa(x)^\ast  \otimes_\Z \Z_\ell$ with
$ 
\del (\iota_\ast (\epsilon(\xi)))=0 
$. 
By Lemma \ref{lem:les}, we get
$$
\iota_\ast (\epsilon(\xi))\in F^i H^{2i-1}(F_{i-1}X,\Z_\ell(i))\cong H^{2i-1}(X,\Z_\ell(i)).
$$

If $k$ is the algebraic closure of a finite field, then $X$ and $\xi$ are both defined over $\F_q$ for some finite field $\F_q\subset k$.
In particular, $X=X_0\times_{\F_q} k$ and the Frobenius $F$ (given by $x\mapsto x^q$ on $X_0$ and by $\id$ on $k$) satisfies
$$
F(\iota_\ast ( \epsilon(\xi)))=\iota_\ast ( \epsilon(\xi^q))=q\cdot \iota_\ast (\epsilon(\xi)) .
$$
Since $X$ is smooth projective, the Weil conjectures \cite{deligne} imply that $q$ cannot appear as an eigenvalue of the action of $F$ on $H^{2i-1}(X,\Q_\ell(i))$ and so $ \iota_\ast ( \epsilon(\xi))$ must be torsion, as claimed.

If $k$ is not the algebraic closure of a finite field, then the result in question follows from spreading out the problem over a finitely generated field, which allows us to specialize to a finite field and so the smooth proper base change theorem yields the result.
This proves the lemma.
\end{proof}

Taking the direct limit of the isomorphisms from Lemma \ref{lem:snake}, we obtain an isomorphism
$$
 \phi :
\CH^i(X)[\ell^\infty] \stackrel{\cong}\longrightarrow  \frac{\ker\left(\del\circ \iota_\ast:\bigoplus_{x\in X^{(i-1)}} H^1( x,\Q_\ell/\Z_\ell(1))  \to \bigoplus_{x\in X^{(i)}} [x] \Q_\ell/\Z_\ell \right) }{ \ker\left(\del\circ \iota_\ast\circ \epsilon:\bigoplus_{x\in X^{(i-1)}} \kappa(x)^\ast\otimes_\Z \Q_\ell \to  \bigoplus_{x\in X^{(i)}}[x]\Q_\ell \right)  } .
$$

\begin{proposition}\label{prop:lambda'}
There is a well-defined map
$$
\lambda':\CH^i(X)[\ell^\infty]\longrightarrow H^{2i-1}(X,\Q_\ell/\Z_\ell(i)),
$$
given by
$$
\lambda'(\phi^{-1}([\xi])):=-\iota_\ast \xi\in F^i H^{2i-1}(F_{i-1}X,\Q_\ell/\Z_\ell(i))\cong H^{2i-1}(X,\Q_\ell/\Z_\ell(i)).
$$ 
\end{proposition}
\begin{proof}
The natural map
$$
H^{2i-1}(X, \Z_\ell( i))_{\tors}\longrightarrow \lim_{\substack{\longrightarrow \\ r}} H^{2i-1}(X,\mu_{\ell^r}^{\otimes i})\cong H^{2i-1}(X,\Q_{\ell}/\Z_\ell( i))
$$
is zero.
  Lemma \ref{lem:weil} thus implies that for any $\xi\in \bigoplus_{x\in X^{(i-1)}} \kappa(x)^\ast\otimes_\Z \Q_\ell $ with $\del(\iota_\ast(\epsilon(\xi)))=0$,  
$$
\iota_\ast(\epsilon(\xi))\in F^i H^{2i-1}(F_{i-1}X,\Q_\ell/\Z_\ell(i))\cong H^{2i-1}(X,\Q_\ell/\Z_\ell(i))
$$
vanishes.
This concludes the proof. 
\end{proof}

The minus sign in Proposition \ref{prop:lambda'} is necessary to make our definition compatible with $\lambda^i_{tr}$ defined in Section \ref{subsec:AJ_tr}; a similar sign issue was noticed by Bloch, see \cite[p.\ 112]{bloch-compositio}.

\begin{lemma}\label{lem:deligne-cycle-map}
The map $\lambda'$ constructed above coincides with the map (\ref{eq:lambda}) constructed by Bloch: $\lambda=\lambda'$. 
\end{lemma}
\begin{proof} 
This follows directly from Lemma \ref{lem:del=delta} by comparing our construction with Bloch's construction via diagram (2.2) in \cite{bloch-compositio}, where we recall that Bloch included the minus sign in \cite[p.\ 112]{bloch-compositio}. 
\end{proof}

As mentioned above, since $k$ is algebraically closed,  (\ref{ax:homolog=alg-div}) holds true.
Lemma \ref{lem:AiX} thus implies 
$$
\Grifftilde^i(X)_{\Z_\ell}=\Griff^i(X)\otimes_{\Z}{\Z_\ell}
$$
is the group of homologically trivial codimension $i$ cycles with coefficients in $\Z_\ell$ modulo algebraic equivalence.
In particular, $\Grifftilde^i(X)[\ell^\infty]=\Griff^i(X)[\ell^\infty]$ is the group of classes in $\Griff^i(X)$ that are annihilated by some power of $\ell$.

\begin{proposition} \label{prop:lambda=lambda_tr}
Let $k$ be an algebraically closed field and let $X$ be a smooth projective variety over $k$.
The map
$$
\lambda^i_{tr}: \Griff^i(X)[\ell^\infty]\longrightarrow  \frac{H^{2i-1} (X,\Q_\ell/\Z_\ell(i))}{N^{i-1}H^{2i-1} (X,\Q_\ell(i) )}
$$
constructed in Section  \ref{subsec:AJ_tr} is induced by Bloch's map in (\ref{eq:lambda}) and hence agrees with the transcendental Abel--Jacobi map if $k=\C$.
\end{proposition}
\begin{proof}
If $k=\C$, then Bloch's map agrees with the Abel--Jacobi map on torsion cycles, see \cite[Proposition 3.7]{bloch-compositio}.
It thus suffices to show that $\lambda_{tr}^i$ from Section  \ref{subsec:AJ_tr} is induced by Bloch's map in (\ref{eq:lambda}).
For this, let $z\in \bigoplus_{x\in X^{(i)}}[x]\Z_\ell$ be a homologically trivial cycle.
Then $\del \alpha=z$ for some $\alpha\in H^{2i-1}(F_{i-1}X,\Z_\ell(i))$.
Assume that $z$ is $\ell^r$-torsion modulo algebraic equivalence.
As in Section \ref{subsec:AJ_tr}, we find classes $\beta\in H^{2i-1}(X,\Z_\ell(i))$ and  $\xi\in \bigoplus_{x\in X^{(i-1)}}H^1(x,\Z_\ell(i))$ with
$$
\beta=\ell^r \cdot \alpha-\iota_\ast \xi\in F^i  H^{2i-1}(F_{i-1}X,\Z_\ell(i)) .
$$
In particular, $\del\circ  \iota_\ast (\xi) /\ell^r=z$ and so $\psi_r([z])=[\xi/\ell^r]$, where $\psi_r$ is the isomorphism from Lemma \ref{lem:snake}.
By our construction of Bloch's map, we thus find
$$
\lambda([z])=\lambda'([z])=-\iota_\ast\xi/\ell^r \in F^i  H^{2i-1}(F_{i-1}X,\Q_\ell/\Z_\ell(i))\cong H^{2i-1}(X,\Q_\ell/\Z_\ell(i)).
$$
On the other hand,
$$
\lambda^i_{tr}([z])=[\beta/\ell^r] \in H^{2i-1}(X,\Q_\ell/\Z_\ell(i))/N^{i-1}H^{2i-1}(X,\Q_\ell(i))
$$
by our construction of $\lambda^i_{tr}$ in Section \ref{subsec:AJ_tr}.
The result thus follows from the fact that
$$
\beta/\ell^r+\iota_\ast\xi/\ell^r =\alpha=0\in  H^{2i-1}(F_{i-1}X,\Q_\ell/\Z_\ell(i)) ,
$$
because $\alpha$ is an integral class.
This proves the proposition.
\end{proof}

\section{Proof of main results ($\ell$-adic)}  

\subsection{$\ell$-adic twisted Borel--Moore cohomology} \label{subsec:notation-proetal}
Fix a field $k$ and a prime $\ell$ that is invertible in $k$.
Let $\VV$ be the category whose objects are separated schemes of finite type over $k$ and such that the morphisms are given by open immersions of schemes of the same dimension.
This is a constructible category of Noetherian schemes as in Definition \ref{def:V}.
Let $\mathcal A\subset \Mod_{\Z_\ell}$ be the full subcategory with objects  $\Z_\ell,\Q_\ell,\Q_\ell/\Z_\ell$ and $\Z/\ell^r$ for all $r\geq 1$.
By Proposition \ref{prop:proetale-coho-arbitrary-field}, $\ell$-adic pro-\'etale Borel--Moore cohomology $H^\ast(-,A(n))$ as defined in (\ref{eq:H^i(X,mu-ell)})--(\ref{eq:H^i(X,Q/Z-ell)}) is a  twisted Borel--Moore cohomology theory on $\VV$ with coefficients in $\mathcal A$ which is $\ell$-adic if $k$ is perfect, see Definitions \ref{def:Borel--Moore-cohomology} and \ref{def:Borel--Moore-cohomology-l-adic}.
In particular, all results from Section \ref{sec:def} and \ref{sec:comparison-thm} hold true in this set-up.
Here we recall that for $i\geq 1$:
\begin{align} \label{eq:Ai-main}
A^i(X)_{\Z_\ell}=\frac{\CH^i(X)_{\Z_\ell}}{N^{i-1}\CH^i(X)_{\Z_\ell}} \ \ \text{and}\ \ N^jA_0^i(X)_{\Z_\ell}=N^jA^i(X)_{\Z_\ell}=\frac{N^j\CH^i(X)_{\Z_\ell}}{N^{i-1}\CH^i(X)_{\Z_\ell} }
\end{align} 
for $0\leq j\leq i-1$, see Lemma \ref{lem:Ai-vs-Ni-1}.  
Moreover, Lemma \ref{lem:AiX} implies that
 $\Grifftilde^i(X)_{\Z_\ell} = \Griff ^i(X)_{\Z_\ell}  $ if $k$ is algebraically closed and $\Grifftilde^i(X)_{\Z_\ell} = \ker(\cl_X^i)\subset \CH^i(X)_{\Z_\ell}$ if $k$ is the perfect closure of a finitely generated field.
 
 \begin{lemma} \label{lem:cl=Jannsen}
 In the above notation,  assume in addition that $X$ is smooth and equi-dimensional.
 Then the cycle class map 
 $$
 \cl_X^i:\CH^i(X)_{\Z_\ell}\longrightarrow H^{2i}(X,\Z_\ell(i)),
 $$ 
 constructed in (\ref{eq:cli}) coincides with Jannsen's cycle class map in continuous \'etale cohomology from \cite{jannsen}.
 \end{lemma}
 \begin{proof}
By Lemma \ref{lem:Chow-inseparable-extension} and the topological invariance of the pro-\'etale topos (see \cite[Lemma 5.4.2]{BS}), we may replace $k$ by its perfect closure and assume that $k$ is perfect.
 Let $X$ be a smooth variety over $k$.
 By Lemma \ref{lem:proet:coho=homo},  $H^\ast(X,A(n))$ agrees with the corresponding continuous \'etale cohomology group.
 The cycle class map in (\ref{eq:cli}) is defined via the Gysin pushforward, where one uses excision to reduce to the case of a cycle whose support is smooth.
Our claim thus follows from \cite[Remark 3.24]{jannsen}.
 \end{proof}

\subsection{Proof of Theorem \ref{thm:singular}}

\begin{proof}[Proof of Theorem \ref{thm:singular}]
We use the notation from Section \ref{subsec:notation-proetal}
and claim that it  suffices to prove Theorem \ref{thm:singular} after replacing $k$ by its perfect closure $k^{per}$.
Indeed, this does not change $\ell$-adic Chow groups by Lemma \ref{lem:Chow-inseparable-extension} and it does not change the (pro-)\'etale topos (see \cite[Lemma 5.4.2]{BS}), so that  $H^\ast(-,A(n))$ remains unchanged by passing from $k$ to $k^{per}$.

We may and will thus from now on assume that $k$ is perfect, so that $H^\ast(-,A(n))$ is an $\ell$-adic   twisted Borel--Moore cohomology theory on $\VV$ by Proposition \ref{prop:proetale-coho-arbitrary-field}.
For any $X\in \VV$, we thus get a cycle class map 
$$
\cl^i_X:\CH^i(X)_{\Z_\ell}:=\CH^i(X)\otimes_{\Z}{\Z_\ell}\longrightarrow H^{2i}(X,\Z_\ell(i)),
$$
constructed in Section \ref{subsec:cli}.
If $X$ is smooth and equi-dimensional, then, by Lemma \ref{lem:proet:coho=homo}, $H^i(X,\Z_\ell(n))\cong H^i_{cont}(X_\et,\Z_\ell(n))$ agrees with Jannsen's $\ell$-adic continuous \'etale cohomology groups, see Section \ref{subsec:jannsen}. 
It follows from the construction of $\cl^i_X$ via the Gysin sequence (see (\ref{ax:Gysin})) that if $X$ is a smooth variety, then $\cl^i_X$ agrees with Jannsen's cycle class map (see Lemma \ref{lem:cl=Jannsen}).

Recall from Definitions \ref{def:Ai} and \ref{def:Griff} the groups $A^i(X)_{\Z_\ell}$ and $A_0^i(X)_{\Z_\ell}$ and recall the description from (\ref{eq:Ai-main}) (that we used as definition in the introduction).

Item (\ref{item:coker-proet}) in Theorem \ref{thm:singular} is then a consequence of Theorem \ref{thm:IHC} and Proposition \ref{prop:Griff}.

By Section \ref{subsec:AJ_tr}, there is a map
$$
\lambda_{tr}^i:\Grifftilde^i(X)[\ell^\infty]\longrightarrow  H^{2i-1}(X,\Q_\ell/\Z_\ell(i))/N^{i-1}H^{2i-1}(X,\Q_\ell(i) ) ,
$$
where $N^jH^i(X,A(n)):=\ker(H^i(X,A(n))\to H^i(F_{j-1}X,A(n)))$.
If $k$ is algebraically closed and $X$ is smooth projective, then this map agrees by Proposition \ref{prop:lambda=lambda_tr} with Bloch's transcendental Abel--Jacobi mapping on torsion cycles from \cite{bloch-compositio} (cf.\ Section \ref{sec:lambda}).
Item (\ref{item:lambda-thm:singular}) thus follows from Theorem \ref{thm:Griff_tors} and Proposition \ref{prop:im-lambda}.
This concludes the proof of Theorem \ref{thm:singular}.
\end{proof}

\subsection{Proof of Theorem \ref{thm:green-ell-adic-arbitrary-field}}

The following lemma shows that in our set-up, the result of Theorem \ref{thm:ker-lambda} holds true whenever $k$ contains all $\ell$-roots of unity.

\begin{lemma} \label{lem:voev-torsionfree}
In the notation of Section \ref{subsec:notation-proetal},  the following holds.
Then for any $X\in \VV$ and any $i$ and $n$, $H^i(F_0X,\Z_\ell(n))$ is torsion-free if one of the following conditions holds:
\begin{enumerate}
\item $n=i-1$; \label{item:n=i-1}
\item $k$ contains all $\ell$-power roots of unity.
\end{enumerate}
\end{lemma}
\begin{proof}
By additivity of the cohomology functor (see Lemma \ref{lem:XsqcupY}), we may assume that $X$ is irreducible with generic point $\eta_X\in X$.
If $k$ contains all $\ell$-power roots of unity, $H^i(F_0X,\Z_\ell(n))\cong H^i(F_0X,\Z_\ell(i-1))$ for all $i$ and $n$.
It thus suffices to prove the lemma under assumption (\ref{item:n=i-1}).
Since $X$ is irreducible,  $H^i(F_0X,\Z_\ell(i-1))= H^i(\eta_X,\Z_\ell(i-1))$ and so the claim follows from Remark \ref{rem:bloch-kato} and Voevodsky's proof of the Bloch--Kato conjecture \cite{Voe:Bloch-Kato}.
\end{proof}

By (\ref{def:bar-lambda_j,tr}) there is a higher Abel--Jacobi mapping
\begin{align}\label{def:bar-lambda_j,tr-2}
\bar \lambda_{j,tr}^i:N^j\Grifftilde^i(X)[\ell^{\infty}]\longrightarrow \overline{J}^i_{j,tr}(X)[\ell^{\infty}]
\end{align}
where $N^\ast$ denotes the coniveau filtration on $\Grifftilde^i(X)[\ell^{\infty}]$ from Section \ref{subsec:coniveau}.

\begin{theorem} \label{thm:green-ell-adic-arbitrary-field-main}
Let $X$ be a separated scheme of finite type over $k$.
Let $i\geq 2$ and assume that one of the following holds:
\begin{enumerate}
\item  $k$ contains all  $\ell$-power roots of unity, or
\item $i=2$.
\end{enumerate} 
Then for all $0\leq j\leq i-2$, we have
$$
N^{j+1}\Grifftilde^i(X)[\ell^{\infty}]=\ker\left(
\bar \lambda_{j,tr}^i:N^j\Grifftilde^i(X)[\ell^{\infty}]\to  \overline{J}^i_{j,tr}(X)[\ell^{\infty}] \right) .
$$
\end{theorem}
\begin{proof}
We aim to apply  Theorem \ref{thm:ker-lambda}.
To this end we need to ensure that for any closed subscheme $Z\subset X$,  $H^{2(i-j)-1}(F_0Z,\Z_\ell(i-j))$ is torsion free.
By Lemma  \ref{lem:voev-torsionfree}, this condition is satisfied if $k$ contains all  $\ell$-power roots of unity, or if $i=2$ and $j=0$.
This concludes the proof.
\end{proof}

We are now in position to proof Theorem \ref{thm:green-ell-adic-arbitrary-field}, which follows from the following slightly stronger result.

\begin{theorem} \label{thm:green-ell-adic-arbitrary-field-body}
Let $X$ be a separated scheme of finite type over a field $k$ and let $\ell$ be a prime invertible in $k$.
Let $i\geq 2$ and assume that one of the following holds:
\begin{enumerate}
\item  $k$ contains all  $\ell$-power roots of unity, or
\item $i=2$.
\end{enumerate} 
Then for all $0\leq j\leq i-2$, we have
$$
N^{j+1}\CH^i(X)[\ell^\infty]=\ker \left(\overline \lambda_{j,tr}^i:N^j\CH^i(X)[\ell^\infty]\to \overline{J}^i_{j,tr}(X)[\ell^\infty] \right).
$$ 
\end{theorem}

\begin{proof} 
By (\ref{eq:Ai-main}), $$
\Grifftilde^i(X)[\ell^{\infty}]=N^0\CH^i(X)_{\Z_\ell}/N^{i-1}\CH^i(X)_{\Z_\ell}.
$$
The higher transcendental Abel--Jacobi mapping from (\ref{def:bar-lambda_j,tr}) thus yields for $0\leq j\leq i-2$ mappings
$$
\bar \lambda_{j,tr}^i:N^j\CH^i(X)[\ell^\infty]\longrightarrow \overline{J}^i_{j,tr}(X)[\ell^\infty]=\lim_{\substack{\longrightarrow \\ Z\subset X}} \frac{H^{2i-2j-1}(Z,\Q_\ell/\Z_\ell(i-j))}{N^1H^{2i-2j-1}(Z,\Q_\ell(i-j) )}
$$
where $Z\subset X$ runs through all closed subschemes with $j=\dim X-\dim Z$.
Theorem \ref{thm:green-ell-adic-arbitrary-field-main} then implies that for $0\leq j\leq i-2$, the kernel of the above map is given by $N^{j+1}\CH^i(X)[\ell^\infty]$, as we want.
\end{proof}

\subsection{Applications of Theorem \ref{thm:green-ell-adic-arbitrary-field} and \ref{thm:green-ell-adic-arbitrary-field-body}}

The simplest (non-trivial) consequence of Theorem \ref{thm:green-ell-adic-arbitrary-field-body} is as follows. 

\begin{corollary} \label{cor:lambda^2}
Let $X$ be a smooth equi-dimensional algebraic  scheme over a finitely generated field $k$.
Let $\ell$ be a prime invertible in $k$ and let $\CH^i_{0}(X)[\ell^\infty]$ denote the group of $\ell$-power torsion cycles with trivial cycle class in Jannsen's continuous $\ell$-adic \'etale cohomology.
Then there is a canonical injection
$$
\lambda_{tr}^2:\CH^2_{0}(X)[\ell^\infty]\hookrightarrow H^3_{cont}(X_{\et},\Q_\ell/\Z_\ell(2))/N^1H^3_{cont}(X_{\et},\Q_\ell(2)) .
$$
\end{corollary}

Corollary \ref{cor:lambda^2} should be compared to a result of Merkurjev--Suslin \cite[\S 18]{MS}, who showed that Bloch's Abel--Jacobi mapping   on $\ell$-power  torsion cycles on smooth projective varieties  over algebraically closed fields \cite{bloch-compositio} is injective on codimension 2 cycles.  
Corollary \ref{cor:lambda^2} has previously been proven in the particular case where $k=\mathbb F_q$ is a finite field and $X$ is smooth projective in \cite[Th\'eor\`eme 4]{CTSS} (in fact, loc. cit. proves that $\CH^2_{0}(X)[\ell^\infty]=0$ in this case; this also follows from our set-up, see Proposition \ref{prop:im-lambda} and note that $H^3_{cont}(X_{\et},\Q_\ell(2))=H^3 (X_{\et},\Q_\ell(2))=0$ for weight reasons, cf.\ \cite[p.\ 780-781]{CTSS}).

\begin{proof}[Proof of Corollary \ref{cor:lambda^2}]
By the same argument as in the proof of Theorem \ref{thm:singular}, we may replace $k$ by its perfect closure and hence assume that $k$ is the perfect closure of a finitely generated field.
The result then follows from Theorem \ref{thm:green-ell-adic-arbitrary-field} (or Theorem \ref{thm:Griff_tors}) together with the fact that $A^2_0(X)_{\Z_\ell}=\CH_0^2(X)_{\Z_\ell} $ if the ground field $k$ is (the perfect closure of a) finitely generated, see Proposition \ref{prop:proetale-coho-arbitrary-field} and Lemma \ref{lem:AiX}.
\end{proof}

Let $X$ and $Y$ be smooth projective equi-dimensional $k$-schemes.
A cycle $\Gamma \in \CH^{\dim X}(X\times Y)$  yields actions on Chow groups that are compatible with the cycle class maps in continuous \'etale cohomology.
Hence we get an action $\Gamma_\ast : N^0\CH^i(X)[\ell^\infty] \to N^0\CH^i(Y)[\ell^\infty]$.
Similarly, there are actions $\Gamma_\ast : H^i_{cont}(X_\et,A(n))\to H^i_{cont}(Y_\et,A(n))$  for $A\in \{ \Q_\ell/\Z_\ell,\Q_\ell\}$.
The latter respect Grothendieck's coniveau  filtration $N^\ast$, as can be checked with the help of a moving lemma (see e.g.\ \cite[Theorem 2.13]{levine}).
We conclude that correspondences act on source and target of the map
$$
\lambda_{tr}^i:N^0\CH^i(X)[\ell^\infty]\longrightarrow  H^{2i-1}_{cont}(X,\Q_\ell/\Z_\ell(i))/N^{i-1}H^{2i-1}_{cont}(X,\Q_\ell (i))
$$
induced from the map in Section \ref{subsec:AJ_tr}  (cf.\ Lemma \ref{lem:Ai-vs-Ni-1}). 
We will show in  \cite{Sch-preparation} that these actions are compatible with the map $\lambda_{tr}^i$.
The case $i=2$ is simpler, and we give a direct proof in the following lemma.

\begin{lemma} \label{lem:compatible-lambda^2}
Let $X,Y$ be smooth projective equi-dimensional $k$-schemes and let $\Gamma \in \CH^{\dim X}(X\times Y)$.
Then the following diagram commutes:
$$
\xymatrix{
A_0^2(X)[\ell^\infty] \ar[d]^{\Gamma_\ast}\ar[rr]^-{\lambda_{tr}^2}& & H^{3}_{cont}(X,\Q_\ell/\Z_\ell(2))/N^{1}H^{3}_{cont}(X,\Q_\ell (2)) \ar[d]^{\Gamma_\ast} \\
A_0^{2}(Y)[\ell^\infty]  \ar[rr]^-{\lambda_{tr}^{2}}&&  H^{3}_{cont}(Y,\Q_\ell/\Z_\ell(2))/ N^{1}H^{3}_{cont}(Y,\Q_\ell (2)) .
}
$$
\end{lemma}
\begin{proof} 
By the description of $\lambda_{tr}^i$ from Lemma \ref{lem:lambda_tr-alternative}, it suffices to show that the isomorphism
$$
A_0^2(X)_{\Z_\ell}\stackrel{\cong}\longrightarrow H^3_{nr}(X,\Z_\ell(2))/H^3(X,\Z_\ell(2))
$$
from Proposition \ref{prop:Griff} 
is compatible with the action of correspondences.
Since $X$ is smooth and equi-dimensional,  the Gersten conjecture \cite{BO,CTHK} identifies $H^3_{nr}(X,\Z_\ell(2))$ with the global sections of the corresponding Bloch--Ogus sheaf, associated to $U\mapsto H^3_{cont}(U_\et,\Z_\ell(2))$.
This description makes it easy to define an action on unramified cohomology, cf.\  \cite[\S 9,  Appendice]{CTV},  (an action in a more general setting has been constructed by Rost \cite{rost}).
Using this description one readily checks that the above isomorphism is compatible with the action of cycles, which concludes the lemma.
\end{proof}

Corollary \ref{cor:lambda^2} implies for instance the Rost nilpotence conjecture for surfaces up to inverting the exponential characteristic, originally due to Gille \cite{gille-inventiones,gille-crelle} and with an alternative proof due to Rosenschon--Sawant \cite{RS}:

\begin{corollary}\label{cor:rost-1}
Let $X$ be a smooth projective equi-dimensional scheme over a field $k$ with base change $\bar X=X\times_k\bar k$, where $\bar k$ denotes an algebraic closure.
Let $\Gamma\in \CH^{\dim X}(X\times X)$ 
and assume that the base change $\bar \Gamma=\Gamma\times_k\bar k$  acts trivially on $H^i(\bar X_{\et},\Q_\ell/\Z_\ell(2))$ for $i\leq 3$.
Then the action 
$$
\Gamma_\ast^{\circ N}:\CH^2(X)[\ell^\infty]\longrightarrow \CH^2(X)[\ell^\infty]
$$  
is zero for $N\geq 10$.
\end{corollary}
\begin{proof}
A straightforward limit argument reduces us to the case where $k$ is finitely generated.
(This uses that \'etale cohomology does not change under algebraically closed field extensions.)
By Lemma \ref{lem:compatible-lambda^2} and Corollary \ref{cor:lambda^2}, it thus suffices to show that  
$ 
\Gamma ^{\circ 5}
$ 
acts trivially on $H^4_{cont}(X_\et,\Z_\ell(2))[\ell^\infty]$ and on $H^3_{cont}(X_\et,\Q_\ell/\Z_\ell(2))$.
The Bockstein sequence yields a canonical surjection $$H^3_{cont}(X_\et,\Q_\ell/\Z_\ell(2))\to H^4_{cont}(X_\et,\Z_\ell(2))[\ell^\infty]$$ that is compatible with the action of correspondences.
It thus suffices to show that 
$ 
\Gamma_\ast^{\circ 5}
$ 
acts trivially on $H^3_{cont}(X_\et,\Q_\ell/\Z_\ell(2))$. 
Note that $H^3_{cont}(X_\et,\Q_\ell/\Z_\ell(2))=\colim_ r H^3_{cont}(X_\et,\mu_{\ell^r}^{\otimes 2})$ and   $H^3_{cont}(X_\et,\mu_{\ell^r}^{\otimes 2})=H^3 (X_\et,\mu_{\ell^r}^{\otimes 2})$, cf.\ \cite[(3.1)]{jannsen}. 
The  assertion in question thus follows from the fact that the Hochschild--Serre spectral sequence for  \'etale cohomology (see \cite[p. 105,  III.2.20]{milne}) is compatible with the action of correspondences. 
(In the last step, we are implicitly working with the separable closure of $k$ and not with the algebraic closure; this is possible because neither Chow groups nor \'etale cohomology change by purely inseparable field extensions, see Lemma \ref{lem:Chow-inseparable-extension}.)
\end{proof}

Recall that the exponential characteristic $e$ of a field $k$ is $1$ if $\operatorname{char} k=0$ and $p$ if $p=\operatorname{char} k>0$.

\begin{corollary} \label{cor:rost-2}
Let $S$ be a smooth projective scheme of pure dimension $2$ over a field $k$ of exponential characteristic $e$.
Let $\Gamma\in \CH^2(S\times_k S)$ be a correspondence with base change $\bar \Gamma\in \CH^2(S_{\bar k}\times_{\bar k} S_{\bar k})$  to the algebraic closure $\bar k$ of $k$.
Assume that $\bar \Gamma $ is  torsion and homologically trivial in $\ell$-adic \'etale cohomology for any prime $\ell$ invertible in $\bar k$.
Then up to inverting $e$, the composition $\Gamma^{\circ N} $ is zero for $N\geq 11$.
\end{corollary}
\begin{proof}
A standard norm argument shows that $\Gamma$ is torsion.
By the Chinese remainder theorem, we may assume that $\Gamma$ is $\ell^r$-torsion for some integer $r\geq 0$ and some prime $\ell$ invertible in $k$.
Let $X:=S\times S$.
The assumptions imply that there is a correspondence $\Omega\in \CH^4(X\times X)[\ell^\infty]$ which is homologically trivial over $\bar k$ and such that $\Omega_\ast^{\circ N}(\Gamma)=\Gamma^{\circ N+1}$.
The assertion thus follows from   Corollary \ref{cor:rost-1}.
\end{proof}

\begin{remark}
Up to inverting the exponential characteristic, Corollary \ref{cor:rost-2} is  slightly stronger than the original conjecture of Rost for surfaces, proven in \cite{gille-inventiones,gille-crelle,RS}.
Indeed, we are only asking that $\bar \Gamma$ is torsion and homologically trivial, while the original formulation asks that $ \Gamma $ is rationally equivalent to $0$ over $\bar k$  (or equivalently over some field extension of $k$).
\end{remark}

\section{Proof of main results over $\C$}
\subsection{Integral twisted Borel--Moore cohomology} \label{subsec:notation-betti}
Let $\VV$ be the category whose objects are separated schemes of finite type over $\C$ and such that the morphisms are given by open immersions of schemes of the same dimension.
This is a constructible category of Noetherian schemes as in Definition \ref{def:V}.
Let $\mathcal A:=\Mod_\Z$ and define for $X\in \VV$ and $A\in \mathcal A$,
$$
H^i(X,A(n)):=H^{BM}_{2d_X-i}(X_{\cx},A(d_X-n)),
$$
where $d_X:=\dim X$ and  $H^{BM}_\ast$ denotes Borel--Moore homology and $X_{\cx}$ denotes the analytic space that underlies $X$ and $A(n):=A\otimes_{\Z}(2\pi i)^n \Z$ denotes the $n$-th Tate twist of $A$.
By Proposition \ref{prop:Betti-coho}, $H^\ast$ defines an integral twisted Borel--Moore cohomology theory that is adapted to algebraic equivalence, cf.\ Definition \ref{def:Borel--Moore-integral}.
It follows that all results from Sections \ref{sec:def} and \ref{sec:comparison-thm} hold true in the above set-up if we formally make the replacements
$ \Z_\ell \rightsquigarrow \Z$, $\Q_\ell \rightsquigarrow \Q$,  $\ell^r \rightsquigarrow r$, and $[\ell^\infty]  \rightsquigarrow { }_{\tors}$.


\subsection{Proof of Theorem \ref{thm:main-2-intro}}

\begin{proof}[Proof of Theorems \ref{thm:main-2-intro}]
We use the notation from Section \ref{subsec:notation-betti}.
Performing the aforementioned formal replacements $ \Z_\ell \rightsquigarrow \Z$, $\Q_\ell \rightsquigarrow \Q$,  $\ell^r \rightsquigarrow r$, and $[\ell^\infty]  \rightsquigarrow { }_{\tors}$, Lemma \ref{lem:AiX} shows that $\Grifftilde^i(X)_{\Z}=\Griff ^i(X)$ is the group of homologically trivial cycles modulo algebraic equivalence.
The arguments in Section \ref{subsec:AJ_tr} yield a map
\begin{align} \label{def:lambda-betti-main}
\lambda_{tr}^{i}:\Griff^i(X)_{\tors}\longrightarrow J^i_{tr}(X)_{\tors}:= H^{2i-1}(X,\Q /\Z (i))/N^{i-1}H^{2i-1}(X,\Q (i) ),
\end{align}
where $N^jH^i(X,A(n)):=\ker(H^i(X,A(n))\to H^i(F_{j-1}X,A(n)))$.
If $X$ is smooth projective, we claim that this map agrees with Griffiths transcendental Abel--Jacobi map restricted to torsion cycles.
By the Chinese remainder theorem, it suffices to show this for classes that are $\ell$-power torsion for some prime $\ell$.
In this case our map identifies by Proposition \ref{prop:lambda=lambda_tr} with Bloch's map, which in turn identifies with Griffiths map by \cite[Proposition 3.7]{bloch-compositio}.

Theorem \ref{thm:main-2-intro} follows now as above from Theorem \ref{thm:IHC}, Proposition \ref{prop:Griff}, Theorem \ref{thm:Griff_tors}, and Proposition \ref{prop:im-lambda} (after performing the  formal replacements $ \Z_\ell \rightsquigarrow \Z$, $\Q_\ell \rightsquigarrow \Q$,  $\ell^r \rightsquigarrow r$, and $[\ell^\infty]  \rightsquigarrow { }_{\tors}$).
This concludes the proof of Theorem \ref{thm:main-2-intro}.
\end{proof}

\subsection{Proof of Theorems \ref{thm:coniveau1} and \ref{thm:green}, and Corollaries \ref{cor:n-torsion} and \ref{cor:jannsen}}

\begin{lemma} \label{lem:voev-torsionfree-betti}
In the notation of Section \ref{subsec:notation-betti},   for any $X\in \VV$ and any $i$ and $n$, $H^i(F_0X,\Z(n))$ is torsion-free.
\end{lemma}
\begin{proof}
By additivity of the cohomology functor (see Lemma \ref{lem:XsqcupY}), we may assume that $X$ is irreducible with generic point $\eta_X\in X$.
By definition, $H^i(F_0X,\Z(n))\cong H^i(F_0X,\Z(i-1))$ for all $i$ and $n$.
Since $X$ is irreducible, the latter coincides with $ H^i(\eta_X,\Z(i-1))$ and so the claim follows from Remark \ref{rem:bloch-kato} and Voevodsky's proof of the Bloch--Kato conjecture \cite{Voe:Bloch-Kato}.
\end{proof}

\begin{proof}[Proof of Theorem \ref{thm:green}] 
Theorem \ref{thm:green} follows with help of Lemma \ref{lem:voev-torsionfree-betti} via the same arguments as in the proof of Theorem \ref{thm:green-ell-adic-arbitrary-field}.
This requires as in Theorem \ref{thm:main-2-intro} the formal replacements $ \Z_\ell \rightsquigarrow \Z$, $\Q_\ell \rightsquigarrow \Q$,  $\ell^r \rightsquigarrow r$, and $[\ell^\infty]  \rightsquigarrow { }_{\tors}$ in Section \ref{sec:comparison-thm}.
\end{proof}

\begin{proof}[Proof of Theorem \ref{thm:coniveau1}] 
This is a consequence of Theorem \ref{thm:green} together with the fact that $\lambda_{tr}^i$ factorizes for smooth projective varieties through Bloch's Abel--Jacobi map for torsion cycles, which in turn agrees with the Abel--Jacobi invariants due to Griffiths in that case, see Proposition \ref{prop:lambda=lambda_tr}.
The assumption $i\geq 2$ is needed, because the assertion in Theorem \ref{thm:green} is empty for $i=1$.
\end{proof}

\begin{proof}[Proof of Corollary \ref{cor:n-torsion}] 
By Theorem \ref{thm:coniveau1}, the $n$-torsion of $N^0\CH^i(X)_{\tors}/N^1\CH^i(X)_{\tors}$ injects into the $n$-torsion of a quotient of $H^{2i-1}(X,\Q/\Z)$, hence is finite.
Moreover, the cycle class map yields an injection of $ \CH^i(X)/N^0\CH^i(X)$
into $H^{2i}(X,\Z)$ and so the $n$-torsion subgroup of $\CH^i(X)_{\tors}/N^0\CH^i(X)_{\tors}$ must be finite as well.
Altogether we conclude that the $n$-torsion subgroup of $\CH^i(X)_{\tors}/N^1\CH^i(X)_{\tors}$ is finite, as claimed.
\end{proof}

\begin{proof}[Proof of Corollary \ref{cor:jannsen}] 
This is an immediate consequence of Theorem \ref{thm:coniveau1}.
\end{proof}

\subsection{Applications of Theorem \ref{thm:green}}
For a complex algebraic scheme $X$, we recall that the coniveau filtration $N^j$ on $\Griff^i(X)$ is defined by saying that a cycle $z\in \Griff^i(X)$ lies in $N^j$ if and only if there is a closed subset $Z\subset X$ with $j=\dim X-\dim Z$ and a homologically trivial cycle $z'$ on $Z$ such that $z$ agrees with the pushforward of $z'$, cf.\  Definition \ref{def:NjCHi}. 
This yields a finite decreasing filtration on $\Griff^i(X)$ with $N^{i-1}=0$, cf.\ Lemmas \ref{lem:Ai-vs-Ni-1} and \ref{lem:AiX}.

\begin{corollary} \label{cor:green-Nj}
Let $X$ be a separated scheme of finite type over $\C$.
Then
$$
\Griff^i(X)_{\tors}=N^1\Griff^i(X)_{\tors}=\dots =N^j\Griff^i(X)_{\tors}\ \ \ \text{for all $j\leq 2i-1-\dim X$}.
$$
\end{corollary}

\begin{proof} 
We use the notation from Section \ref{subsec:notation-betti}.
Proposition \ref{prop:im-lambda} implies that for any separated scheme $X$ of finite type over $\C$, we have
$$
\im(\bar \lambda_{tr}^i)=\frac{N^{i-1}H^{2i-1}(X,\Q /\Z (i))_{div}}{N^{1}H^{2i-1}(X,\Q (i) )} ,
$$
where $H^{2i-1}(X,\Q /\Z (i))_{div}=\im( H^{2i-1}(X,\Q (i))\to H^{2i-1}(X,\Q /\Z (i)))$.
Since affine varieties have no cohomology in degrees greater than their dimension, and because $H^i(X,A(n))$ depends only on the underlying reduced scheme and agrees with singular cohomology if $X$ is smooth, we find that $N^{1}H^{2i-1}(X,\Q (i) )= H^{2i-1}(X,\Q (i) )$ for $2i-1>\dim X$ and so $\im(\bar \lambda_{tr}^i)=0$ for $2i-1>\dim X$.
It follows that $\im(\bar \lambda_{j,tr}^i)=0$ for $2i-2j-1>\dim X-j$, i.e.\ for $j<2i-1-\dim X$.
Theorem \ref{thm:green} thus implies that $N^{j+1}\Griff^i(X)_{\tors}=N^j\Griff^i(X)_{\tors}$ for $j<2i-1-\dim X$, which proves Corollary \ref{cor:green-Nj}.
\end{proof}

 \subsection{Algebraic cycles and traditional unramified cohomology in arbitrary degree}

For $j\geq m$, there is a canonical restriction map
\begin{align} \label{eq:H_j,nr-to-H_m,nr}
H^i_{j,nr}(X,A(n))\longrightarrow H^i_{m,nr}(X,A(n)) .
\end{align}
Recall from Section \ref{sec:def} that we denote its image by $F^{j+1}H^i_{m,nr}(X,A(n))$.  
 We describe the effect of applying the restriction map (\ref{eq:H_j,nr-to-H_m,nr}) to Theorem \ref{thm:main-2-intro} next.

\begin{corollary} \label{cor:main-intro}
Let $X$ be a separated scheme of finite type over $\C$.
Then for any $0\leq j\leq i-2$, there are canonical exact sequences
\begin{align*} 
&\lim_{ \longrightarrow }  Z^{i-j}(Z)_{\tors}\longrightarrow Z^{i}(X)_{\tors}\longrightarrow \frac{F^{i-1}H^{2i-1}_{j-1,nr}(X,\Q/\Z(i) )}{F^{i-1}H^{2i-1}_{j-1,nr}(X,\Q(i))}\longrightarrow 0 ,
\\ 
\lim_{ \longrightarrow }  \Griff^{i-j}(Z)\longrightarrow &\Griff^i(X) \longrightarrow \frac{F^{i-1}H^{2i-1}_{j-1,nr}(X,\Z(i))}{H^{2i-1}(X,\Z(i)) }\longrightarrow 0 ,
\\  
&\lim_{ \longrightarrow }  \mathcal T^{i-j}(Z)\longrightarrow \mathcal T^{i}(X)  \longrightarrow \frac{F^{i-2}H^{2i-2}_{j-1,nr}(X,\Q /\Z(i)) }{ G^{i} F^{i-2} H^{2i-2}_{j-1,nr}(X,\Q/ \Z (i))}\longrightarrow 0,
\end{align*}
where in the direct limits, $Z\subset X$ runs through all reduced closed subschemes of $X$ with $\dim (X)-\dim(Z)=j$. 
\end{corollary} 

The above corollary is particularly interesting for $j=1$.
In this case the refined unramified cohomology groups above agree with traditional unramified cohomology  $H^i_{nr}(X,A(n))=H^i_{0,nr}(X,A(n))$.
The corollary then identifies certain graded pieces of traditional unramified cohomology with certain birationally invariant quotients of the above cycle groups.
In particular, 
non-trivial elements in certain pieces of the $F^\ast$ filtration on traditional unramified cohomology $H^i_{nr}(X,A(n))=H^i_{0,nr}(X,A(n))$ detect exactly those cycles on $X$ that are not supported in codimension 1 in the sense that they are not pushforwards of the respective cycle groups on some divisor on $X$.
This  improves some results obtained independently by Ma in \cite{ma2}.

\begin{proof} [Proof of Corollary \ref{cor:main-intro}] 
The corollary follows from Theorem \ref{thm:main-2-intro} and Corollary \ref{cor:restr-H_nr}.
We give some details for $\Griff^i(X)$; the other cases are similar.

We use the same notation as in the proof of Theorem \ref{thm:main-2-intro} and fix the integral twisted Borel--Moore cohomology theory $H^\ast(-,A(n))$ on separated schemes of finite type over $\C$ from Proposition \ref{prop:Betti-coho}.
By Theorem \ref{thm:main-2-intro}, there is a canonical isomorphism
$$
\Griff^i(X)\cong H^{2i-1}_{i-2,nr}(X,\Z(i))/H^{2i-1}(X,\Z(i)).
$$
By Corollary \ref{cor:restr-H_nr},   for any $0\leq j\leq i-1$,  there is a canonical exact sequence
$$
\lim_{\longrightarrow}H^{2(i-j)-1}_{i-j-2,nr}(Z,\Z(i-j))\stackrel{\iota_\ast}\longrightarrow H^{2i-1}_{i-2,nr}(X,\Z(i))\longrightarrow F^{i-1} H^{2i-1}_{j-1,nr}(X,\Z(i))\longrightarrow 0,
$$
where the direct limit runs through all closed reduced subschemes $Z\subset X$ of dimension $\dim Z=\dim X-j$.
Here the first map is induced by the pushforward map with respect to  $Z\hookrightarrow X$ and the second map is the canonical restriction map.
The  latter is surjective by definition of the filtration $F^\ast$ (see Definition \ref{def:F}).

The above sequence induces a sequence
$$
\lim_{\longrightarrow}\frac{H^{2(i-j)-1}_{i-j-2,nr}(Z,\Z(i-j))}{H^{2(i-j)-1} (Z,\Z(i-j))} \stackrel{\iota_\ast}\longrightarrow \frac{H^{2i-1}_{i-2,nr}(X,\Z(i))}{ H^{2i-1}(X,\Z(i))}\longrightarrow \frac{F^{i-1} H^{2i-1}_{j-1,nr}(X,\Z(i))}{ H^{2i-1} (X,\Z(i))} \longrightarrow 0,
$$
and one directly checks that this sequence remains exact.
By Theorem \ref{thm:main-2-intro}, this sequence identifies to an exact sequence
$$
\lim_{\longrightarrow}\Griff^{i-j}(Z)\stackrel{\iota_\ast}\longrightarrow\Griff^i(X) \longrightarrow \frac{F^{i-1} H^{2i-1}_{j-1,nr}(X,\Z(i))}{ H^{2i-1} (X,\Z(i))} \longrightarrow 0.
$$
It follows from the functoriality of the Gysin sequence with respect to proper pushforwards (see (\ref{ax:Gysin})) that the first map above agrees with the pushforward of cycles induced by $Z\hookrightarrow X$.
This concludes the proof of the corollary.
\end{proof}

\begin{remark}
Theorem \ref{thm:singular} together with Corollary \ref{cor:restr-H_nr} implies analogues of
 Corollary \ref{cor:main-intro} over arbitrary fields. 
 We leave it to the reader to formulate and prove those results.
\end{remark}

\subsection{Applications of Theorem \ref{thm:main-2-intro}} 
If $X$ is an integral scheme over $\C$, we write in this section 
$$
H^i(\C(X),A):=H^i(F_0X,A)=\lim_{\substack{\longrightarrow \\ \emptyset\neq U\subset X}} H^i(U,A) ,
$$
which is consistent with some of the notation used in the literature (see e.g.\ \cite{CTV,Voi-unramified}).
The above group  is the cohomology of the generic point of $X$ as defined in (\ref{eq:Hi(kappa(x))}).
If $A=\Z/\ell^r$ or $A=\Q/\Z$,  this group coincides by \cite[p.\ 88, III.1.16]{milne} with the corresponding Galois cohomology group of the field $\C(X)$.

We will need the following result that is proven with methods from \cite{Sch-JAMS}.

\begin{proposition} \label{prop:example}
For any positive integer $n$, there is a smooth projective unirational variety $Y$ of dimension $3n$ over $\C$ such that the composition
$$
H^{2i}(Y,\Z/2)\longrightarrow H^{2i}(\C(Y),\Z/2)\longrightarrow H^{2i}(\C(Y),\Q/\Z)
$$
is nonzero for all $i=1,\dots ,n$.
\end{proposition}
\begin{proof}
By the proof of \cite[Theorem 1.5]{Sch-JAMS}, there is a unirational smooth complex projective threefold $T$ together with a morphism $f:T\to \CP^2$ whose generic fibre is a conic, such that the following holds:
\begin{itemize}
\item the class $\alpha=(x_1/x_0,x_2/x_0)\in H^2(\C(\CP^2),\Z/2)$ has the property that $f^\ast \alpha\in H^2_{nr}(T,\Z/2)$ is unramified and non-trivial;
\item there is a specialization $T_0$ of $T$ such that the specialization $f_0:T_0\to \CP^2$ of $f$ has the property that its generic fibre has a $\C(\CP^2)$-rational point in its smooth locus.
\end{itemize}

Let us now consider $Y:=T^n$, which is a smooth complex projective variety of dimension $3n$ that is unirational.
Let $\pr_j:Y\to T$ denote the projection onto the $j$-th factor and consider the class 
$$
\gamma_i:=\pr_1^\ast f^\ast\alpha\cup\pr_2^\ast f^\ast \alpha\cup \cdots \cup \pr_i^\ast f^\ast \alpha \in H^{2i}(\C(Y),\Z/2).
$$
Since $\alpha$ is of degree two, the unramified class $f^\ast \alpha$ admits a lift to a class in $H^2(T,\Z/2)$, see Corollary \ref{cor:F^i}.
Hence, $\gamma_i$ admits a lift to a class in $H^{2i}(Y,\Z/2)$ and so 
$$
\gamma_i\in F^iH^{2i}_{nr}(Y,\Z/2).
$$
It remains to show that the image $\gamma'_i$ of $\gamma_i$ in $H^{2i}(\C(Y),\Q/\Z)$ is nonzero for all $i=1,\dots ,n$.
By construction of the class $\gamma_i$, it suffices to prove that $\gamma'_n$ is nonzero and our argument is similar to the proofs of \cite[Proposition 6.1]{Sch-JAMS} and \cite[Theorem 5.3(3)]{Sch-torsion}.

Assume for a contradiction that $\gamma':=\gamma'_n$ is zero in $H^{2n}(\C(Y),\Q/\Z)$.
Let us then specialize $T$ to $T_0$.
Then $Y$ specializes to a projective variety $Y_0$ together with a morphism $Y_0\to (\CP^2)^n$ whose generic fibre admits a rational point in its smooth locus.
The specialization $\gamma'_0$ of $\gamma'$ vanishes, because $\gamma'$ vanishes by assumption.
It follows that the restriction of $\gamma_0'$ to the rational point in the smooth locus of the generic fibre of $Y_0\to (\CP^2)^n$ is zero.
This restriction in turn computes explicitly as the image of
$$
\pr_1^\ast  \alpha\cup\pr_2^\ast  \alpha\cup \cdots \cup \pr_n^\ast  \alpha \in H^{2n}\left( \C\left( (\CP^2)^n\right) ,\Z/2 \right)  
$$ 
in $  H^{2n}(\C((\CP^2)^n),\Q/\Z)$.
But this class is nonzero, as one can check by computing successive residues.
This is a contradiction, which concludes the proof.
\end{proof}

\subsubsection{Integral Hodge conjecture for uniruled varieties} \label{sec:IHC}

Recall that for any algebraic scheme $X$  of dimension $d$ over $\C$, there is a cycle class map $\cl^i_X:\CH^i(X)\to H^{2i}(X,\Z)$, where  $H^i(X,A):=H_{2d-i}^{BM}(X_{\cx},A)$.
We denote its cokernel by $Z^i(X):=\coker(\cl_X^i)$.

\begin{theorem} \label{thm:IHC:example}
For any $n\geq 1$, there is a smooth complex projective unirational variety $Y$ of dimension $3n$ and an elliptic curve $E$ such that $X:=E\times Y$ satisfies
\begin{align} \label{eq:IHC:examples}
\coker \left( 
 \lim_{\longrightarrow}Z^{i-1}(D)_{\tors}\longrightarrow Z^{i}(X)_{\tors} \right) \neq 0 \ \ \text{for all $2\leq i\leq n+1$,}
\end{align}
where $D$ runs through all closed reduced subvarieties $D\subset X$ of codimension $ 1$.
\end{theorem}

Note that for any closed subscheme $Z\subsetneq X$ of codimension $c\geq 1$, the pushforward $Z^{i-c}(Z)\to Z^i(X)$ factors through $Z^{i-1}(D)$ for any divisor $D\subset X$ that contains $Z$, and so the non-trivial class in the cokernel of the above corollary is not hit by $Z^{i-c}(Z)_{\tors}$ and hence in particular not by the torsion in $H^{2i-2c}(Z,\Z)=H^{BM}_{2d_X-2i}(Z,\Z)$, where $d_X=\dim X$. 
In particular, the above theorem implies Corollary \ref{cor:IHC} stated in the introduction.

\begin{proof}[Proof of Theorem \ref{thm:IHC:example}]
By Proposition \ref{prop:example}, there is a unirational smooth complex projective variety $Y$ of dimension $3n$ such that $H^{2i}(Y,\Q/\Z)\longrightarrow H^{2i}(\C(Y),\Q/\Z)$ is nonzero for all $i=1,\dots ,n$. 
It thus follows from a theorem of Colliot-Thélène \cite[Theorem 1.1]{CT-manuscripta} that there is an elliptic curve $E$ such that
the product $X=Y\times E$ has the property that
$$
H^{2i+1}(X,\Q/\Z)\longrightarrow H^{2i+1}(\C(X),\Q/\Z)
$$
is nonzero for all $i=1,\dots ,n$.
Since the Chow group of zero-cycles of $X$ is supported on a curve (i.e.\ $\CH_0(\{pt.\}\times E)\to \CH_0(Y\times E)$ is surjective), the rational unramified cohomology groups of $X$ above degree one vanish by a simple Bloch--Srinivas decomposition of the diagonal argument, see e.g.\ \cite[Proposition 3.3.(i)]{CTV}.\footnote{This step uses the existence of an action of algebraic cycles on unramified cohomology, hence \cite{BO} or \cite{Sch-preparation},  but it does not use the Bloch--Kato conjectures, as  we are only concerned about the vanishing of unramified cohomology with rational coefficients and so torsion-freeness of $H^i_{nr}(X,\Z)$ is not needed.}
The result thus follows from Corollary \ref{cor:main-intro}.
\end{proof}

\subsubsection{Applications to the Artin--Mumford invariant}
In \cite{AM}, Artin and Mumford showed that for any smooth complex projective variety $X$, the torsion subgroup of $H^3(X,\Z)$ is a birational invariant and used this to construct unirational threefolds that are not rational.
For $i>3$, the torsion subgroup of $H^{i}(X,\Z)$ is not a birational invariant.
However, Voisin observed (see \cite[Remark 2.4]{Voi-unramified}) that the Bloch--Kato conjecture proven by Voevodsky implies that the torsion subgroup of $H^{5}(X,\Z)/ N^2H^{5}(X,\Z)$ 
is a birational invariant.
By Proposition \ref{prop:gr_N}, there is a canonical surjection
$$
\varphi:\Tors \left( \frac{H^{i}(X,\Z )}{N^2H^{i}(X,\Z )}\right) \twoheadrightarrow \frac{ G^{\lceil i/2\rceil} H^{i-1}_{nr}(X,\Q/\Z )}{ H^{i-1}_{nr}(X,\Q )} .
$$
(It follows from the Bloch--Kato conjecture, proven by Voevodsky, that this surjection is in fact an isomorphism, see Remark \ref{rem:gr_N}, but we will not use this.)

As an application,  we prove that Voisin's generalization of the Artin--Mumford invariant is non-trivial in any odd degree.  

\begin{corollary}\label{cor:higher-artin-mumford}
For any positive integer $i$, there is a unirational smooth complex projective variety $X$ with a torsion class in $H^{2i+1}(X,\Z)$ that is non-zero in the quotient
$$
 H^{2i+1}(X,\Z)/ N^2H^{2i+1}(X,\Z) .
$$
\end{corollary}
\begin{proof} 
Rationally connected varieties have no rational unramified cohomology in positive degrees, see e.g.\  \cite[Proposition 3.3.(i)]{CTV}.
The claim in Corollary \ref{cor:higher-artin-mumford} follows therefore directly from Propositions \ref{prop:Betti-coho},  \ref{prop:gr_N}, and  \ref{prop:example}.
\end{proof}

\section*{Acknowledgements} 
I am grateful for conversations with and comments from Giuseppe Ancona,  Theodosis Alexandrou,  Samet Balkan, Jean-Louis Colliot-Th\'el\`ene,  Matthias Paulsen,  Anand Sawant,   Domenico Valloni,  Claire Voisin, and Lin Zhou.
Thanks to Hélène Esnault and the referees for  helpful comments on the presentation and to  Bhargav Bhatt  for checking (a previous version of) Section \ref{subsec:constructible-complex}. 
  	This project has received funding from the European Research Council (ERC) under the European Union's Horizon 2020 research and innovation programme under grant agreement No 948066 (ERC-StG RationAlgic).


\begin{thebibliography}{HKLRR} 

\bibitem[SGA4.3]{SGA4.3}
M.\ Artin, A.\ Grothendieck, and J.-L.\ Verdier, {\em Th\'eorie des Topos et Cohomologie Etale des Sch\'emas}, SGA 4, Tome 3, Lecutre Notes in Mathematics 305, Springer, Berlin 1973.

\bibitem[SGA4$\frac{1}{2}$]{SGA4.5}
P.\ Deligne, {\em Cohomologie \'etale}, Lecture Notes in Mathematics \textbf{569} Springer-Verlag, Berlin, 1977, S\'eminaire de G\'eom\'etrie Alg\'ebrique du Bois-Marie SGA 41/2, avec la collaboration de J.\ F.\ Boutot, A.\ Grothendieck, L.\ Illusie et J.-L.\ Verdier.
 

\bibitem[AM72]{AM}
M.\ Artin and D.\ Mumford, {\em Some elementary examples of unirational varieties which are not rational}, Proc.\ London Math.\ Soc.\ (3) \textbf{25} (1972), 75--95.
 

 \bibitem[AH62]{AH}
M.F.\ Atiyah and F.\ Hirzebruch, {\em Analytic cycles on complex manifolds}, Topology  \textbf{1} (1962), 25--45.
 

%
\bibitem[BeOt20a]{Benoist-Ottem-Enriques}
O.\ Benoist and J.C.\ Ottem, {\em Failure of the integral Hodge conjecture for threefolds of Kodaira dimension zero}, Commentarii Mathematici Helvetici \textbf{95} (2020), 27-35.
 

\bibitem[BeWi20]{BW}
O.\ Benoist and O.\ Wittenberg,  {\em Integral Hodge conjecture for real varieties},  Invent.\ Math.\ \textbf{222} (2020), 1--77.

\bibitem[BS15]{BS}
B.\ Bhatt and P.\ Scholze, {\em The pro-\'etale topology of schemes},  Ast\'erisque \textbf{369} (2015),  99--201.
 
\bibitem[Blo79]{bloch-compositio}
S.\ Bloch, {\em Torsion algebraic cycles and a theorem of Roitman}, Compositio Mathematica \textbf{39} (1979), 107--127.

\bibitem[Blo85]{bloch-duke}
S.\ Bloch, {\em Algebraic cycles and values of L-functions II}, Duke Math.\ J.\  \textbf{52} (1985), 379--397.

\bibitem[Blo10]{bloch}
S.\ Bloch, {\em Lectures on algebraic cycles}, 2nd ed., Cambridge 2010.

\bibitem[BO74]{BO}
S.\ Bloch and A.\ Ogus, {\em Gersten's conjecture and the homology of schemes}, Ann.\ Sci.\ \'Ec.\ Norm.\ Sup\'er., \textbf{7} (1974), 181--201.
 

\bibitem[Bre97]{bredon}
G.E.\ Bredon, {\em Sheaf Theory}, Springer, New York, 1997. 


\bibitem[BM60]{BM}
A.\ Borel and J.C.\ Moore,  {\em Homology theory for locally compact spaces}, Michigan Mathematical Journal \textbf{7} 137--159.
 

\bibitem[Cle83]{clemens}
H.\ Clemens,
{\em Homological equivalence modulo algebraic equivalenceis not finitely generated}, Publications mathématiques de l’I.H.É.S. \textbf{58} (1983),  19--38.
%
\bibitem[CT95]{CT}
J.-L.\ Colliot-Th\'el\`ene, {\em Birational invariants, purity and the Gersten conjecture}, K-theory and algebraic geometry: connections with quadratic forms and division algebras (Santa Barbara, CA, 1992), 1--64, Proc.\ Sympos.\ Pure Math.\, \textbf{58}, AMS, Providence, RI, 1995.

\bibitem[CT19]{CT-manuscripta}
J.-L.\ Colliot-Th\'el\`ene, {\em Cohomologie non ramifiée dans le produit avec une courbe elliptique}, Manuscripta Math.\ \textbf{160} (2019), 561--565.

\bibitem[CTSS83]{CTSS}
J.-L.\ Colliot-Th\'el\`ene ,  J.-J.\ Sansuc, and C.\ Soulé, {\em Torsion dans le groupe de Chow de codimension deux}. Duke Math.\ J.\ \textbf{50} (1983), 763--801.
 

\bibitem[CTHK97]{CTHK}
J.-L.\ Colliot-Th\'el\`ene,  R.T.\  Hoobler,  and  B.\  Kahn, {\em The Bloch-Ogus-Gabber theorem},  Algebraic K-theory (Toronto, ON, 1996), \textbf{16}, Fields Inst.\ Commun.\ AMS, Providence, RI, 1997. 

\bibitem[CTO89]{CTO}
J.-L.\ Colliot-Th\'el\`ene and M.\ Ojanguren, {\em Vari\'et\'es unirationnelles non rationnelles : au-del\`a de l'exemple d'Artin et Mumford}, Invent.\ Math.\ \textbf{97} (1989), 141--158.
%

\bibitem[CTV12]{CTV}
J.-L.\ Colliot-Th\'el\`ene and C.\ Voisin, {\em Cohomologie non ramifi\'ee et conjecture de Hodge enti\`ere}, Duke Math.\ J.\ \textbf{161} (2012), 735--801.
 
\bibitem[Del74]{deligne}
P.\ Deligne, {\em La conjecture de Weil, I}, Publ.\ Math.\ I.H.\'E.S., \textbf{43} (1974),  273--307.

 

\bibitem[Eke90]{Eke}
T.\ Ekhedal, {\em On the adic formalism}, Grothendieck Festschrift, Vol.\ II, Progr.\ Math.\ \textbf{87}, Birkh\"auser, 1990, 197--218.

 
 
\bibitem[Ful98]{fulton}
W.\ Fulton, {\em Intersection theory}, Springer--Verlag, 1998. 

\bibitem[Gil10]{gille-inventiones}
S.\ Gille, {\em The Rost nilpotence theorem for geometrically rational surfaces}, Invent.\ Math.\ \textbf{181} (2010), 1--19.

\bibitem[Gil14]{gille-crelle}
S.\ Gille, {\em On Chow motives of surfaces}, J.\ reine angew.\ Math.\ \textbf{686} (2014), 149--166.

\bibitem[Gre98]{green}
M.\ Green,  {\em 
Higher Abel-Jacobi maps},  Proceedings of the International Congress of Mathematicians, Vol. II (Berlin, 1998).
Doc. \ Math.\ 1998, Extra Vol. II,  267--276. 


\bibitem[Gri69]{griffiths}
P.\ Griffiths, {\em On the periods of certain rational integrals I, II}, Ann.\ Math.\ \textbf{90} (1969), 460--541.

\bibitem[Gro68]{grothendieck}
A.\ Grothendieck, {\em Le groupe de Brauer I, II, III}, in: Dix exposés sur la cohomologie des schémas, Advanced Studies in Pure Mathematics vol.\ 3, Masson et North-Holland, 1968. 

\bibitem[ILO14]{ILO}
L.\ Illusie, Y.\ Laszlo,  and F:\ Orgogozo (eds.), {\em Travaux de Gabber sur l’uniformisation locale et lacohomologie étale des schémas quasi-excellents}, Société Mathématique de France, Paris, 2014 (French).  
 With the collaboration of Frédéric Déglise, Alban Moreau, Vincent Pilloni, Michel Raynaud, Joël Riou, BenoîtStroh, Michael Temkin and Weizhe Zheng; Astérisque \textbf{363-364 }(2014).
 


\bibitem[Jan88]{jannsen}
U.\ Jannsen, {\em Continuous \'Etale cohomology}, Math.\ Ann.\ \textbf{280} (1988), 207--245.
 

\bibitem[Jan00]{jannsen-3}
U.\ Jannsen, {\em Equivalence Relations on Algebraic Cycles},  225-260, in: B.B.\ Gordon, J.D.\  Lewis, S.\  M\"uller-Stach, S.\  Saito, N.\  Yui (eds) The Arithmetic and Geometry of Algebraic Cycles. NATO Science Series (Series C: Mathematical and Physical Sciences), vol 548. Springer, Dordrecht, 2000.

\bibitem[Kah12]{Kahn}
 B.\ Kahn, {\em Classes de cycles motiviques étales}, Algebra \& Number Theory \textbf{6}(7) (2012), 1369--1407.
 
 \bibitem[Kat86]{kato}
K.\ Kato, {\em  A Hasse principle for two-dimensional global fields}, J.\  Reine Angew.\ Math.\ \textbf{366}, 142--183, 1986. 
With an appendix by Jean-Louis Colliot-Th\'el\`ene. 

\bibitem[KeSa12]{KeSa}
Moritz Kerz and Shuji Saito, {\em Cohomological Hasse principle and motivic cohomology
for arithmetic schemes},  Publ. \ Math. \ Inst.\  Hautes \'Etudes Sci., \textbf{115} (2012), 123--183.

\bibitem[Kl05]{kleiman}
S.L.\ Kleiman, {\em The Picard scheme}, in: Fundamental Algebraic Geometry (Grothendieck's FGA Explained), American Mathematical Society 2005. 

\bibitem[Lev05]{levine}
M.\ Levine, {\em Mixed Motives},
in: Handbook of K-theory, vol.\ \textbf{1},
E.M.\ Friedlander, D.R.\ Grayson, eds.,
429--522. Springer-Verlag, 2005.

\bibitem[Ma17]{Ma}
S.\ Ma, {\em Torsion 1-cycles and the coniveau spectral sequence}, Documenta Math.\ \textbf{22} (2017), 1501--1517.

\bibitem[Ma20]{ma2}
S.\ Ma, {\em Unramified cohomology, integral coniveau filtration and Griffiths group}, arXiv:2009.14447.

\bibitem[MS83]{MS}
A.\ Merkurjev and A.\ Suslin, {\em K-cohomology of Severi--Brauer varieties and the norm residue homomorphism}, Izv.\ Akad.\ Nauk SSSR Ser.\ Mat.\ \textbf{46} (1982), 1011--1046, 1135--1136. Eng.\ trans., Math.\ USSR Izv.\ \textbf{21} (1983), 307--340.
 

\bibitem[Mil80]{milne}
J.S.\ Milne, {\em \'Etale cohomology}, Princeton University Press, Princeton, NJ, 1980.
 

\bibitem[Nor93]{nori}
M.\ Nori, {\em  Algebraic cycles and Hodge-theoretic connectivity}, Invent.\ Math.\ \textbf{111} (1993), 349--373.

\bibitem[Nér52]{neron}
A. \ N\'eron, {\em Probl\`emes arithm\'etiques et g\'eom\'etriques rattach\'es \`a la notion de rang d’une courbe alg\'ebrique dans un corps}, Bull.\  Soc.\  Math.\  France \textbf{80} (1952),101--166.
 

\bibitem[Per22]{perry}
A.\ Perry, {\em The integral Hodge conjecture for two-dimensional Calabi-Yau categories},  Compositio Math.\ \textbf{158} (2022), 287--333.


\bibitem[Ros96]{rost}
M.\ Rost, {\em Chow groups with coefficients}, Doc.\ Math.\ \textbf{1} (1996), 319--393.

\bibitem[RoSr10]{Rosenschon-Srinivas}
A.\ Rosenschon and V.\ Srinivas, {\em The Griffiths group of the generic abelian 3-fold},  Cycles, motives and Shimura varieties, 449--467. Tata Inst.\ Fund.\ Res., Mumbai, 2010.

\bibitem[RoSa18]{RS}
A.\ Rosenschon and A.\ Sawant, {\em Rost nilpotence and \'etale motivic cohomology}, Adv.\ Math.\ \textbf{330} (2018), 420--432.

 

\bibitem[Sch19]{Sch-JAMS}
S.\ Schreieder, {\em Stably irrational hypersurfaces of small slopes},
J.\ Amer.\ Math.\ Soc.\ \textbf{32} (2019), 1171--1199.

 

\bibitem[Sch21a]{Sch-torsion}
S.\ Schreieder, {\em Torsion orders of Fano hypersurfaces},  Algebra \& Number Theory \textbf{15} (2021), 241--270.

\bibitem[Sch21b]{Sch-Griffiths}
S.\ Schreieder, {\em Infinite torsion in the Griffiths groups}, arXiv:2011.15047, to appear in JEMS.

\bibitem[Sch22]{Sch-preparation}
S.\ Schreieder,  {\em A moving lemma for cohomology with support}, Preprint 2022,  arXiv:2207.08297.
 
\bibitem[Schoe00]{schoen-product}
C.\ Schoen, {\em  On certain exterior product maps of Chow groups}, Math.\ Res.\ Lett.\ \textbf{7} (2000), 177--194.
%
\bibitem[SV05]{SV}
C.\ Soul\'e and C.\ Voisin, {\em Torsion cohomology classes and algebraic cycles on complex manifolds}, Adv.\ Math.\ \textbf{198} (2005),  107--127.
 

\bibitem[Tia20]{tian}
Z.\ Tian, {\em Zero cycles on rationally connected varieties over Laurent fields}, arXiv:2010.04996.
 
\bibitem[Tot97]{totaro-JAMS}
B.\ Totaro, {\em Torsion algebraic cycles and complex cobordism}, J.\ Amer.\ Math.\ Soc.\ \textbf{10} (1997), 467--493. 
 
\bibitem[Tot16]{totaro-annals}
B.\ Totaro, {\em Complex varieties with infinite Chow groups modulo 2}, Annals of Mathematics \textbf{183} (2016), 363--375.

\bibitem[Ver67]{verdier}
J.-L.\ Verdier, {\em A duality theorem in the etale cohomology of schemes}, in: Springer, Tonny Albert (ed.), Proceedings of a Conference on Local Fields: NUFFIC Summer School held at Driebergen (The Netherlands) in 1966, Berlin, New York: Springer-Verlag, 184--198, 1967.
 

\bibitem[Voe11]{Voe:Bloch-Kato}
V.\ Voevodsky, {\em On motivic cohomology with $Z/l$-coefficients}, Ann.\ of Math.\ \textbf{174} (2011), 401--438.

\bibitem[Voi99]{voisin-annals}
C.\ Voisin, {\em Some results on Green's higher Abel-Jacobi map},
Ann.\ of Math.\ \textbf{149} (1999),  451--473. 

\bibitem[Voi00]{voisin-duke}
C.\ Voisin, {\em The Griffiths group of a general Calabi-Yau threefold is not finitely generated}, Duke Math.\ J.\ \textbf{102} (2000),  151--186. 

\bibitem[Voi06]{voisin-IHC}
C.\ Voisin, {\em On integral Hodge classes on uniruled or Calabi-Yau threefolds}, in \emph{Moduli Spaces and Arithmetic Geometry}, Advanced Studies in Pure Mathematics \textbf{45} (2006), 43--73.

\bibitem[Voi12]{Voi-unramified}
C.\ Voisin, {\em Degree $4$  unramified cohomology with finite coefficients and torsion codimension  $3$ cycles}, in Geometry and Arithmetic, (C.\ Faber, G.\ Farkas, R.\ de Jong Eds), Series of Congress Reports, EMS 2012, 347--368.



\end{thebibliography}
\end{document}